\documentclass[a4paper]{article}

\usepackage[english]{babel}
\usepackage[utf8]{inputenc}
\usepackage[T1]{fontenc}
\usepackage{csquotes}
\usepackage{amsmath}
\usepackage{amsthm}
\usepackage{amsfonts}
\usepackage{amssymb}
\usepackage{fancyhdr}
\usepackage[mathscr]{eucal}
\usepackage{pdfpages}
\usepackage[all]{xy}
\usepackage{pgf,tikz}
\usepackage[g]{esvect}
\usepackage{color}
\usepackage[margin=3cm]{geometry}
\usetikzlibrary{arrows}
\usepackage[pdftex,pdfborder={0 0 0},linktoc=all]{hyperref}

\theoremstyle{plain}
	\newtheorem{thm}{Theorem}[section]
	\newtheorem{cor}[thm]{Corollary}
	\newtheorem{lem}[thm]{Lemma}
	\newtheorem{prop}[thm]{Proposition}

\theoremstyle{definition}
	\newtheorem{dfn}[thm]{Definition}
	\newtheorem{ntn}[thm]{Notation}
	\newtheorem{dfns}[thm]{Definitions}
	\newtheorem{ntns}[thm]{Notations}

\theoremstyle{remark}
	\newtheorem{rem}[thm]{Remark}
	\newtheorem{rems}[thm]{Remarks}

\numberwithin{equation}{section}
\numberwithin{figure}{section}

\newcommand{\C}{\mathbb{C}}
\newcommand{\D}{\mathcal{D}}
\newcommand{\E}{\mathcal{E}}
\newcommand{\N}{\mathbb{N}}
\newcommand{\R}{\mathbb{R}}
\newcommand{\X}{\mathcal{X}}

\newcommand{\cov}[2]{\Cov\!\left( #1, #2 \right)}
\newcommand{\dx}{\dmesure\!}
\newcommand{\debar}{\overline{\partial}}
\newcommand{\deron}[2]{\frac{\partial #1}{\partial #2}}
\newcommand{\esp}[2][]{\mathbb{E}_{#1}\!\left[ #2 \right]}
\newcommand{\espcond}[3][]{\mathbb{E}_{#1}\!\left[ #2\hspace{-1mm} \rule{0pt}{4mm}\mvert \! #3 \right]}
\newcommand{\odet}[1]{\norm{\det\! ^\perp\!\left(#1\right)}}
\newcommand{\mvert}{\mathrel{}\middle|\mathrel{}}
\newcommand{\norm}[1]{\left\lvert #1 \right\rvert}
\newcommand{\Norm}[1]{\left\lVert #1 \right\rVert}
\newcommand{\prsc}[2]{\left\langle #1\,, #2 \right\rangle}
\newcommand{\rmes}[1]{\norm{\dmesure\! V_{#1}}}
\newcommand{\var}[1]{\Var\!\left( #1 \right)}
\newcommand{\vol}[1]{\Vol\left(#1\right)}

\renewcommand{\L}{\mathcal{L}}
\renewcommand{\P}{\mathbb{P}}
\renewcommand{\S}{\mathbb{S}}
\renewcommand{\H}{H^0(\X,\E \otimes \L^d)}

\renewcommand{\epsilon}{\varepsilon}
\renewcommand{\geq}{\geqslant}
\renewcommand{\leq}{\leqslant}
\renewcommand{\tilde}{\widetilde}
\renewcommand{\hat}{\widehat}

\DeclareMathOperator{\Cov}{Cov}
\DeclareMathOperator{\dmesure}{d}
\DeclareMathOperator{\End}{End}
\DeclareMathOperator{\ev}{ev}
\DeclareMathOperator{\Id}{Id}
\DeclareMathOperator{\Var}{Var}
\DeclareMathOperator{\Vol}{Vol}

\author{Thomas Letendre\,\thanks{Thomas Letendre, Sorbonne Université, Institut de Mathématiques de Jussieu -- Paris Rive Gauche; e-mail: \url{letendre@math.cnrs.fr}. Thomas Letendre was supported by the French National Research Agency through the ANR grants SpInQS (ANR-17-CE40-0011) and UniRaNDom (ANR-17-CE40-0008).}
\and Martin Puchol\,\thanks{Martin Puchol, Université Paris-Sud, Laboratoire de Mathématiques d'Orsay, F-91405 Orsay, France; e-mail: \url{martin.puchol@math.cnrs.fr}. Martin Puchol was supported by the LABEX MILYON (ANR-10-LABX-0070) of Université de Lyon, within the program ``Investissements d'Avenir'' (ANR-11-IDEX-0007) operated by the French National Research Agency (ANR).}}
\date{\today}
\title{\vspace{-1cm} Variance of the volume of random real algebraic submanifolds~II}

\begin{document}

\maketitle

\vspace{-5mm}

\begin{abstract}
Let $\X$ be a complex projective manifold of dimension $n$ defined over the reals and let $M$ be its real locus. We study the vanishing locus $Z_{s_d}$ in $M$ of a random real holomorphic section $s_d$ of $\E \otimes \L^d$, where $\L \to \X$ is an ample line bundle and $\E\to \X$ is a rank $r$ Hermitian bundle, $r\in \{1,\dots,n\}$. We establish the asymptotics of the variance of the linear statistics associated with $Z_{s_d}$, as $d$ goes to infinity. These asymptotics are of order $d^{r -\frac{n}{2}}$. As a special case, we get the asymptotic variance of the volume of $Z_{s_d}$.

The present paper extends the results of \cite{Let2016a}, by the first-named author, in essentially two ways. First, our main theorem covers the case of maximal codimension ($r=n$), which was left out in \cite{Let2016a}. Second, we show that the leading constant in our asymptotics is positive. This last result is proved by studying the Wiener--It{\=o} expansion of the linear statistics associated with the common zero set in $\R\P^n$ of $r$ independent Kostlan--Shub--Smale polynomials.
\end{abstract}




\section{Introduction}
\label{sec introduction}

In recent years, the study of random submanifolds has been a very active research area~\cite{CH2016,GW2015,GW2016,MPRW2016,NS2016,SW2016}. There exist several models of random submanifolds, built on the following principle. Given $M$ a dimension $n$ ambient manifold and $r \in \{1,\dots,n\}$, we consider the common zero set of~$r$ independent random functions on $M$. Under some technical assumption, this zero set is almost surely a codimension $r$ smooth submanifold.

In this paper, we are interested in a model of random real algebraic submanifolds in a projective manifold. It was introduced in this generality by Gayet and Welschinger in~\cite{GW2011} and studied in \cite{GW2015,GW2016,Let2016,Let2016a}, among others. This model is the real counterpart of the random complex algebraic submanifolds considered by Bleher, Shiffman and Zelditch~\cite{BSZ2000a,SZ1999,SZ2010}.

\paragraph*{Framework.}

Let us describe more precisely our framework. More details are given in Sect.~\ref{sec random real algebraic submanifolds}, below. Let $\X$ be a smooth complex projective manifold of dimension $n\geq 1$. Let $\L$ be an ample holomorphic line bundle over~$\X$ and let $\E$ be a rank $r \in \{1,\dots,n\}$ holomorphic vector bundle over~$\X$. We assume that $\X$, $\E$ and $\L$ are endowed with compatible real structures and that the real locus of $\X$ is not empty. We denote by $M$ this real locus which is a smooth closed (i.e.~compact without boundary) manifold of real dimension~$n$.

Let $h_\E$ and $h_\L$ denote Hermitian metrics on $\E$ and $\L$ respectively, which are compatible with the real structures. We assume that $h_\L$ has positive curvature $\omega$, so that $\omega$ is a Kähler form on $\X$. This $\omega$ induces a Riemannian metric $g$ on $\X$, hence on $M$. Let us denote by $\rmes{M}$ the Riemannian volume measure on $M$ induced by $g$.

For any $d \in \N$, the measure $\rmes{M}$ and the metrics $h_\E$ and $h_\L$ induce a Euclidean inner product on the space $\R \H$ of global real holomorphic sections of $\E \otimes \L^d \to \X$ (see Eq.~\eqref{eq definition inner product}). Given $s \in \R \H$, we denote by $Z_s=s^{-1}(0)\cap M$ the real zero set of $s$. For $d$ large enough, for almost every $s$ with respect to the Lebesgue measure, $Z_s$ is a codimension $r$ smooth closed submanifold of $M$, possibly empty. We denote by $\rmes{s}$ the Riemannian volume measure on $Z_s$ induced by $g$. In the following, we consider $\rmes{s}$ as a Radon measure on $M$, that is a continuous linear form on $(\mathcal{C}^0(M),\Norm{\cdot}_\infty)$, where $\Norm{\cdot}_\infty$ denotes the sup norm.

\begin{rem}
\label{rem codim max}
If $n=r$ then $Z_s$ is a finite subset of $M$ for almost every $s$. In this case, $\rmes{s}$ is the sum of the unit Dirac masses on the points of $Z_s$.
\end{rem}

Let $s_d$ be a standard Gaussian vector in $\R \H$. Then $\rmes{s_d}$ is a random positive Radon measure on $M$. We set $Z_{d}=Z_{s_d}$ and $\rmes{d}=\rmes{s_d}$ in order to simplify notations. We are interested in the asymptotic distribution of the linear statistics $\prsc{\rmes{d}}{\phi} = \int_{Z_d} \phi \rmes{d}$, where $\phi : M \to \R$ is a continuous test-function. In particular, $\prsc{\rmes{d}}{\mathbf{1}}$ is the volume of $Z_d$ (its cardinal if $n=r$), where $\mathbf{1}$ is the unit constant function on $M$.

As usual, we denote by $\esp{X}$ the mathematical expectation of the random vector $X$. The asymptotic expectation of $\prsc{\rmes{d}}{\phi}$ was computed in~\cite[Sect.~5.3]{Let2016}.

\begin{thm}[\cite{Let2016}]
\label{thm expectation}
Let $\X$ be a complex projective manifold of positive dimension $n$ defined over the reals, we assume that its real locus $M$ is non-empty. Let $\E \to \X$ be a rank $r \in \{1,\dots,n\}$ Hermitian vector bundle and let $\L \to \X$ be a positive Hermitian line bundle, both equipped with compatible real structures. For every $d \in \N$, let $s_d$ be a standard Gaussian vector in $\R \H$. Then the following holds as $d \to +\infty$:
\begin{equation*}
\label{eq expectation}
\forall \phi \in \mathcal{C}^0(M), \qquad \esp{\prsc{\rmes{d}}{\phi}} = d^\frac{r}{2} \left(\int_M \phi \rmes{M}\right) \frac{\vol{\S^{n-r}}}{\vol{\S^n}} + \Norm{\phi}_\infty O\!\left(d^{\frac{r}{2}-1}\right).
\end{equation*}
Moreover the error term $O\!\left(d^{\frac{r}{2}-1}\right)$ does not depend on $\phi$.
\end{thm}

The asymptotic variance of $\prsc{\rmes{d}}{\phi}$, as $d$ goes to infinity, was proved to be a $O\!\left(d^{r-\frac{n}{2}}\right)$ when the codimension of $Z_d$ is $r <n$ (see~\cite[Thm.~1.6]{Let2016a}). Our first main theorem (Thm.~\ref{thm asymptotics variance} below) extends this result to  the maximal codimension case.

\paragraph*{Statement of the main results.}

Before we state our main result, let us introduce some more notations. We denote by $\cov{X}{Y} = \esp{\left(X-\esp{X}\right)\left(Y-\esp{Y}\right)}$ the covariance of the real random variables $X$ and $Y$. Let $\var{X}=\cov{X}{X}$ denote the variance of $X$. Finally, we call \emph{variance} of $\rmes{d}$ and we denote by $\var{\rmes{d}}$ the symmetric bilinear form on $\mathcal{C}^0(M)$ defined by:
\begin{equation*}
\label{def cov dVd}
\forall \phi_1,\phi_2 \in \mathcal{C}^0(M), \qquad \var{\rmes{d}}\left(\phi_1,\phi_2\right) = \cov{\prsc{\rmes{d}}{\phi_1}}{\prsc{\rmes{d}}{\phi_2}}.
\end{equation*}

\begin{dfn}
\label{def continuity modulus}
Let $\phi \in \mathcal{C}^0(M)$, we denote by $\varpi_\phi$ its \emph{continuity modulus}, which is the function from $(0,+\infty)$ to $[0,+\infty)$ defined by:
\begin{equation*}
\varpi_\phi: \epsilon \longmapsto \sup \left\{ \norm{\phi(x)-\phi(y)} \mvert (x,y) \in M^2, \rho_g(x,y)\leq \epsilon \right\},
\end{equation*}
where $\rho_g(\cdot,\cdot)$ stands for the geodesic distance on $(M,g)$.
\end{dfn}

We denote by $\mathcal{M}_{rn}(\R)$ the space of matrices of size $r \times n$ with real coefficients.

\begin{dfn}
\label{def Jacobian}
Let $L : V \to V'$ be a linear map between two Euclidean spaces. We denote the \emph{Jacobian} of $L$ by $\odet{L}= \sqrt{\det\left(LL^*\right)}$, where $L^*:V' \to V$ is the adjoint operator of $L$. Similarly, let $A \in \mathcal{M}_{rn}(\R)$, we define its \emph{Jacobian} to be $\odet{A} = \sqrt{\det\left(A A^\text{t}\right)}$.
\end{dfn}

\begin{dfn}
\label{def XYt}
For every $t>0$, we define $\left(X(t),Y(t)\right)$ to be a centered Gaussian vector in $\mathcal{M}_{rn}(\R)\times \mathcal{M}_{rn}(\R)$ such that the following hold:
\begin{itemize}
\item the couples $\left(X_{ij}(t),Y_{ij}(t)\right)$ with $i,j \in \{1,\dots,n\}$ are independent from one another,
\item the variance matrix of $\left(X_{ij}(t),Y_{ij}(t)\right)$ is:
\begin{align*}
&\begin{pmatrix}
1-\frac{te^{-t}}{1-e^{-t}}& e^{-\frac{t}{2}}\left(1-\frac{t}{1-e^{-t}}\right)\\
e^{-\frac{t}{2}}\left(1-\frac{t}{1-e^{-t}}\right) & 1-\frac{te^{-t}}{1-e^{-t}}
\end{pmatrix} & &\text{if } j=1, \text{ and} & &\begin{pmatrix}
1& e^{-\frac{t}{2}}\\ e^{-\frac{t}{2}} & 1
\end{pmatrix} & &\text{otherwise.}
\end{align*}
\end{itemize}
\end{dfn}

We can now state our main result.
\begin{thm}
\label{thm asymptotics variance}
Let $\X$ be a complex projective manifold of dimension $n\geq 1$ defined over the reals, we assume that its real locus $M$ is non-empty. Let $\E \to \X$ be a rank $r \in \{1,\dots,n\}$ Hermitian vector bundle and let $\L \to \X$ be a positive Hermitian line bundle, both equipped with compatible real structures. For every $d \in \N$, let $s_d$ be a standard Gaussian vector in $\R \H$.

Let $\beta \in (0,\frac{1}{2})$, then there exists $C_\beta >0$ such that, for all $\alpha \in \left(0,1\right)$, for all $\phi_1$ and $\phi_2 \in \mathcal{C}^0(M)$, the following holds as $d \to +\infty$:
\begin{multline}
\label{eq asymptotics variance}
\var{\rmes{d}}\left(\phi_1,\phi_2\right) = d^{r-\frac{n}{2}} \left(\int_M \phi_1 \phi_2 \rmes{M}\right) \left(\frac{\vol{\S^{n-1}}}{(2\pi)^r} \mathcal{I}_{n,r} + \delta_{rn}\frac{2}{\vol{\S^n}}\right)\\
+ \Norm{\phi_1}_{\infty}\Norm{\phi_2}_{\infty}  O\!\left(d^{r-\frac{n}{2}-\alpha}\right) + \Norm{\phi_1}_{\infty}\varpi_{\phi_2}\left(C_\beta d^{-\beta}\right) O\!\left(d^{r-\frac{n}{2}}\right),
\end{multline}
where $\delta_{rn}$ is the Kronecker symbol, equal to $1$ if $r=n$ and $0$ otherwise, and
\begin{equation}
\label{eq def Inr}
\mathcal{I}_{n,r} = \frac{1}{2} \int_0^{+\infty} \left(\frac{\esp{\odet{X(t)}\odet{Y(t)}}}{\left(1-e^{-t}\right)^\frac{r}{2}} -(2\pi)^r\left(\frac{\vol{\S^{n-r}}}{\vol{\S^n}}\right)^2\right)t^{\frac{n-2}{2}}\dx t< + \infty.
\end{equation}
Moreover the error terms $O\!\left(d^{r-\frac{n}{2}-\alpha}\right)$ and $O\!\left(d^{r-\frac{n}{2}}\right)$ in~\eqref{eq asymptotics variance} do not depend on $\left(\phi_1,\phi_2\right)$.
\end{thm}

\begin{rem}
Applying Thm.~\ref{thm asymptotics variance} with $\phi_1=\mathbf{1}=\phi_2$ gives the asymptotic variance of the Riemannian volume of $Z_d$.
\end{rem}

\begin{thm}
\label{thm positivity}
For any $n \in \N^*$ and $r \in \{1,\dots,n\}$, the universal constant:
\begin{equation*}
\frac{\vol{\S^{n-1}}}{(2\pi)^r} \mathcal{I}_{n,r} + \delta_{rn}\frac{2}{\vol{\S^n}}
\end{equation*}
appearing in Thm.~\ref{thm asymptotics variance} is positive.
\end{thm}

\begin{rem}
Thm.~\ref{thm positivity} was proved for $r=n=1$ in~\cite{Dal2015}, and for $r=n \geq 2$ in~\cite{AADL2017}. Note that Thm.~\ref{thm positivity} states that $\mathcal{I}_{n,r} >0$ if $r<n$, but this is not necessarily the case when $r=n$. Indeed, $\mathcal{I}_{1,1} <0$ by \cite[Prop. 3.1 and Rem. 1]{Dal2015}.
\end{rem}

Let us state some corollaries of Thm.~\ref{thm asymptotics variance}. Cor.~\ref{cor concentration}, \ref{cor connected components} and~\ref{cor as convergence} below are extensions to the case $r \leq n$ of Cor.~1.9, 1.10 and~1.11 of~\cite{Let2016a}, respectively. The proofs that Thm.~\ref{thm asymptotics variance} implies Cor.~\ref{cor concentration}, \ref{cor connected components} and~\ref{cor as convergence} were given in \cite[Sect.~5]{Let2016a} in the case $r<n$. They are still valid for $r\leq n$. We do not reproduce these proofs in the present paper.

\begin{cor}[Concentration in probability]
\label{cor concentration}
In the same setting as Thm.~\ref{thm asymptotics variance}, let $\alpha > - \frac{n}{4}$ and let $\phi \in \mathcal{C}^0(M)$. Then, for every $\epsilon >0$, we have:
\begin{equation*}
\P\left( d^{-\frac{r}{2}}\norm{\prsc{\rmes{d}}{\phi}-\rule{0pt}{4mm}\esp{\prsc{\rmes{d}}{\phi}}} > d^{\alpha} \epsilon \right)  = \frac{1}{\epsilon^2} O\!\left( d^{-(\frac{n}{2}+2\alpha)} \right),
\end{equation*}
where the error term is independent of $\epsilon$, but depends on $\phi$.
\end{cor}

\begin{cor}
\label{cor connected components}
In the same setting as Thm.~\ref{thm asymptotics variance}, let $U \subset M$ be an open subset, then as $d\to + \infty$ we have $\P\left( Z_d \cap U = \emptyset \right) = O\!\left( d^{-\frac{n}{2}} \right)$.
\end{cor}

Let us denote by $\dx \nu_d$ the standard Gaussian measure on $\R\H$ (see~\eqref{eq Gaussian density}). Let $\dx \nu$ denote the product measure $\bigotimes_{d \in \N} \dx \nu_d$ on $\prod_{d \in \N} \R \H$. Then we have the following.

\begin{cor}[Strong law of large numbers]
\label{cor as convergence}
In the setting of Thm.~\ref{thm asymptotics variance}, let us assume $n\geq 3$. Let $(s_d)_{d \in \N} \in \displaystyle\prod_{d \in \N} \R \H$ be a random sequence of sections. Then, $\dx \nu$-almost surely,
\begin{equation*}
d^{-\frac{r}{2}}\rmes{s_d} \xrightarrow[d \to +\infty]{} \frac{\vol{\S^{n-r}}}{\vol{\S^n}}\rmes{M},
\end{equation*}
in the sense of the weak convergence of measures. That is, $\dx \nu$-almost surely,
\begin{equation*}
\forall \phi \in \mathcal{C}^0(M), \qquad d^{-\frac{r}{2}}\prsc{\rmes{s_d}}{\phi} \xrightarrow[d \to +\infty]{} \frac{\vol{\S^{n-r}}}{\vol{\S^n}}\left(\int_M \phi \rmes{M}\right).
\end{equation*}
\end{cor}

\paragraph*{Related works and novelty of the main results.}

This paper extends the results of~\cite{Let2016a}, by the first-named author. In~\cite[Thm.~1.6]{Let2016a}, our main result (Thm.~\ref{thm asymptotics variance} above) was proved for $r<n$ and $\alpha \in (0,\alpha_0)$, where $\alpha_0 \in (0,1)$ is some explicit constant depending on $n$ and $r$. The main novelty in Thm.~\ref{thm asymptotics variance} is that it covers the case of maximal codimension ($r=n$), that is the case where $Z_d$ is almost surely a finite subset of $M$. This case was not considered in~\cite{Let2016a} because of additional singularities arising in the course of the proof, which caused it to fail when $r=n$.

An important contribution of the present paper is that we prove new estimates (see Lem.~\ref{lem asymptotic Theta d}, \ref{lem asymptotic Omega d} and \ref{lem asymptotic Lambda d}) for operators related to the Bergman kernel of $\E \otimes \L^d$, which is the correlation kernel of the random field $(s_d(x))_{x \in M}$. These estimates are one of the key improvements that allow us to prove Thm.~\ref{thm asymptotics variance} in the case $r=n$. They also allow us to consider $\alpha\in (0,1)$ instead of $\alpha \in (0,\alpha_0)$. Finally, the use of these estimates greatly clarifies the proof of Thm.~\ref{thm asymptotics variance} in the case $r<n$, compared to the proof given in~\cite{Let2016a}. For this reason, we give the proof of Thm.~\ref{thm asymptotics variance} in the general case $r \in \{1,\dots,n\}$ and not only for $r=n$. This does not lengthen the proof.

The second main contribution of this article is the proof of the positivity of the leading constant in Thm.~\ref{thm asymptotics variance} (cf.~Thm.~\ref{thm positivity}). This result did not appear in~\cite{Let2016a}. Since the leading constant in Thm.~\ref{thm asymptotics variance} is universal, when $r=n$ one can deduce Thm.~\ref{thm positivity} from results of Dalmao \cite{Dal2015} (if $r=n=1$) and Armentano--Azaïs--Dalmao--Le\'{o}n \cite{AADL2017} (if $r=n\geq 2$). In~\cite{AADL2017,Dal2015}, the authors proved Thm.~\ref{thm asymptotics variance} in the special case where $Z_d$ is the zero set in $\R\P^n$ of $n$ independent Kostlan--Shub--Smale polynomials (see Sect.~\ref{subsec KSS polynomials} below). Their results include the positivity of the leading constant, hence implies Thm.\ref{thm positivity} in this case. To the best of our knowledge, Thm.~\ref{thm positivity} is completely new for $r<n$.

Note that when $n=r=1$, our setting covers the case of the binomial polynomials on $\C$ with standard Gaussian coefficients. Much more is known in this case, including variance estimates for the number of real zeros of non-Gaussian ensembles of real polynomials (see \cite{TV2015}).

Our proof of Thm.~\ref{thm positivity} uses the Wiener--It{\=o} expansion of the linear statistics associated with the field $(s_d(x))_{x \in M}$. This kind of expansion has been studied by Slud~\cite{Slu1991} and Kratz--Le\'{o}n~\cite{KL1997,KL2001}. It was used in a random geometry context in~\cite{AADL2017,Dal2015,DNPR2016,MPRW2016}. In~\cite{DNPR2016,MPRW2016}, the authors used these Wiener chaos techniques to prove Central Limit Theorems for the volume of the zero set of Arithmetic Random Waves on the two-dimensional flat torus (see also~\cite{Dal2015} in an algebraic setting). In~\cite{AADL2017,Dal2015}, these methods where used to prove Thm.~\ref{thm positivity} when $r=n$.

In the related setting of Riemannian Random Waves, Canzani and Hanin \cite{CH2016a} obtained recently an asymptotic upper bound for the variance of the linear statistics. To the best of our knowledge, in this Riemannian setting, the precise asymptotics of the variance of the volume of random submanifolds are known only when the ambient manifold is $\S^2$ (cf.~\cite{Wig2010}) or $\mathbb{T}^2$ (cf.~\cite{DNPR2016,KKW2013}). We refer to the introduction of~\cite{Let2016a} for more details about related works.

\paragraph*{About the proofs.}

The proof of Thm.~\ref{thm asymptotics variance} broadly follows the lines of the proof of~\cite[Thm.~1.6]{Let2016a}. The random section $s_d$ defines a centered Gaussian field $(s_d(x))_{x \in M}$, whose correlation kernel is $E_d$, the Bergman kernel of $\E \otimes \L^d$ (see~Sect.~\ref{subsec correlation kernel}). Thanks to results of Dai--Liu--Ma~\cite{DLM2006} and Ma--Marinescu~\cite{MM2015}, we know that this kernel decreases exponentially fast outside of the diagonal $\Delta = \{(x,y) \in M^2 \mid x=y \}$ and that it admits a universal local scaling limit close to~$\Delta$ (see Sect.~\ref{sec estimates for the bergman kernel} for details).

By an application of Kac--Rice formulas (cf.~Thm.~\ref{thm Kac-Rice 1} and~\ref{thm Kac-Rice 2}), we can express the covariance of $\prsc{\rmes{d}}{\phi_1}$ and $\prsc{\rmes{d}}{\phi_2}$ as a double integral over $M \times M$ of $\phi_1(x)\phi_2(y)$ times a density function $\D_d(x,y)$ that depends only on $E_d$. Our main concern is to understand the asymptotics of the integral of $\D_d(x,y)$, as $d \to +\infty$.

Thanks to the exponential decay of the Bergman kernel, we can show that the leading term in our asymptotics is given by the integral of $\D_d$ over a neighborhood $\Delta_d$ of $\Delta$, of typical size $\frac{1}{\sqrt{d}}$ (see~Prop.~\ref{prop off diagonal is small}). Changing variables so that we integrate on a domain of typical size independent of $d$ leads to the apparition of a factor $d^{-\frac{n}{2}}$. Besides, $\D_d$ takes values of order $d^r$ on $\Delta_d$ (see Prop.~\ref{prop asymptotic Ddxz}). This explains why the asymptotic variance is of order $d^{r -\frac{n}{2}}$ in Thm.~\ref{thm asymptotics variance}.

The behavior of $E_d$ allows to prove that $\D_d$ admits a universal local scaling limit on $\Delta_d$. The main difficulty in our proof is to show that the convergence to this scaling limit is uniform on $\Delta_d$ (see~Prop.~\ref{prop asymptotic Ddxz} for a precise statement). This difficulty comes from the fact that $\D_d$ is singular along $\Delta$, just like almost everything in this problem. This is where our proof differs from~\cite{Let2016a}. In~\cite{Let2016a}, the uniform convergence of $\D_d$ to its scaling limit on $\Delta_d$ is not established, and one has to work around this lack of uniformity. This yields a complicated proof that fails when $r=n$. Here, we manage to prove this uniform convergence, thanks to some new key estimates (see Lem.~\ref{lem asymptotic Theta d}, \ref{lem asymptotic Omega d} and \ref{lem asymptotic Lambda d}) that form the technical core of the paper. This allows us to both improve on the results of~\cite{Let2016a} and simplify their proof.

As we explained, our proof relies on two properties of the Bergman kernel $E_d$: namely, the existence of a scaling limit around any point at scale $\frac{1}{\sqrt{d}}$, and its exponential decay outside of the diagonal. These features are also exhibited by Bergman kernels in other settings such as those of~\cite{Bay2018a} or~\cite{Ber2009}, so one might hope to generalize our results to these settings, at least in the bulk. Unfortunately, we also need a precise understanding of the scaling limit of $E_d$, which is possible in our framework because it is universal (it only depends on $n$) and invariant under isometries (see Sect.~\ref{sec properties of the limit distribution} for more details). As far as we know, it is much more complicated to study this scaling limit in other settings (such as those of~\cite{Bay2018a} and~\cite{Ber2009}), so we do not pursue this line of inquiry in the present paper and leave it for future research.

Let us now discuss the proof of Thm.~\ref{thm positivity}. One would expect to prove this by computing a good lower bound for $\mathcal{I}_{n,r}$, directly from its expression (see Eq.~\eqref{eq def Inr}). To the best of our knowledge this approach fails, and we have to use subtler techniques.

Since the leading constant in~\eqref{eq asymptotics variance} only depends on $n$ and $r$, we can focus on the case of the volume of $Z_d$ (where $\phi_1=\mathbf{1}=\phi_2$) in a particular geometric setting. We consider the common real zero set of $r$ independent Kostlan--Shub--Smale polynomials in $\R \P^n$ (see Sect.~\ref{subsec KSS polynomials} for details). This allows for explicit computations since the Bergman kernel is explicitly known in this setting. Moreover, the distribution of these polynomials is invariant under the action of $O_{n+1}(\R)$ on $\R\P^n$, which leads to useful independence proprieties that are not satisfied in general.

In this framework, we adapt the strategy of~\cite{AADL2017,Dal2015} to the case $r<n$. First, we compute the Wiener--It{\=o} expansion of the volume of $Z_d$. That is, we expand $\vol{Z_d}$ as $\sum_{q \in \N} \vol{Z_d}[q]$, where the convergence is in the space of $L^2$ random variables on our probability space, and $\vol{Z_d}[q]$ denotes the $q$-th chaotic component of $\vol{Z_d}$. In particular, $\vol{Z_d}[0]$ is the expectation of $\vol{Z_d}$ and the $(\vol{Z_d}[q])_{q \in \N}$ are pairwise orthogonal $L^2$ random variables. Hence,
\begin{equation*}
\var{\vol{Z_d}} = \sum_{q \geq 1} \var{\vol{Z_d}[q]}.
\end{equation*}
The chaotic components of odd order of $\vol{Z_d}$ are zero, but we prove that $\var{\vol{Z_d}[2]}$ is equivalent to $d^{r-\frac{n}{2}}C$ as $d \to +\infty$, where $C>0$ (see~Lem.~\ref{lem second chaos is positive}). This implies Thm.~\ref{thm positivity}.

\paragraph*{Outline of the paper.}

In Sect.~\ref{sec random real algebraic submanifolds} we describe precisely our framework and the construction of the random measures $\rmes{d}$. We also introduce the Bergman kernel of $\E \otimes \L^d$ and prove that it is the correlation kernel of $(s_d(x))_{x \in M}$. In Sect.~\ref{sec estimates for the bergman kernel}, we recall estimates for the Bergman kernel, and its scaling limit. Sect.~\ref{sec properties of the limit distribution} is dedicated to the study of the Bargmann--Fock process, that is the Gaussian centered random process on $\R^n$ whose correlation function is:
\begin{equation*}
(w,z) \longmapsto \exp\left(-\frac{1}{2}\Norm{w-z}^2\right).
\end{equation*}
This field is the local scaling limit of the random field $(s_d(x))_{x \in M}$, in a sense to be made precise below. Sect.~\ref{sec proof of the main theorem} and~\ref{sec proof of the positivity} are concerned with the proofs of Thm.~\ref{thm asymptotics variance} and Thm.~\ref{thm positivity} respectively. Note that in Sect.~\ref{sec proof of the positivity} we have to study in details the model of Kostlan--Shub--Smale polynomials, which is the simplest example of our general real algebraic setting. We conclude this paper by two appendices, App.~\ref{sec technical computations for Section limit distrib} and App.~\ref{sec technical computations for Section main proof}, in which we gathered the proofs of the technical lemmas of Sect.~\ref{sec properties of the limit distribution} and Sect.~\ref{sec proof of the main theorem} respectively.

\paragraph*{Acknowledgments.}

We wish to thank Federico Dalmao and Maurizia Rossi for interesting discussions about Wiener--It{\=o} expansions, and Xiaonan Ma for helpful discussions about Bergman kernels and for having encouraged our collaboration. Thomas Letendre also thanks Damien Gayet for suggesting this problem in the first place, and for his support.


\tableofcontents


\section{Random real algebraic submanifolds}
\label{sec random real algebraic submanifolds}

In this section, we introduce the main objects we will be studying throughout this paper. We first recall some basic definitions in Sect.~\ref{subsec random vectors}. In Sect.~\ref{subsec general setting}, we introduce our geometric framework. In Sect.~\ref{subsec random real algebraic submanifolds}, we describe our model of random real algebraic submanifolds. Finally, we relate these random submanifolds to Bergman kernels in Sect.~\ref{subsec correlation kernel}.


\subsection{Random vectors}
\label{subsec random vectors}

Let us recall some facts about random vectors (see for example \cite[appendix~A]{Let2016}). In this paper, we only consider centered random vectors, so we give the following definitions in this setting.

Let $X_1$ and $X_2$ be centered random vectors taking values in Euclidean (or Hermitian) vector spaces $V_1$ and $V_2$ respectively, then we define their \emph{covariance operator} as:
\begin{equation*}
\label{eq def covariance operator}
\cov{X_1}{X_2} : v \longmapsto \esp{X_1 \prsc{v}{X_2}}
\end{equation*}
from $V_2$ to $V_1$. For all $v \in V_2$, we set $v^*=\prsc{\cdot}{v} \in V_2^*$. Then $\cov{X_1}{X_2} = \esp{X_1 \otimes X_2^*}$ is an element of $V_1 \otimes V_2^*$. Let $X$ be a centered random vector in a Euclidean space $V$. The \emph{variance operator} of $X$ is defined as $\var{X} = \cov{X}{X} = \esp{X \otimes X^*} \in V \otimes V^*$.  Let $\Lambda$ be a non-negative self-adjoint operator on $(V,\prsc{\cdot}{\cdot})$, we denote by $X \sim \mathcal{N}(\Lambda)$ the fact that $X$ is a centered Gaussian vector with variance operator $\Lambda$. This means that the characteristic function of $X$ is $\xi \mapsto \exp\left(-\frac{1}{2}\prsc{\Lambda \xi}{\xi}\right)$. Finally, we say that $X \in V$ is a \emph{standard} Gaussian vector if $X \sim \mathcal{N}(\Id)$, where $\Id$ is the identity operator on $V$.

If $\Lambda$ is positive, the distribution of $X \sim \mathcal{N}(\Lambda)$ admits the density:
\begin{equation}
\label{eq Gaussian density}
x \mapsto \frac{1}{\sqrt{2\pi}^{N}\sqrt{\det(\Lambda)}} \exp\left(-\frac{1}{2}\prsc{\Lambda^{-1}x}{x}\right)
\end{equation}
with respect to the normalized Lebesgue measure of $V$, where $N = \dim(V)$. Otherwise, $X$ takes values in $\ker(\Lambda)^\perp$ almost surely, and it admits a similar density as a variable in $\ker(\Lambda)^\perp$.


\subsection{General setting}
\label{subsec general setting}

Let us introduce more precisely our geometric framework. Let $\X$ be a smooth complex projective manifold of positive complex dimension $n$. Let $c_\X$ be a real structure on $\X$, i.e.~an anti-holomorphic involution. The real locus of $(\X,c_\X)$ is the set $M$ of fixed points of $c_\X$. In the following, we assume that $M$ is non-empty. It is known that $M$ is a smooth closed submanifold of $\X$ of real dimension~$n$ (see \cite[chap.~1]{Sil1989}).

Let $\E \to \X$ be a holomorphic vector bundle of rank $r \in \{1,\dots,n\}$, we denote by $\pi_\E$ its bundle projection. Let $c_\E$ be a real structure on $\E$, compatible with $c_\X$ in the sense that $c_\X \circ \pi_\E = \pi_\E \circ c_\E$ and $c_\E$ is fiberwise $\C$-anti-linear. Let $h_\E$ be a Hermitian metric on $\E$ such that $c_\E^\star(h_\E)=\overline{h_\E}$. A Hermitian metric satisfying this condition is said to be \emph{real}. Similarly, let $\L \to \X$ be an ample holomorphic line bundle equipped with a compatible real structure $c_\L$ and a real Hermitian metric~$h_\L$.

We assume that $(\L,h_\L)$ has positive curvature, that is its curvature form $\omega$ is Kähler. Recall that, if $\zeta_0$ is a local non-vanishing holomorphic section of $\L$, then $\omega = \frac{1}{2i}\partial\debar \ln\left(h_\L(\zeta_0,\zeta_0)\right)$ locally. This Kähler form defines a Riemannian metric $g$ on $\X$ (see~\cite[Sect.~0.2]{GH1994} for example). In turn, $g$ induces a Riemannian volume measure on $\X$ and on any smooth submanifold of $\X$. We denote by $\dx V_\X=\frac{\omega^n}{n!}$ the Riemannian volume form on $(\X,g)$. Similarly, let $\rmes{M}$ denote the Riemannian measure on~$(M,g)$.

Let $d \in \N$, we endow $\E \otimes \L^d$ with the real structure $c_d = c_\E \otimes (c_\L)^d$, which is compatible with $c_\X$, and the real Hermitian metric $h_d = h_\E \otimes h_\L^d$. Let $\Gamma(\E \otimes \L^d)$ denote the space of smooth sections of $\E \otimes \L^d$, we define a Hermitian inner product on $\Gamma(\E \otimes \L^d)$ by:
\begin{equation}
\label{eq definition inner product}
\forall s_1, s_2 \in \Gamma(\E \otimes \L^d), \qquad \prsc{s_1}{s_2} = \int_\X h_d(s_1(x),s_2(x)) \dx V_\X.
\end{equation}

\begin{rem}
\label{rem bracket}
In this paper, $\prsc{\cdot}{\cdot}$ will either denote the inner product of a Euclidean (or Hermitian) space or the duality pairing between a Banach space and its topological dual. Which one should be clear from the context.
\end{rem}

We say that a section $s \in \Gamma(\E \otimes \L^d)$ is \emph{real} if it is equivariant for the real structures, that is: $c_d \circ s = s \circ c_\X$. We denote by $\R \Gamma(\E \otimes \L^d)$ the real vector space of real smooth sections of $\E \otimes \L^d$. The restriction of $\prsc{\cdot}{\cdot}$ to $\R \Gamma(\E \otimes \L^d)$ is a Euclidean inner product. Note that, despite their name, real sections are defined on the whole complex locus $\X$ and that the Euclidean inner product is defined by integrating over $\X$, not just $M$.

Let $x \in M$, then the fiber $(\E \otimes \L^d)_x$ is a dimension~$r$ complex vector space and the restriction of $c_d$ to this space is a $\C$-anti-linear involution. We denote by $\R(\E \otimes \L^d)_x$ the set of fixed points of this involution, which is a real $r$-dimensional vector space. Then, $\R(\E \otimes \L^d) \to M$ is a rank $r$ real vector bundle and, for any $s \in \R \Gamma(\E \otimes \L^d)$, the restriction of $s$ to $M$ is a smooth section of $\R(\E \otimes \L^d) \to M$.

Let $\H$ denote the space of global holomorphic sections of $\E \otimes \L^d$. This space is known to be finite-dimensional (compare \cite[Thm.~1.4.1]{MM2007}). Let $N_d$ denote the complex dimension of $\H$. We denote by:
\begin{equation*}
\R \H = \left\{ s \in \H \mvert c_d \circ s = s \circ c_\X \right\}
\end{equation*}
the space of global real holomorphic sections of $\E \otimes\L^d$. The restriction of the inner product \eqref{eq definition inner product} to $\R \H$ (resp.~$\H$) makes it into a Euclidean (resp.~Hermitian) space of real (resp.~complex) dimension $N_d$.


\subsection{Random real algebraic submanifolds}
\label{subsec random real algebraic submanifolds}

This section is concerned with the definition of the random submanifolds we consider and some related random variables.

Let $d \in \N$ and $s \in \R \H$, we denote by $Z_s$ the real zero set $s^{-1}(0) \cap M$ of $s$. If the restriction of $s$ to $M$ vanishes transversally, then $Z_s$ is a smooth closed submanifold of codimension~$r$ of $M$ (note that this includes the case where $Z_s$ is empty). In this case, we denote by $\rmes{s}$ the Riemannian volume measure on $Z_s$ induced by $g$. In the following, we consider $\rmes{s}$ as the continuous linear form on $(\mathcal{C}^0(M),\Norm{\cdot}_\infty)$ defined by:
\begin{equation*}
\label{eq def linear stats}
\forall \phi \in \mathcal{C}^0(M), \qquad \prsc{\rmes{s}}{\phi} = \int_{x \in Z_s} \phi(x) \rmes{s}.
\end{equation*}

\begin{dfn}[compare \cite{Nic2015}]
\label{def 0 ample}
We say that $\R \H$ is \emph{$0$-ample} if, for any $x \in M$, the evaluation map $\ev_x^d:s \mapsto s(x)$ from $\R\H$ to $\R \left(\E \otimes \L^d\right)_x$ is surjective.
\end{dfn}

\begin{lem}
\label{lem def d1}
There exists $d_1 \in \N$, depending only on $\X$, $\E$ and $\L$, such that for all $d \geq d_1$, $\R\H$ is $0$-ample.
\end{lem}

\begin{proof}
This can be deduced from the Riemann--Roch formula, for example. It is also a by-product of the computations of the present paper and will be proved later on, see Cor.~\ref{cor def d1} below.
\end{proof}

Let us now consider a random section in $\R\H$. Recall that $\R\H$, endowed with the inner product \eqref{eq definition inner product}, is a Euclidean inner product of dimension $N_d$. Let $s_d$ be a standard Gaussian vector in $\R \H$.

\begin{lem}
\label{lem as submanifold}
For every $d \geq d_1$, $Z_{s_d}$ is almost surely a smooth closed codimension~$r$ submanifold of~$M$.
\end{lem}

\begin{proof}
Since $d \geq d_1$, $\R\H$ is $0$-ample. By a transversality argument (see~\cite[Sect.~2.6]{Let2016} for details), this implies that the restriction of $s$ to $M$ vanishes transversally for almost every $s \in \R\H$ (with respect to the Lebesgue measure). Thus, almost surely, $s_d$ restricted to $M$ vanishes transversally, and its zero set is a smooth closed submanifold of codimension~$r$.
\end{proof}

From now on, we only consider the case $d \geq d_1$, so that $Z_{s_d}$ is almost surely a random smooth closed submanifold of $M$ of codimension $r$. For simplicity, we denote $Z_d = Z_{s_d}$ and $\rmes{d}=\rmes{s_d}$. Let $\phi\in \mathcal{C}^0(M)$ and $d \geq d_1$, the real random variable $\prsc{\rmes{d}}{\phi}$ is called the \emph{linear statistic} of degree~$d$ associated with $\phi$. For example, $\prsc{\rmes{d}}{\mathbf{1}}$ is the volume of $Z_d$.


\subsection{The correlation kernel}
\label{subsec correlation kernel}

For any $d \in \N$, the random section $s_d \in \R\H$ defines a centered Gaussian process $(s_d(x))_{x \in \X}$. In this section, we recall the relation between the distribution of this process and the Bergman kernel of $\E \otimes \L^d$.

Recall that $(\mathcal{E}\otimes \mathcal{L}^d) \boxtimes (\mathcal{E}\otimes \mathcal{L}^d)^*$ stands for the bundle $P_1^\star\left(\mathcal{E}\otimes \mathcal{L}^d\right) \otimes P_2^\star \left(\left(\mathcal{E}\otimes \mathcal{L}^d\right)^*\right)$ over $\mathcal{X}\times \mathcal{X}$, where $P_1$ (resp.~$P_2$) denotes the projection from $\X \times \X$ onto the first (resp.~second) $\X$ factor. The distribution of $(s_d(x))_{x \in \X}$ is characterized by its covariance kernel, that is the section of $(\mathcal{E}\otimes \mathcal{L}^d) \boxtimes (\mathcal{E}\otimes \mathcal{L}^d)^*$ defined by: $(x,y) \mapsto \cov{s_d(x)}{s_d(y)} = \esp{s_d(x) \otimes s_d(y)^*}$.

\begin{dfn}
\label{def Bergman kernel}
Let $E_d$ denote the \emph{Bergman kernel} of $\E \otimes \L^d \to \X$, that is the Schwartz kernel of the orthogonal projection from $\Gamma(\E \otimes \L^d)$ onto $\H$.
\end{dfn}

Let $(s_{1,d},\dots,s_{N_d,d})$ be an orthonormal basis of $\R \H$, then it is also an orthonormal basis of $\H$. Recall that, $E_d$ is given by:
\begin{equation*}
\label{eq Bergman in coordinates}
E_d : (x,y) \longmapsto \sum_{i=1}^{N_d} s_{i,d}(x) \otimes s_{i,d}(y)^*.
\end{equation*}
This shows that $E_d$ is a real global holomorphic section of $(\mathcal{E}\otimes \mathcal{L}^d) \boxtimes (\mathcal{E}\otimes \mathcal{L}^d)^*$. The following proves that the distribution of $(s_d(x))_{x \in \X}$ is totally described by $E_d$.

\begin{prop}[compare \cite{Let2016a}, Prop.~2.6]
\label{prop covariance equals Bergman}
Let $d \in \N$ and let $s_d$ be a standard Gaussian vector in $\R \H$. Then, for all $x$ and $y \in \X$, we have: $\cov{s_d(x)}{s_d(y)} = E_d(x,y)$.
\end{prop}

Thus, the Bergman kernel of $\E \otimes \L^d$ gives the correlations between the values of the random section $s_d$. By taking partial derivatives of this relation, we obtain the correlations between the values of $s_d$ and its derivatives. More details about what follows can be found in \cite[Sect.~2.3]{Let2016a}.

Let $\nabla^d$ be a metric connection on $\E \otimes \L^d$, it induces a dual connection $(\nabla^d)^*$ on $(\E \otimes \L^d)^*$, which is compatible with the metric (cf.~\cite[Sect.~0.5]{GH1994}). We can then define a natural metric connection $\nabla_1^d$ on $P_1^\star(\E \otimes \L^d) \to \X \times \X$ whose partial derivatives are: $\nabla^d$ with respect to the first variable, and the trivial connection with respect to the second. Similarly, $(\nabla^d)^*$ induces a metric connection $\nabla_2^d$ on $P_2^\star\left((\E \otimes \L^d)^*\right)$ and $\nabla_1^d \otimes \Id + \Id \otimes \nabla_2^d$ is a metric connection on $(\E \otimes \L^d) \boxtimes (\E \otimes \L^d)^*$.

We denote by $\partial_x$ (resp. $\partial_y$) the partial derivative of $\nabla_1^d \otimes \Id + \Id \otimes \nabla_2^d$ with respect to the first (resp. second) variable. Let $\partial_y^\sharp E_d(x,y) \in T_y\X \otimes \left(\E \otimes \L^d\right)_x \otimes \left(\E \otimes \L^d\right)_y^*$ be defined by:
\begin{equation*}
\label{eq def sharp}
\forall w \in T_y\X, \qquad \partial_y^\sharp E_d(x,y) \cdot w^* = \partial_y E_d(x,y) \cdot w.
\end{equation*}
Similarly, let $\partial_x\partial_y^\sharp E_d(x,y) \in T^*_x\X \otimes T_y\X \otimes \left(\E \otimes \L^d\right)_x \otimes \left(\E \otimes \L^d\right)_y^*$ be defined by:
\begin{equation*}
\label{eq def sharp 2}
\forall (v,w) \in T_x\X \times T_y\X, \qquad \partial_x\partial_y^\sharp E_d(x,y) \cdot (v,w^*) = \partial_x\partial_y E_d(x,y)\cdot (v,w).
\end{equation*}
The following corollary was proved in~\cite[Cor.~2.13]{Let2016a}.
\begin{cor}
\label{cor variance of the 1-jet sharp}
Let $d \in \N$, let $\nabla^d$ be a metric connection on $\E \otimes \L^d$ and let $s_d$ be a standard Gaussian vector in $\R \H$. Then, for all $x$ and $y \in \X$, we have:
\begin{align*}
\cov{\nabla^d_xs}{s(y)} &= \esp{\nabla^d_xs \otimes s(y)^*} = \partial_x E_d(x,y),\\
\cov{s(x)}{\nabla^d_ys} &= \esp{s(x) \otimes \left(\nabla^d_ys\right)^*} =\partial_y^\sharp E_d(x,y),\\
\cov{\nabla^d_xs}{\nabla^d_ys} &= \esp{\nabla^d_xs \otimes \left(\nabla^d_ys\right)^*} = \partial_x\partial_y^\sharp E_d(x,y).
\end{align*}
\end{cor}


\section{Estimates for the Bergman kernel}
\label{sec estimates for the bergman kernel}

In this section, we recall useful estimates for the Bergman kernels. In Sect.~\ref{subsec real normal trivialization} we recall the definition of a preferred trivialization of $\E \otimes \L^d \to \X$. Then we state near-diagonal and off-diagonal estimates for a scaled version of $E_d$ in Sect.~\ref{subsec near-diagonal estimates} and~\ref{subsec off-diagonal estimates}.


\subsection{Real normal trivialization}
\label{subsec real normal trivialization}

\begin{ntn}
In the following, $B_A(a,R)$ denotes the open ball of center $a$ and radius $R$ in the metric space~$A$.
\end{ntn}

Let $d \in \N$, let us define a preferred trivialization of $\E \otimes \L^d$ in a neighborhood of any point of~$M$. The properties of this trivialization were studied in \cite[Sect.~3.1]{Let2016a}. Recall that the metric $g$ on $\X$ is induced by the curvature of $(\L,h_\L)$. Since, $h_\L$ is compatible with the real structures, $c_\X$ is an isometry of $(\X,g)$ (see \cite[Sect.~2.1]{Let2016a} for details).

Let $R >0$ be such that $2R$ is less than the injectivity radius of $\X$. Let $x \in M$, then the exponential map $\exp_x : B_{T_x\X}(0,2R) \to B_\X(x,2R)$ defines a chart around $x$ such that:
\begin{equation*}
\label{eq real equivariance exp}
d_xc_\X = (\exp_x)^{-1} \circ c_\X \circ \exp_x.
\end{equation*}
In particular, since $T_xM = \ker(d_xc_\X - \Id)$, we have that $\exp_x$ induces a diffeomorphism from $B_{T_xM}(0,2R)$ to $B_M(x,2R)$ that coincides with the exponential map of $(M,g)$ at $x$.

We can now trivialize $\E \otimes \L^d$ over $B_\X(x,2R)$, by identifying each fiber with $\left(\E \otimes \L^d\right)_x$ by parallel transport along geodesics, with respect to the Chern connection of $\E \otimes \L^d$. This defines a bundle map 
\begin{equation*}
\varphi_x : B_{T_x\X}(0,2R) \times \left(\E \otimes \L^d\right)_x \to \left(\E \otimes \L^d\right)_{/B_\X(x,2R)}
\end{equation*}
that covers $\exp_x$. We call this trivialization the \emph{real normal trivialization} of $\E \otimes \L^d$ around $x$.

\begin{dfn}
\label{def real connection}
A connection $\nabla^d$ on $\E \otimes \L^d \to \X$ is said to be \emph{real} if for every smooth section $s$ we have:
\begin{equation*}
\forall y \in \X, \qquad \nabla^d_y \left( c_d \circ s \circ c_\X\right) = c_d \circ \nabla^d_{c_\X(y)} s \circ d_yc_\X.
\end{equation*}
Such a connection induces a connection on $\R\!\left(\E \otimes \L^d\right) \to M$ by restriction.
\end{dfn}

Recall that $c_d$ denotes the real structure of $\E \otimes \L^d$. Let $c_{d,x}$ denote its restriction to $\left(\E \otimes \L^d\right)_x$, then $(d_xc_\X,c_{d,x})$ is a $\C$-anti-linear involution of $B_{T_x\X}(0,2R) \times \left(\E \otimes \L^d\right)_x$ which is compatible with the real structure on the first factor. Since the Chern connection of $\E \otimes \L^d$ is real (see~\cite[Lem.~3.4]{Let2016a}), the real normal trivialization is well-behaved with respect to the real structures, in the sense that for all $z \in B_{T_x\X}(0,2R)$ and $\zeta \in \left(\E \otimes \L^d\right)_x$,
\begin{equation*}
\label{eq equivariance trivialization real structures}
c_d(\varphi_x(z,\zeta))= \varphi_x\left(d_xc_\X \cdot z, c_{d,x}(\zeta)\right).
\end{equation*}
Thus, $\varphi_x$ can be restricted to a bundle map
\begin{equation*}
B_{T_xM}(0,2R) \times \R\!\left(\E \otimes \L^d\right)_x \to \R\!\left(\E \otimes \L^d\right)_{/B_M(x,2R)}
\end{equation*} 
that covers $\exp_x$.

Finally, it is known (cf.~\cite[Sect.~3.1]{Let2016a}) that $\varphi_x$ is a unitary trivialization, i.e.~its restriction to each fiber is an isometry. Similarly, its restriction to the real locus is an orthogonal trivialization of $\R\!\left(\E \otimes \L^d\right)_{/B_M(x,2R)}$.

The point is the following. The usual differentiation for maps from $T_x\X$ to $\left(\E \otimes \L^d\right)_x$ defines locally a connection $\tilde{\nabla}^d$ on $\left(\E \otimes \L^d\right)_{/B_{\X}(x,2R)}$ via the real normal trivialization. Since this trivialization is well-behaved with respect to both the real and the metric structures, $\tilde{\nabla}^d$ is a real metric connection. Then, by a partition of unity argument, there exists a global real metric connection $\nabla^d$ on $\E \otimes \L^d$ that agrees with $\tilde{\nabla}^d$ on $B_\X(x,R)$, i.e.~$\nabla^d$ is trivial in the real normal trivialization, in a neighborhood of $x$. The existence of such a connection will be useful in the proof of our main theorem.


\subsection{Near-diagonal estimates}
\label{subsec near-diagonal estimates}

In this section, we state estimates for a scaled version of the Bergman kernel in a neighborhood of the diagonal of $M \times M$. As in the previous section, let $R>0$ be such that $2R$ is less than the injectivity radius of $\X$. Let $x \in M$, then the real normal trivialization $\varphi_x$ induces a trivialization of $\left(\E \otimes \L^d\right) \boxtimes \left(\E \otimes \L^d\right)^*$ over $B_\X(x,2R) \times B_\X(x,2R)$ that covers $\exp_x \times \exp_x$. This trivialization agrees with the real normal trivialization of $\left(\E \otimes \L^d\right) \boxtimes \left(\E \otimes \L^d\right)^*$ around $(x,x)$.

In the normal chart $\exp_x$, the Riemannian measure $\dx V_\X$ admits a positive density with respect to the Lebesgue measure of $(T_x\X,g_x)$, denoted by $\kappa : B_{T_x\X}(0,2R) \to \R_+$. Then, in the chart $\exp_x:B_{T_xM}(0,2R) \to B_M(x,2R)$, the density of $\rmes{M}$ with respect to the Lebesgue measure of $(T_xM,g_x)$ is $\sqrt{\kappa}$.

Let us identify $E_d$ with its expression in the real normal trivialization of $\left(\E \otimes \L^d\right) \boxtimes \left(\E \otimes \L^d\right)^*$ around $(x,x)$. Thus, the restriction of $E_d$ to the real locus is a map from $T_xM \times T_xM$ to $\End\left(\R\!\left(\E \otimes \L^d\right)_x\right)$. Then, by~\cite[Thm.~4.18']{DLM2006} (see also~\cite[Thm.~3.5]{Let2016a} for a statement with the same notations as the present paper) we get the following estimate for $E_d$ and its derivatives of order at most~$6$.

\begin{thm}[Dai--Liu--Ma]
\label{thm Dai Liu Ma}
There exist $C_1$ and $C_2 >0$, such that $\forall k \in \{0,1,\dots,6\}$, $\forall d \in \N^*$, $\forall x \in M$, $\forall w,z \in B_{T_xM}(0,R)$,
\begin{multline*}
\Norm{D^k_{(w,z)}\left(E_d(w,z) - \left(\frac{d}{\pi}\right)^n\frac{\exp\left(-\frac{d}{2}\Norm{z-w}^2\right)}{\sqrt{\kappa(w)}\sqrt{\kappa(z)}}\Id_{\R(\E\otimes\L^d)_x}\right)}\\
\leq C_1 d^{n+\frac{k}{2}-1}\left(1+\sqrt{d}(\Norm{w}+\Norm{z})\right)^{2n+12} \exp\left(-C_2\sqrt{d}\Norm{z-w}\right)+O\!\left(d^{-\infty}\right),
\end{multline*}
where $D^k$ is the $k$-th differential for a map from $T_xM \times T_xM$ to $\End\left(\R\!\left(\E \otimes \L^d\right)_x\right)$, the norm on $T_xM$ is induced by $g_x$ and the norm on $\left(T^*_xM\right)^{\otimes k}\otimes \End\left(\left(\E \otimes \L^d\right)_x\right)$ is induced by $g_x$ and $(h_d)_x$.

The notation $O\!\left(d^{-\infty}\right)$ means that, for any $l \in \N$, this term is $O\!\left(d^{-l}\right)$ with a constant that does not depend on $x$, $w$, $z$ nor $d$.
\end{thm}

Recall that $x$ is fixed. We denote by the $e_d$ the following scaled version of the Bergman kernel around $x$:
\begin{equation}
\label{eq def ed}
\forall w,z \in B_{T_xM}\left(0,2R \sqrt{d}\right), \qquad e_d(w,z) = \left(\frac{\pi}{d}\right)^n E_d\left(\exp_x\left(\frac{w}{\sqrt{d}}\right),\exp_x\left(\frac{z}{\sqrt{d}}\right)\right).
\end{equation}
We consider $e_d$ as a map with values in $\End \left(\R\!\left(\E \otimes \L^d\right)_x\right)$ using the real normal trivialization around $x$. Note that $e_d$ highly depends on $x$, even if this is not reflected in the notation. In the following, the base point $x$ will always be clear from the context.

Let $\xi : \R^n \times \R^n \to \R$ be defined by $\xi(w,z) = \exp \left(-\frac{1}{2}\Norm{w-z}^2\right)$, where $\Norm{\cdot}$ is the usual Euclidean norm. Let $x \in M$, any isometry from $T_xM$ to $\R^n$ allows us to see $\xi$ as a map from $T_xM \times T_xM$ to $\R$. Let $b_n$ be a positive constant depending only on $n$ and whose value will be chosen later on. Then, using the same notations as in Thm.~\ref{thm Dai Liu Ma} we get the following.

\begin{prop}
\label{prop near diag estimates}
Let $\alpha \in (0,1)$, then there exists $C>0$, depending only on $\alpha$, $n$ and the geometry of $\X$, such that $\forall k \in \{0,1,\dots,6\}$, $\forall d \in \N^*$, $\forall x \in M$, $\forall w,z \in B_{T_xM}(0,b_n \ln d)$, we have:
\begin{equation*}
\label{eq near diag estimates}
\Norm{D^k_{(w,z)}e_d - \left(D^k_{(w,z)}\xi\right)\Id_{\R(\E\otimes\L^d)_x}} \leq Cd^{-\alpha}.
\end{equation*}
\end{prop}

\begin{proof}
First, we apply Thm.~\ref{thm Dai Liu Ma} for the scaled kernel $e_d$. This yields that $\forall k \in \{0,1,\dots,6\}$, $\forall d \in \N^*$, $\forall x \in M$, $\forall w,z \in B_{T_xM}(0,b_n \ln d)$:
\begin{equation*}
\Norm{D^k_{(w,z)}\left(e_d(w,z)-\frac{\xi(w,z)}{\sqrt{\tilde{\kappa}(w)\tilde{\kappa}(z)}}\Id_{\R(\E\otimes\L^d)_x}\right)} \leq \frac{C_1}{d} \left(1+2b_n \ln d\right)^{2n+12} +O\!\left(d^{-\infty}\right) = O\!\left(d^{-\alpha}\right),
\end{equation*}
where $\tilde{\kappa}: z \mapsto \kappa\left(\frac{z}{\sqrt{d}}\right)$. Since we used normal coordinates to define $\kappa$, the following relations hold uniformly on $B_{T_x\X}(0,R)$:
\begin{align*}
\kappa(z) &= 1 + O\!\left(\Norm{z}^2\right), & D_z\kappa &= O\!\left(\Norm{z}\right) & &\text{and} & \forall k \in \{2,\dots,6\}, \ D_z^k \kappa & = O(1).
\end{align*}
By compactness, these estimates can be made independent of $x \in M$. Then, we get the following estimates for $\tilde{\kappa}$ and its derivatives, uniformly in $x \in M$ and $z \in B_{T_xM}(0,b_n \ln d)$:
\begin{align*}
\tilde{\kappa}(z) &= 1 + O\!\left(\frac{(b_n \ln d)^2}{d}\right), & D_z\tilde{\kappa} &= O\!\left(\frac{b_n \ln d}{d}\right) & &\text{and} & \forall k \in \{2,\dots,6\}, \ D_z^k \tilde{\kappa} & = O\!\left(\frac{1}{d}\right).
\end{align*}
Therefore, $\forall k \in \{0,1,\dots,6\}$, $\forall d \in \N^*$, $\forall x \in M$, $\forall w,z \in B_{T_xM}(0,b_n \ln d)$:
\begin{equation*}
\Norm{D^k_{(w,z)}\left(\frac{\xi(w,z)}{\sqrt{\tilde{\kappa}(w)\tilde{\kappa}(z)}}\right) - D^k_{(w,z)}\xi} = O\!\left(d^{-\alpha}\right). \qedhere
\end{equation*}
\end{proof}

We will use the expressions of some of the partial derivatives of $\xi$. Let us choose any orthonormal basis of $T_xM$ and denote by $\partial_{x_i}$ (resp.~$\partial_{y_i}$) the partial derivative with respect to the $i$-th component of the the first (resp.~second) variable.

\begin{lem}
\label{lem values of limit derivatives}
Let $i,j \in \{1,\dots,n\}$, for all $w=(w_1,\dots,w_n)$ and $z=(z_1,\dots,z_n) \in T_xM$ we have:
\begin{align*}
\partial_{x_i}\xi(w,z) &= -(w_i - z_i)\exp \left(-\frac{1}{2}\Norm{w-z}^2\right),\\
\partial_{y_j}\xi(x,y) &= (w_j - z_j)\exp \left(-\frac{1}{2}\Norm{w-z}^2\right),\\
\text{and} \qquad \partial_{x_i}\partial_{y_j}\xi(x,y) &= \left(\delta_{ij} -(w_i - z_i)(w_j-z_j)\right)\exp \left(-\frac{1}{2}\Norm{w-z}^2\right),
\end{align*}
where $\delta_{ij}$ equals $1$ if $i=j$ and $0$ otherwise.
\end{lem}

\begin{proof}
This is given by a direct computation.
\end{proof}


\subsection{Off-diagonal estimates}
\label{subsec off-diagonal estimates}

Finally, let us recall estimates quantifying the long range decay of the Bergman kernel $E_d$. These estimates were proved by Ma and Marinescu in \cite[Thm.~5]{MM2015}.

Let $S$ be a smooth section of $\R\!\left(\E \otimes \L^d\right) \boxtimes \R\!\left(\E \otimes \L^d\right)^*$ and  let $x,y \in M$. We denote by $\Norm{S(x,y)}_{\mathcal{C}^k}$ the maximum of the norms of $S$ and its derivatives of order at most $k$ at $(x,y)\in M\times M$, where the derivatives of $S$ are computed with respect to the connection induced by the Chern connection of $\E\otimes\L^d$ and the Levi--Civita connection on $M$. The norms are the natural ones induced by $h_d$ and $g$.

\begin{thm}[Ma--Marinescu]
\label{thm off diag estimates}
There exist $d_0 \in \N^*$, and positive constants $C_1$ and $C_2$ such that, for all $k \in \{0,1,2\}$, $\forall d \geq d_0$, $\forall x,y \in M$, we have:
\begin{equation*}
\Norm{E_d(x,y)}_{\mathcal{C}^k} \leq C_1 d^{n+\frac{k}{2}} \exp \left(-C_2 \sqrt{d}\, \rho_g(x,y)\right),
\end{equation*}
where $\rho_g(\cdot,\cdot)$ denotes the geodesic distance in $(M,g)$.
\end{thm}


\section{Properties of the limit distribution}
\label{sec properties of the limit distribution}

The estimates of Sect.~\ref{subsec near-diagonal estimates} show that the family of random fields $(s_d(x))_{x \in M}$ has a local scaling limit around any point $x \in M$, as $d$ goes to infinity. Moreover, the limit field does not depend on $x$. The limit is a Gaussian centered random process from $\R^n$ to $\R^r$ whose correlation kernel is $e_\infty : (w,z) \mapsto \xi(w,z) I_r$, where $I_r$ stands for the identity of $\R^r$ and $\xi$ was defined in Sect.~\ref{subsec near-diagonal estimates}. This limit process is known as the Bargmann--Fock process.

The goal of this section is to establish some properties of the Bargmann--Fock process. These results will be useful in the next section to prove that, for $d$ large enough, the local behavior of $s_d$ around any given $x \in M$ is the same as that of the limit process.

In the following, we denote by $(s(z))_{z \in \R^n}$ a copy of the Bargmann--Fock process. Since $\xi$ is smooth, we can assume the trajectories of $s$ to be smooth. Note that $s$ is both stationary and isotropic. Moreover, since $e_\infty = \xi I_r$, the field $s$ is just a tuple of $r$ independent identically distributed centered Gaussian fields whose correlation kernel is $\xi$.


\subsection{Variance of the values}
\label{subsec variance of the values}

The first thing we want to understand about $s$ is the distribution of $(s(0),s(z))\in \R^r \oplus \R^r$ for any $z \in \R^n$. In the following, we canonically identify $\R^r \oplus \R^r$ with $\R^2 \otimes \R^r$.

Let $z \in \R^n$, then $(s(0),s(z))$ is a centered Gaussian vector in $\R^2 \otimes \R^r$ whith variance operator:
\begin{equation}
\label{eq def Theta}
\Theta(z) = \begin{pmatrix}
e_\infty(0,0) & e_\infty(0,z) \\ e_\infty(z,0) & e_\infty(z,z)
\end{pmatrix} = \begin{pmatrix}
\xi(0,0)I_r & \xi(0,z)I_r \\ \xi(z,0)I_r & \xi(z,z)I_r
\end{pmatrix} = \begin{pmatrix}
1 & e^{-\frac{1}{2}\Norm{z}^2} \\ e^{-\frac{1}{2}\Norm{z}^2} & 1
\end{pmatrix} \otimes I_r.
\end{equation}

Let $Q = \frac{1}{\sqrt{2}} \left(\begin{smallmatrix} 1 & -1 \\ 1 & 1 \end{smallmatrix} \right) \in O_2(\R)$ denote the rotation of angle $\frac{\pi}{4}$ in $\R^2$. We can explicitly diagonalize $\Theta(z)$ as follows.

\begin{lem}
\label{lem diagonalization Theta}
For any $z \in \R^n$ we have the following:
\begin{equation*}
\left(Q \otimes I_r\right) \Theta(z) \left(Q \otimes I_r\right)^{-1} = \begin{pmatrix}
1 - e^{-\frac{1}{2}\Norm{z}^2} & 0 \\ 0 & 1+ e^{-\frac{1}{2}\Norm{z}^2} \end{pmatrix} \otimes I_r.
\end{equation*}
\end{lem}

\begin{proof}
Since $\left(Q \otimes I_r\right)^{-1} = Q^\text{t} \otimes I_r$, by Eq.~\eqref{eq def Theta}, it is enough to notice that:
\begin{equation*}
Q \begin{pmatrix}
1 & e^{-\frac{1}{2}\Norm{z}^2} \\ e^{-\frac{1}{2}\Norm{z}^2} & 1
\end{pmatrix} Q^\text{t} = \begin{pmatrix}
1 - e^{-\frac{1}{2}\Norm{z}^2} & 0 \\ 0 & 1+ e^{-\frac{1}{2}\Norm{z}^2} \end{pmatrix}.\qedhere
\end{equation*}
\end{proof}

\begin{lem}
\label{lem det Theta}
For all $z \in \R^n$, $\det\left(\Theta(z)\right) = \left(1-e^{-\Norm{z}^2}\right)^r$. In particular, the distribution of $(s(0),s(z))$ is non-degenerate for all $z \in \R^n \setminus \{0\}$.
\end{lem}

\begin{proof}
We take the determinant of both sides in Eq.~\eqref{eq def Theta}.
\end{proof}


\subsection[Variance of the 1-jets]{Variance of the $1$-jets}
\label{subsec variance of the 1-jets}

Let us now study the variance structure of the $1$-jets of $s$. For any $z \in \R^n$, we know that $(s(0),s(z),d_0s,d_zs)$ is a centered Gaussian vector in:
\begin{equation*}
\label{eq splitting}
\R^r \oplus \R^r \oplus \left((\R^n)^* \otimes \R^r\right) \oplus \left((\R^n)^* \otimes \R^r\right) \simeq \left(\R \oplus \R \oplus (\R^n)^* \oplus (\R^n)^*\right) \otimes \R^r.
\end{equation*}
Our goal in this section is to better understand its variance operator $\Omega(z)$. In the following, we write $\Omega(z)$ by blocks according to the previous splitting. Let $\partial_x$ (resp.~$\partial_y$) denote the partial derivative with respect to the first (resp.~second) variable for maps from $\R^n \times \R^n$ to $\End(\R^r)$. Let us also define $\partial_y^\sharp$ as in Sect.~\ref{subsec correlation kernel}. Then, we have:
\begin{equation}
\label{eq expression Omega}
\Omega(z) = \left( \begin{array}{cc|cc}
e_\infty(0,0) & e_\infty(0,z) & \partial^\sharp_y e_\infty(0,0) & \partial^\sharp_y e_\infty(0,z) \\ \rule[-5pt]{0pt}{15pt} e_\infty(z,0) & e_\infty(z,z) & \partial^\sharp_y e_\infty(z,0) & \partial^\sharp_y e_\infty(z,z) \\
\hline
\rule[-5pt]{0pt}{15pt}
\partial_x e_\infty(0,0) & \partial_x e_\infty(0,z) & \partial_x \partial^\sharp_y e_\infty(0,0) & \partial_x \partial^\sharp_y e_\infty(0,z) \\
\partial_x e_\infty(z,0) & \partial_x e_\infty(z,z) & \partial_x \partial^\sharp_y e_\infty(z,0) & \partial_x \partial^\sharp_ye_\infty(z,z)
\end{array} \right) = \Omega'(z) \otimes I_r,
\end{equation}
where
\begin{equation*}
\label{eq def Omega prime}
\Omega'(z) = \left( \begin{array}{cc|cc}
\xi(0,0) & \xi(0,z) & \partial^\sharp_y \xi(0,0) & \partial^\sharp_y \xi(0,z) \\ \rule[-5pt]{0pt}{15pt} \xi(z,0) & \xi(z,z) & \partial^\sharp_y \xi(z,0) & \partial^\sharp_y \xi(z,z) \\
\hline
\rule[-5pt]{0pt}{15pt}
\partial_x \xi(0,0) & \partial_x \xi(0,z) & \partial_x \partial^\sharp_y \xi(0,0) & \partial_x \partial^\sharp_y \xi(0,z) \\
\partial_x \xi(z,0) & \partial_x \xi(z,z) & \partial_x \partial^\sharp_y \xi(z,0) & \partial_x \partial^\sharp_y \xi(z,z)
\end{array} \right).
\end{equation*}

Let $\left(\deron{}{x_1},\dots,\deron{}{x_n}\right)$ be any orthonormal basis of $\R^n$ such that $z = \Norm{z}\deron{}{x_1}$ and let $(dx_1,\dots,dx_n)$ denote its dual basis. Let $(e_1,e_2)$ denote the canonical basis of $\R^2$, we denote by $\mathcal{B}_z$ the following orthonormal basis of $\R^2 \otimes \left(\R \oplus (\R^n)^*\right) \simeq \R \oplus \R \oplus (\R^n)^* \oplus (\R^n)^*$:
\begin{equation*}
\label{eq def Bz}
\mathcal{B}_z = (e_1\otimes 1, e_2 \otimes 1, e_1 \otimes dx_1,e_2 \otimes dx_1, \dots, e_1 \otimes dx_n, e_2 \otimes dx_n).
\end{equation*}

\begin{lem}
\label{lem variance operator 1-jets}
For any $z \in \R^n$, the matrix of $\Omega'(z)$ in the basis $\mathcal{B}_z$ is:
\begin{equation*}
\left(\begin{array}{c|c}
\tilde{\Omega}(\Norm{z}^2) & 0 \\
\hline
\rule{0pt}{14pt}
0 & \left(\begin{smallmatrix}
1 & e^{-\frac{1}{2}\Norm{z}^2} \\ e^{-\frac{1}{2}\Norm{z}^2} & 1
\end{smallmatrix}\right) \otimes I_{n-1}
\end{array}\right),
\end{equation*}
where $I_{n-1}$ is the identity matrix of size $n-1$ and, for all $t \geq 0$, we set:
\begin{equation}
\label{eq def Omega tilde}
\tilde{\Omega}(t) = \begin{pmatrix}
1 & e^{-\frac{1}{2}t} & 0 & -\sqrt{t}e^{-\frac{1}{2}t}\\
e^{-\frac{1}{2}t} & 1 & \sqrt{t}e^{-\frac{1}{2}t} & 0\\
0 & \sqrt{t}e^{-\frac{1}{2}t} & 1 & (1-t)e^{-\frac{1}{2}t}\\
-\sqrt{t}e^{-\frac{1}{2}t} & 0 & (1-t)e^{-\frac{1}{2}t} & 1
\end{pmatrix}.
\end{equation}
\end{lem}

\begin{proof}
A direct computation yields the result, using the fact that $z = (\Norm{z},0,\dots,0)$ in the basis $\left(\deron{}{x_1},\dots,\deron{}{x_n}\right)$. Recall that the partial derivatives of $\xi$ are given by Lem.~\ref{lem values of limit derivatives}.
\end{proof}

Let $z \in \R^n$, and recall that $z^* \in (\R^n)^*$ was defined as $z^* = \prsc{\cdot}{z}$, where $\prsc{\cdot}{\cdot}$ is the canonical scalar product of~$\R^n$. We denote by $z^* \otimes z \in \End((\R^n)^*)$ the map $\eta \mapsto \eta(z)z^*$. Then Lem.~\ref{lem variance operator 1-jets} shows that, as an operator on $\R \oplus \R \oplus (\R^n)^* \oplus (\R^n)^*$:
\begin{equation}
\label{eq Omega prime}
\Omega'(z) = \left(\begin{array}{cc|cc}
1 & e^{-\frac{1}{2}\Norm{z}^2} & 0 & -e^{-\frac{1}{2}\Norm{z}^2} z \\
e^{-\frac{1}{2}\Norm{z}^2} & 1 & e^{-\frac{1}{2}\Norm{z}^2} z & 0 \\
\hline
\rule{0pt}{12pt} 0 & e^{-\frac{1}{2}\Norm{z}^2} z^* & I_n & e^{-\frac{1}{2}\Norm{z}^2}(I_n - z^* \otimes z) \\
-e^{-\frac{1}{2}\Norm{z}^2} z^* & 0 & e^{-\frac{1}{2}\Norm{z}^2}(I_n - z^* \otimes z) & I_n
\end{array}\right),
\end{equation}
where $z^*$ is to be understood as the constant map $t \mapsto z^*$ from $\R$ to $(\R^n)^*$, $z$ is to be understood as the evaluation at the point $z$ from $(\R^n)^*$ to $\R$ and $I_n$ is the identity of $\R^n$. Indeed both sides of~\eqref{eq Omega prime} have the same matrix in the basis $\mathcal{B}_z$.

We will now diagonalize explicitly $\Omega(z)$, as we did for $\Theta(z)$ in the previous section. The main step is to diagonalize $\tilde{\Omega}(z)$.

\begin{dfns}
\label{def a and vi}
We denote by $v_1$, $v_2$, $v_3$, $v_4$ and $a$ the following functions from $[0,+\infty)$ to $\R$:
\begin{align*}
v_1 : t \longmapsto & 1 - e^{-\frac{1}{2}t}\left(\frac{t}{2} - \sqrt{1 + \left(\frac{t}{2}\right)^2}\right), & v_2 : t \longmapsto & 1 - e^{-\frac{1}{2}t}\left(\frac{t}{2} + \sqrt{1 + \left(\frac{t}{2}\right)^2}\right),\\
v_3 : t \longmapsto & 1 + e^{-\frac{1}{2}t}\left(\frac{t}{2} - \sqrt{1 + \left(\frac{t}{2}\right)^2}\right), & v_4 : t \longmapsto & 1 + e^{-\frac{1}{2}t}\left(\frac{t}{2} + \sqrt{1 + \left(\frac{t}{2}\right)^2}\right),\\
a : t \longmapsto & \frac{1- \frac{t}{2}}{\sqrt{1 + \left(\frac{t}{2}\right)^2}}.
\end{align*}
\end{dfns}

Note that, for all $t \geq 0$, $\norm{a(t)} \leq 1$, so that the following makes sense.

\begin{dfns}
\label{def b and P}
Let $b_+ : t \mapsto \sqrt{1+a(t)}$ and $b_-: t \mapsto \sqrt{1-a(t)}$ from $[0,+\infty)$ to $\R$. For all $t \geq 0$, let us also denote:
\begin{equation*}
P(t) = \frac{1}{2} \begin{pmatrix}
b_-(t) & -b_-(t) & -b_+(t) & -b_+(t)\\
b_+(t) & -b_+(t) & b_-(t) & b_-(t)\\
b_-(t) & b_-(t) & -b_+(t) & b_+(t)\\
b_+(t) & b_+(t) & b_-(t) & -b_-(t)
\end{pmatrix}.
\end{equation*}
One can check that, for all $t \geq 0$, $P(t)$ is an orthogonal matrix.
\end{dfns}

\begin{lem}
\label{lem diagonalization Omega tilde}
For every $t \in [0,+\infty)$, we have:
\begin{equation*}
P(t) \tilde{\Omega}(t) P(t)^{-1} = \begin{pmatrix}
v_1(t) & 0 & 0 & 0\\
0 & v_2(t) & 0 & 0\\
0 & 0 & v_3(t) & 0\\
0 & 0 & 0 & v_4(t)
\end{pmatrix}.
\end{equation*}
\end{lem}

\begin{proof}
See Appendix~\ref{sec technical computations for Section limit distrib}.
\end{proof}

\begin{cor}
\label{cor diagonalization Omega'}
Let $z \in \R^n$, identifying $\Omega'(z)$ with its matrix in $\mathcal{B}_z$, we have:
\begin{multline*}
\left(\begin{array}{c|c}
P(\Norm{z}^2) & 0 \\
\hline
0 & Q \otimes I_{n-1}
\end{array}\right)
\Omega'(z)
\left(\begin{array}{c|c}
P(\Norm{z}^2) & 0 \\
\hline
0 & Q \otimes I_{n-1}
\end{array}\right)^{-1}\\
= 
\left(\begin{array}{c|c}
\begin{smallmatrix}
v_1(\Norm{z}^2) & 0 & 0 & 0\\
0 & v_2(\Norm{z}^2) & 0 & 0\\
0 & 0 & v_3(\Norm{z}^2) & 0\\
0 & 0 & 0 & \rule[-3pt]{0pt}{1pt} v_4(\Norm{z}^2)
\end{smallmatrix} & 0 \\
\hline
\rule{0pt}{16pt}
0 & \left(\begin{smallmatrix}
1 - e^{-\frac{1}{2}\Norm{z}^2} & 0 \\ 0 & 1 + e^{-\frac{1}{2}\Norm{z}^2}
\end{smallmatrix}\right) \otimes I_{n-1}
\end{array}\right).
\end{multline*}
\end{cor}
By Eq.~\eqref{eq expression Omega}, we get a diagonalization of $\Omega(z)$ by tensoring each factor by $I_r$ in Corollary~\ref{cor diagonalization Omega'}.

\begin{lem}
\label{lem non-degeneracy Omega}
For all $z \in \R^n \setminus \{0\}$, we have $\det\left(\Omega(z)\right) >0$, that is the distribution of the random vector $(s(0),s(z),d_0s,d_zs)$ is non-degenerate.
\end{lem}

\begin{proof}
See Appendix~\ref{sec technical computations for Section limit distrib}.
\end{proof}


\subsection{Conditional variance of the derivatives}
\label{subsec conditional variance of the derivatives}

The next step is to study the conditonal distribution of $(d_0s,d_zs)$ given that $s(0)=0=s(z)$, for any $z \in \R^n \setminus \{0\}$. Recall that $\left(s(0),s(z),d_0s,d_zs\right)$ is a centered Gaussian vector with variance $\Omega(z)$ (see Eq.~\eqref{eq expression Omega}). Moreover, if $z \neq 0$, the distribution of $(s(0),s(z))$ is non-degenerate by Lem.~\ref{lem det Theta}.

Thus, $(d_0s,d_zs)$ given that $s(0)=0=s(z)$ is a centered Gaussian vector in $\left((\R^n)^* \oplus (\R^n)^*\right) \otimes \R^r$ with variance operator:
\begin{equation*}
\Lambda(z) = \left(\begin{smallmatrix}
\partial_x \partial^\sharp_y e_\infty(0,0) & \partial_x \partial^\sharp_y e_\infty(0,z) \\
\partial_x \partial^\sharp_y e_\infty(z,0) & \partial_x \partial^\sharp_ye_\infty(z,z)
\end{smallmatrix}\right) - \left(\begin{smallmatrix}
\partial_x e_\infty(0,0) & \partial_x e_\infty(0,z)\\
\partial_x e_\infty(z,0) & \partial_x e_\infty(z,z)
\end{smallmatrix}\right)\hspace{-2pt}
\left(\begin{smallmatrix}
e_\infty(0,0) & e_\infty(0,z) \\ e_\infty(z,0) & e_\infty(z,z)
\end{smallmatrix}\right)^{\hspace{-2pt}-1}\hspace{-2pt}
\left(\begin{smallmatrix}
\partial^\sharp_y e_\infty(0,0) & \partial^\sharp_y e_\infty(0,z) \\
\partial^\sharp_y e_\infty(z,0) & \partial^\sharp_y e_\infty(z,z)
\end{smallmatrix}\right).
\end{equation*}
By Equations~\eqref{eq expression Omega} and~\eqref{eq Omega prime}, for all $z \in \R^n \setminus \{0\}$, we have $\Lambda(z) = \Lambda'(z) \otimes I_r$, where:
\begin{equation}
\label{eq expression Lambda prime}
\Lambda'(z) = \begin{pmatrix}
I_n - \frac{e^{-\Norm{z}^2}}{1-e^{-\Norm{z}^2}} z^* \otimes z & e^{-\frac{1}{2}\Norm{z}^2}\left(I_n - \frac{1}{1-e^{-\Norm{z}^2}}z^* \otimes z\right)\\
e^{-\frac{1}{2}\Norm{z}^2}\left(I_n - \frac{1}{1-e^{-\Norm{z}^2}}z^* \otimes z\right) & I_n - \frac{e^{-\Norm{z}^2}}{1-e^{-\Norm{z}^2}} z^* \otimes z
\end{pmatrix}.
\end{equation}

As in the previous section, let us denote by $\left(\deron{}{x_1},\dots,\deron{}{x_n}\right)$ an orthonormal basis of $\R^n$ such that $z = \Norm{z}\deron{}{x_1}$ and let $(dx_1,\dots,dx_n)$ denote its dual basis. Let $(e_1,e_2)$ denote the canonical basis of $\R^2$, we define $\mathcal{B}'_z$ to be the following orthonormal basis of $\R^2 \otimes(\R^n)^* \simeq (\R^n)^* \oplus (\R^n)^*$:
\begin{equation*}
\label{eq def B prime z}
\mathcal{B}'_z = (e_1 \otimes dx_1,e_2 \otimes dx_1, \dots, e_1 \otimes dx_n, e_2 \otimes dx_n).
\end{equation*}

\begin{lem}
\label{lem conditional variance operator 1-jets}
For any $z \in \R^n \setminus \{0\}$, the matrix of $\Lambda'(z)$ in the basis $\mathcal{B}'_z$ is:
\begin{equation*}
\left(\begin{array}{c|c}
\tilde{\Lambda}(\Norm{z}^2) & 0 \\
\hline
\rule{0pt}{14pt}
0 & \left(\begin{smallmatrix}
1 & e^{-\frac{1}{2}\Norm{z}^2} \\ e^{-\frac{1}{2}\Norm{z}^2} & 1
\end{smallmatrix}\right) \otimes I_{n-1}
\end{array}\right),
\end{equation*}
where, for all $t >0$, we set:
\begin{equation}
\label{eq def Lambda tilde}
\tilde{\Lambda}(t) = \begin{pmatrix}
1 - \frac{te^{-t}}{1-e^{-t}} & e^{-\frac{1}{2}t}\left(1 - \frac{t}{1-e^{-t}}\right)\\
e^{-\frac{1}{2}t}\left(1 - \frac{t}{1-e^{-t}}\right) & 1 - \frac{te^{-t}}{1-e^{-t}}
\end{pmatrix}.
\end{equation}
\end{lem}

\begin{proof}
Since $z = \Norm{z} \deron{}{x_1}$, we have $z^* \otimes z = \Norm{z}^2 dx_1 \otimes \deron{}{x_1}$. Hence, the matrix of $z^*\otimes z$ in $(dx_1,\dots,dx_n)$ is:
\begin{equation*}
\left(\begin{array}{c|c}
\Norm{z}^2 & 0
\\
\hline
0 & 0
\end{array}\right).
\end{equation*}
Then the conclusion follows from Eq.~\eqref{eq expression Lambda prime}.
\end{proof}

\begin{rem}
\label{rem extension of Lambda}
We can extend continuously $\tilde{\Lambda}$ at $t=0$ by setting $\tilde{\Lambda}(0) = 0$. This yields continuous extensions of $\Lambda'$ and $\Lambda$ at $z=0$. Note that $\Lambda(0)$ is not the variance operator of $(d_0s,d_0s)$ given that $s(0)=0$.
\end{rem}

\begin{dfns}
\label{def u}
Let $u_1$ and $u_2$ denote the following functions from $\R$ to $\R$:
\begin{align*}
u_1 : t \longmapsto & \frac{1 - e^{-t} + te^{-\frac{1}{2}t}}{1+e^{-\frac{1}{2}t}} & u_2 : t \longmapsto & \left\{\begin{aligned}
&\frac{1 - e^{-t} - te^{-\frac{1}{2}t}}{1-e^{-\frac{1}{2}t}} & &\text{if }t \neq 0,\\
&0 & &\text{if }t=0.
\end{aligned}\right.
\end{align*}
\end{dfns}

Once again, we will need an explicit diagonalization of $\Lambda(z)$. Such a diagonalization is given by the following lemma, once we tensor each factor by $I_r$.

\begin{lem}
\label{lem diagonalization Lambda prime}
Let $z \in \R^n$, identifying $\Lambda'(z)$ with its matrix in $\mathcal{B}'_z$, we have:
\begin{equation*}
\left(Q \otimes I_n\right) \Lambda'(z) \left(Q \otimes I_n\right)^{-1} = 
\left(\begin{array}{c|c}
\begin{smallmatrix}
u_1(\Norm{z}^2) & 0\\
0 & u_2(\Norm{z}^2)\rule[-2pt]{0pt}{1pt}
\end{smallmatrix} & 0 \\
\hline
\rule{0pt}{16pt}
0 & \left(\begin{smallmatrix}
1 - e^{-\frac{1}{2}\Norm{z}^2} & 0 \\ 0 & 1 + e^{-\frac{1}{2}\Norm{z}^2}
\end{smallmatrix}\right) \otimes I_{n-1}
\end{array}\right).
\end{equation*}
\end{lem}

\begin{proof}
By Lem.~\ref{lem conditional variance operator 1-jets}, we only need to check that, for all $t \geq 0$,
\begin{equation*}
Q \tilde{\Lambda}(t) Q^\text{t} = \begin{pmatrix} u_1(t) & 0 \\ 0 & u_2(t)\end{pmatrix}.\qedhere
\end{equation*}
\end{proof}

\begin{lem}
\label{lem non-degeneracy Lambda}
For all $z \in \R^n \setminus \{0\}$, we have $\det\left(\Lambda(z)\right) >0$, i.e.~the distribution of $(d_0s,d_zs)$ given that $s(0)=0=s(z)$ is non-degenerate.
\end{lem}

\begin{proof}
See Appendix~\ref{sec technical computations for Section limit distrib}.
\end{proof}

By Lem.~\ref{lem non-degeneracy Omega}, $\Omega(z)$ is a positive self-adjoint operator on $\left(\R \oplus \R \oplus (\R^n)^* \oplus (\R^n)^*\right) \otimes \R^r$, for all $z \in \R^n \setminus \{0\}$, and so is its inverse. Hence, $\Omega(z)^{-1}$ admits a unique positive square root, that we denote by $\Omega(z)^{-\frac{1}{2}}$. Similarly, by Lem.~\ref{lem non-degeneracy Lambda}, $\Lambda(z)$ is a positive self-adjoint operator on $\left((\R^n)^* \oplus (\R^n)^*\right) \otimes \R^r$ and we denote by $\Lambda(z)^\frac{1}{2}$ its positive square root.

\begin{lem}
\label{lem boundedness sqrt Lambda sqrt Omega}
The map $z \longmapsto \begin{pmatrix}
0 & \Lambda(z)^\frac{1}{2}
\end{pmatrix}\Omega(z)^{-\frac{1}{2}}$ is bounded on $\R^n \setminus \{0\}$.
\end{lem}

\begin{proof}
See Appendix~\ref{sec technical computations for Section limit distrib}.
\end{proof}


\subsection{Finiteness of the leading constant}
\label{subsec finiteness of the leading constant}

The goal of this section, it to prove that the constant $\mathcal{I}_{n,r}$ defined by Eq.~\eqref{eq def Inr} and appearing in Thm.~\ref{thm asymptotics variance} is well-defined and finite.

\begin{dfn}
\label{def Dnr}
Let $n \in \N^*$ and $r\in \{1,\dots,n\}$, for every $t >0$ we set:
\begin{equation*}
\label{eq def Dnr}
D_{n,r}(t) = \frac{\esp{\odet{X(t)}\odet{Y(t)}}}{\left(1-e^{-t}\right)^\frac{r}{2}} - (2\pi)^r \left(\frac{\vol{\S^{n-r}}}{\vol{\S^n}}\right)^2,
\end{equation*}
where $(X(t),Y(t))$ is the centered Gaussian vector in $\mathcal{M}_{rn}(\R)\times\mathcal{M}_{rn}(\R)$ defined in Def.~\ref{def XYt}.
\end{dfn}

By the definition of $\mathcal{I}_{n,r}$ (see Eq.~\eqref{eq def Inr}), we have:
\begin{equation*}
\mathcal{I}_{n,r} = \frac{1}{2} \int_0^{+\infty} D_{n,r}(t)t^\frac{n-2}{2} \dx t.
\end{equation*}
Hence, we have to prove that $t \mapsto D_{n,r}(t)t^\frac{n-2}{2}$ is integrable on $(0,+\infty)$, which boils down to computing the asymptotic expansions of $\esp{\odet{X(t)}\odet{Y(t)}}$ as $t \to 0$ and as $t \to +\infty$.

Let us now relate this to the Bargmann--Fock process $(s(z))_{z \in \R^n}$.
\begin{lem}
\label{lem relation Bargmann--Fock XY}
Let $z \in \R^n \setminus\{0\}$. Let $\left(\deron{}{x_1},\dots,\deron{}{x_n}\right)$ be an orthonormal basis of $\R^n$ such that $z = \Norm{z}\deron{}{x_1}$ and let $(\zeta_1,\dots,\zeta_r)$ be any orthonormal basis of $\R^r$. Then, the matrices of $d_0s$ and $d_zs$ in these bases, given that $s(0)=0=s(z)$, form a random vector in $\mathcal{M}_{rn}(\R) \times \mathcal{M}_{rn}(\R)$ which is distributed as $\left(X(\Norm{z}^2),Y(\Norm{z}^2)\right)$.
\end{lem}

\begin{proof}
Let us denote by $\tilde{X}(z)$ and $\tilde{Y}(z)$ the matrices of $d_0s$ and $d_zs$ in the bases $\left(\deron{}{x_1},\dots,\deron{}{x_n}\right)$ and $(\zeta_1,\dots,\zeta_r)$, given that $s(0)=0=s(z)$. We denote by $\tilde{X}_{ij}(z)$ (resp.~$\tilde{Y}_{ij}(z)$) the coefficients of $\tilde{X}$ (resp.~$\tilde{Y}$) for $i \in \{1,\dots,r\}$ and $j \in \{1,\dots,n\}$. By Lem.~\ref{lem conditional variance operator 1-jets}, the couples $\left(\tilde{X}_{ij},\tilde{Y}_{ij}\right)$ are centered Gaussian vectors in $\R^2$ which are independant from one another. Moreover, the variance matrix of $\left(\tilde{X}_{ij}(z),\tilde{Y}_{ij}(z)\right)$ equals $\tilde{\Lambda}(\Norm{z}^2)$ if $j=1$ and $\left(\begin{smallmatrix}
1 & e^{-\frac{1}{2}\Norm{z}^2} \\ e^{-\frac{1}{2}\Norm{z}^2} & 1
\end{smallmatrix}\right)$ otherwise. By Def.~\ref{def XYt}, this is exactly saying that $\left(\tilde{X}(z),\tilde{Y}(z)\right)$ is distributed as $\left(X(\Norm{z}^2),Y(\Norm{z}^2)\right)$.
\end{proof}

\begin{lem}
\label{lem estimates 0 expectation of odet XY}
Let $n \in \N^*$ and $r \in \{1,\dots,n\}$. Then, as $t \to 0$, we have the following:
\begin{equation*}
\esp{\odet{X(t)}\odet{Y(t)}} \sim \left\{ \begin{matrix}
\dfrac{(n-1)!}{(n-r-1)!} & \text{if } r<n,\\
\dfrac{n!}{2}t & \text{if } r=n.
\end{matrix}\right.
\end{equation*}
\end{lem}

\begin{proof}
See Appendix~\ref{sec technical computations for Section limit distrib}.
\end{proof}

\begin{lem}
\label{lem estimates infty expectation of odet XY}
For all $n \in \N^*$ and $r \in \{1,\dots,n\}$, we have the following as $t \to +\infty$:
\begin{equation*}
\esp{\odet{X(t)}\odet{Y(t)}} = (2\pi)^r \left(\frac{\vol{\S^{n-r}}}{\vol{\S^n}}\right)^2 + O\!\left(te^{-\frac{t}{2}}\right)
\end{equation*}
\end{lem}

\begin{proof}
See Appendix~\ref{sec technical computations for Section limit distrib}.
\end{proof}

Lem.~\ref{lem estimates 0 expectation of odet XY} and~\ref{lem estimates infty expectation of odet XY} and the definition of $D_{n,r}$ (Def.~\ref{def Dnr}) allows to derive the following.

\begin{cor}
\label{cor estimates Dnr and integrability}
Let $n \in \N^*$ and $r \in \{1,\dots,n\}$, then we have: 
\begin{equation*}
t^\frac{n-2}{2}D_{n,r}(t) = \left\{ \begin{aligned} &O\!\left(\frac{1}{\sqrt{t}}\right) & &\text{as }t \to 0,\\ &O\!\left(e^{-\frac{t}{4}}\right) & &\text{as }t \to +\infty. \end{aligned}\right.
\end{equation*}
In particular, $\mathcal{I}_{n,r} = \frac{1}{2} \int_0^{+\infty} D_{n,r}(t)t^\frac{n-2}{2} \dx t$ is well-defined and finite.
\end{cor}


\section{Proof of Theorem~\ref{thm asymptotics variance}}
\label{sec proof of the main theorem}

This section is concerned with the proof of our main result (Thm.~\ref{thm asymptotics variance}).  Recall that $\X$ is a compact Kähler manifold of complex dimension $n\geq 1$ defined over the reals and that $M$ denotes its real locus, assumed to be non-empty. Let $\E\to \X$ be a rank $r \in\{1,\dots,n\}$ real Hermitian vector bundle and $\L\to \X$ be a real Hermitian line bundle whose curvature form is $\omega$, the Kähler form of~$\X$. We assume that $\E$ and $\L$ are endowed with compatible real structures. For all $d \in \N$, we still denote by $E_d$ the Bergman kernel of $\E \otimes \L^d$. Finally, let $s_d$ denote a standard Gaussian vector in $\R \H$, whose real zero set is denoted by $Z_d$, and let $\rmes{d}$ denote the Riemannian volume measure on $Z_d$.

In Sect.~\ref{subsec Kac-Rice formulas} we recall Kac--Rice formulas and use them to derive an integral expression of $\var{\rmes{d}}$. Sect.~\ref{subsec expression of some covariances} is concerned with the study of some relevant random variables related to $(s_d(x))_{x \in M}$. Finally, we conclude the proof in two steps, in Sect.~\ref{subsec far off-diagonal term} and~\ref{subsec near-diagonal term}.


\subsection{Kac--Rice formulas}
\label{subsec Kac-Rice formulas}

In this section, we use Kac--Rice formulas to derive an integral expression of $\var{\rmes{d}}$. Classical references for this material are~\cite[chap.~11.5]{AT2007} and~\cite[Thm.~6.3]{AW2009}. Since our probability space is the finite-dimensional vector space $\R\H$, it is possible to derive Kac--Rice formulas under weaker hypothesis than those given in~\cite{AT2007} and~\cite{AW2009}. This uses Federer's coaera formula and the so-called double fibration trick, see~\cite[App.~C]{Let2016} and the references therein. The first Kac--Rice formula we state (Thm.~\ref{thm Kac-Rice 1} below) was proved in~\cite[Thm.~5.3]{Let2016} and the second (Thm.~\ref{thm Kac-Rice 2} below) was proved in~\cite[Thm.~4.4]{Let2016a}.

Recall that the Jacobian $\odet{L}$ of an operator $L$ was defined in Def.~\ref{def Jacobian}, that $d_1$ was defined in Lem.~\ref{lem def d1} and that a connection is said to be real if it satisfies the condition given in Def.~\ref{def real connection}.

\begin{thm}[Kac--Rice formula 1]
\label{thm Kac-Rice 1}
Let $d \geq d_1$, let $\nabla^d$ be any real connection on $\E \otimes \L^d$ and let $s_d \sim \mathcal{N}(\Id)$ in $\R\H$. Then for all $\phi \in \mathcal{C}^0(M)$ we have:
\begin{equation}
\label{eq thm Kac-Rice 1}
\esp{\int_{x \in Z_d} \phi(x) \rmes{d}}= (2\pi)^{-\frac{r}{2}} \int_{x \in M} \frac{\phi(x)}{\odet{\ev_x^d}}\espcond{\odet{\nabla^d_{x}s_d}}{s_d(x)=0} \rmes{M}.
\end{equation}
The expectation on the right-hand side of Eq.~\eqref{eq thm Kac-Rice 1} is to be understood as the conditional expectation of $\odet{\nabla^d_{x}s_d}$ given that $s_d(x)=0$.
\end{thm}

\begin{ntn}
\label{notation Delta}
Let $\Delta = \{(x,y) \in M^2 \mid x=y\}$ denote the diagonal of $M^2$.
\end{ntn}

\begin{dfn}
\label{def eval xy}
Let $d \in \N$ and let $(x,y) \in M^2 \setminus \Delta$ we denote by $\ev^d_{x,y}$ the evaluation map:
\begin{equation*}
\label{eq dfn ev map}
\begin{array}{rccc}
\ev_{x,y}^d:& \R \H & \longrightarrow & \R\!\left(\E \otimes \L^d\right)_x \oplus \R\!\left(\E \otimes \L^d\right)_y.\\
& s & \longmapsto & \left(s(x),s(y)\right).
\end{array}
\end{equation*}
\end{dfn}

\begin{lem}
\label{lem def d2}
There exists $d_2 \in \N$, such that for all $(x,y) \in M^2 \setminus \Delta$, $\ev^d_{x,y}$ is surjective.
\end{lem}

This was proved in \cite[Prop.~4.2]{Let2016a} in the case $r<n$, using Kodaira's embedding theorem. The proof can be adapted verbatim to the case $r\leq n$. We will give an alternative proof using only estimates on the Bergman kernel (Lem.~\ref{lem almost uncorrelated Theta} and~\ref{lem asymptotic Theta d}), see p.~\pageref{proof existence d2} below. 

\begin{thm}[Kac--Rice formula 2]
\label{thm Kac-Rice 2}
Let $d \geq d_2$, let $\nabla^d$ be any real connection on $\E \otimes \L^d$ and let $s_d \sim \mathcal{N}(\Id)$ in $\R \H$. Then for all $\phi_1$ and $\phi_2 \in \mathcal{C}^0(M)$ we have:
\begin{equation}
\label{eq thm Kac-Rice var}
\begin{aligned}
\esp{\int_{(x,y) \in (Z_d)^2 \setminus \Delta} \phi_1(x)\phi_2(y) \rmes{d}^2}&= \\
\frac{1}{(2\pi)^r} \int_{(x,y) \in M^2 \setminus \Delta} \frac{\phi_1(x)\phi_2(y)}{\odet{\ev_{x,y}^d}}&
\espcond{\odet{\nabla^d_{x}s_d}\odet{\nabla^d_{y}s_d}}{\ev_{x,y}^d(s_d)=0} \rmes{M}^2.
\end{aligned}
\end{equation}
Here, $\rmes{M}^2$ (resp.~$\rmes{d}^2$) stands for the product measure on $M^2$ (resp.~$(Z_d)^2$) induced by $\rmes{M}$ (resp.~$\rmes{d}$). The expectation on the right-hand side of Eq.~\eqref{eq thm Kac-Rice var} is the conditional expectation of $\odet{\nabla^d_{x}s_d}\odet{\nabla^d_{y}s_d}$ given that $\ev_{x,y}^d(s_d)=0$.
\end{thm}

\begin{proof}
This formula was proved in~\cite[Thm.~4.4]{Let2016a} in the case $1 \leq r<n$. The hypothesis $r<n$ does not appear in the proof and can be changed to $r \leq n$ without any other modification.
\end{proof}

\begin{dfn}
\label{def Dd}
Let $d \geq \max(d_1,d_2)$, and let $\nabla^d$ be any real connection on $\E \otimes \L^d$. We denote by $\D_d : M^2 \setminus \Delta \to \R$ the map defined by:
\begin{multline*}
\label{eq def density}
\D_d(x,y) = \left(\frac{\espcond{\odet{\nabla^d_{x}s_d}\!\odet{\nabla^d_{y}s_d}}{s_d(x)=0=s_d(y)}}{\odet{\ev_{x,y}^d}}\right.\\
\left.-\frac{\espcond{\odet{\nabla^d_{x}s_d}}{s_d(x)=0}\espcond{\odet{\nabla^d_{y}s_d}}{s_d(y)=0}}{\odet{\ev_x^d}\odet{\ev_y^d}}\right).
\end{multline*}
\end{dfn}

\begin{rem}
\label{rem choice of nablad}
Note that $\D_d$ does not depend on the choice of $\nabla^d$. Indeed, we only consider derivatives of $s_d$ at points where it vanishes.
\end{rem}

\begin{prop}
\label{prop integral expression variance}
For all $d \geq \max(d_1,d_2)$, we have for any $\phi_1, \phi_2 \in \mathcal{C}^0(M)$:
\begin{equation*}
\var{\rmes{d}}\left(\phi_1,\phi_2\right) = \frac{1}{(2\pi)^r} \int_{M^2 \setminus \Delta}\phi_1(x)\phi_2(y) \D_d(x,y) \rmes{M}^2 + \delta_{rn}\ \esp{\prsc{\rmes{d}}{\phi_1\phi_2}},
\end{equation*}
where $\delta_{rn}$ equals $1$ if $r=n$ and $0$ otherwise.
\end{prop}

\begin{proof}
This was proved in \cite[Sect.~4.2]{Let2016a} for $r<n$, the case $r=n$ requires an extra argument. The following proof if valid for any $r \in \{1,\dots,n\}$. Let $\phi_1$ and $\phi_2 \in \mathcal{C}^0(M)$, we have:
\begin{equation}
\label{eq integral expression variance}
\var{\rmes{d}}\left(\phi_1,\phi_2\right) = \esp{\prsc{\rmes{d}}{\phi_1}\prsc{\rmes{d}}{\phi_2}} - \esp{\prsc{\rmes{d}}{\phi_1}}\esp{\prsc{\rmes{d}}{\phi_2}}.
\end{equation}
Since $Z_d$ has almost surely dimension $n-r$, the diagonal in $Z_d \times Z_d$ is negligible if and only if $r<n$. Moreover, if $r=n$ then both $\rmes{d}$ and $\rmes{d}^2$ are counting measures. Then,
\begin{align*}
\prsc{\rmes{d}}{\phi_1}\prsc{\rmes{d}}{\phi_2} &= \int_{(x,y)\in (Z_d)^2} \phi_1(x)\phi_2(y) \rmes{d}^2\\
&= \int_{(x,y)\in (Z_d)^2 \setminus \Delta} \phi_1(x)\phi_2(y) \rmes{d}^2 + \delta_{rn} \int_{x \in Z_d}\phi_1(x)\phi_2(x) \rmes{d},
\end{align*}
almost surely. Hence
\begin{equation}
\label{eq integral expression}
\esp{\prsc{\rmes{d}}{\phi_1}\prsc{\rmes{d}}{\phi_2}} = \esp{\int_{(x,y)\in (Z_d)^2 \setminus \Delta} \phi_1(x)\phi_2(y) \rmes{d}^2} + \delta_{rn} \esp{\prsc{\rmes{d}}{\phi_1\phi_2}}.
\end{equation}
We apply Thm.~\ref{thm Kac-Rice 2} to the first term on the right-hand side of Eq.~\eqref{eq integral expression}. Similarly, we apply Thm.~\ref{thm Kac-Rice 1} to $\esp{\prsc{\rmes{d}}{\phi_i}}$ for $i\in \{1,2\}$. This yields the result by Eq.~\eqref{eq integral expression variance}.
\end{proof}

By Thm.~\ref{thm expectation}, if $r=n$, for all $\phi_1,\phi_2 \in \mathcal{C}^0(M)$ we have:
\begin{equation*}
\esp{\prsc{\rmes{d}}{\phi_1\phi_2}} = d^\frac{n}{2} \frac{2}{\vol{\S^n}}\left(\int_M \phi_1\phi_2 \rmes{M}\right) + \Norm{\phi_1}_\infty \Norm{\phi_2}_\infty O\!\left(d^{\frac{n}{2}-1}\right).
\end{equation*}
Hence, in order to prove Thm.~\ref{thm asymptotics variance}, we have to show that, for any $n \in \N^*$ and $r \in \{1,\dots,n\}$:
\begin{multline}
\label{eq main estimate}
\int_{M^2 \setminus \Delta}\phi_1(x)\phi_2(y) \D_d(x,y) \rmes{M}^2 = d^{r-\frac{n}{2}} \left(\int_M \phi_1 \phi_2 \rmes{M}\right) \vol{\S^{n-1}}\mathcal{I}_{n,r}\\
+ \Norm{\phi_1}_{\infty}\Norm{\phi_2}_{\infty}  O\!\left(d^{r-\frac{n}{2}-\alpha}\right) + \Norm{\phi_1}_{\infty}\varpi_{\phi_2}\left(C_\beta d^{-\beta}\right) O\!\left(d^{r-\frac{n}{2}}\right),
\end{multline}
where $\alpha$, $\beta$, $C_\beta$ and $\mathcal{I}_{n,r}$ are as in Thm.~\ref{thm asymptotics variance}.

This is done in two steps. The mass of the integral on the left-hand side of Eq.~\eqref{eq main estimate} concentrates in a neighborhood of $\Delta$ of typical size $\frac{1}{\sqrt{d}}$. More precisely, let us now fix the value of the constant $b_n$ appearing in Prop.~\ref{prop near diag estimates}.

\begin{dfns}
\label{def bn and Delta d}
We set $b_n = \frac{1}{C_2}\left(\frac{n}{2}+1\right)$, where $C_2>0$ is the constant appearing in the exponential in Thm.~\ref{thm off diag estimates}. Moreover, for all $d \in \N^*$, we denote:
\begin{equation*}
\Delta_d = \left\{(x,y) \in M^2 \mvert \rho_g(x,y) < b_n\frac{\ln d}{\sqrt{d}} \right\},
\end{equation*}
where, $\rho_g$ is the geodesic distance in $(M,g)$.
\end{dfns}

In Sect.~\ref{subsec far off-diagonal term} below, we show that, in Eq.~\eqref{eq main estimate}, the integral over $M^2 \setminus \Delta_d$ only contributes what turns out to be an error term. We refer to this term as the \emph{far off-diagonal term}. In Sect.~\ref{subsec near-diagonal term} we complete the proof of~\eqref{eq main estimate} by studying the \emph{near-diagonal term}, that is the integral of $\phi_1(x)\phi_2(y)\D_d(x,y)$ over $\Delta_d \setminus \Delta$. This turns out to be the leading term.


\subsection{Expression of some covariances}
\label{subsec expression of some covariances}

In order to prove~\eqref{eq main estimate}, we need to study the distribution of the random variables appearing in the definition of $\D_d$ (see Def.~\ref{def Dd}). The purpose of this section is to introduce several variance operators that will appear in the proof. In the following $\nabla^d$ denotes a real connection on $\E \otimes \L^d \to \X$.


\subsubsection{Uncorrelated terms}
\label{subsubsec uncorrelated terms}

First of all, let us consider the distribution of $s_d(x)$ for any $x \in M$. Since $s_d \sim \mathcal{N}(\Id)$ and $\ev^d_x$ is linear (see Def.~\ref{def 0 ample}), $s_d(x)$ is a centered Gaussian vector in $\R \left(\E \otimes \L^d\right)_x$ with variance operator $\ev^d_x(\ev^d_x)^*=E_d(x,x)$.

\begin{lem}
\label{lem asymptotic det Ed}
For all $x \in M$, we have $\odet{\ev_x^d} = \left(\dfrac{d}{\pi}\right)^\frac{rn}{2}(1+O(d^{-1}))$, where the error term is independent of $x$.
\end{lem}

\begin{proof}
We have $\odet{\ev_x^d} = \det(E_d(x,x))^\frac{1}{2}$, and by Thm.~\ref{thm Dai Liu Ma} $E_d(x,x) = \left(\dfrac{d}{n}\right)^{n} \left(\Id +O(d^{-1})\right)$.
\end{proof}

\begin{cor}
\label{cor def d1}
There exists $d_1 \in \N$, such that for all $d \geq d_1$, for all $x \in M$, $\ev_x^d$ is surjective, that is  $(s_d(x))$ is non-degenerate.
\end{cor}

Then, let $d \in \N$ and $x \in M$, we denote by $j_x^d:s \mapsto (s(x),\nabla^d_xs)$ the evaluation of the $1$-jet of a section at the point $x$. The distribution of the random vector $(s_d(x),\nabla^d_xs_d)$ is a centered Gaussian in $\R \left(\E \otimes \L^d\right)_x  \oplus \left(\R \left(\E \otimes \L^d\right)_x \otimes T_x^*M\right)$ with variance operator:
\begin{align*}
\label{eq variance 1-jet}
j_x^d(j_x^d)^* = \esp{j_x^d(s_d)\otimes j_x^d(s_d)^*} &= \begin{pmatrix}
\esp{s_d(x) \otimes s_d(x)^*} & \esp{s_d(x) \otimes (\nabla^d_xs_d)^*}\\
\esp{(\nabla^d_xs_d) \otimes s_d(x)^*} & \esp{(\nabla^d_xs_d) \otimes (\nabla^d_xs_d)^*}
\end{pmatrix}\\
&= \begin{pmatrix}
E_d(x,x) & \partial_y^\sharp E_d(x,x)\\
\partial_xE_d(x,x) & \partial_x\partial_y^\sharp E_d(x,x)
\end{pmatrix}.
\end{align*}
If $d \geq d_1$, then $s_d(x)$ is non-degenerate and the distribution of $\nabla^d_xs_d$ given that $s_d(x)=0$ is a centered Gaussian whose variance equals
\begin{equation*}
\partial_x\partial_y^\sharp E_d(x,x)- \partial_xE_d(x,x) \left(E_d(x,x)\right)^{-1}\partial_y^\sharp E_d(x,x).
\end{equation*}
By Thm.~\ref{thm Dai Liu Ma}, this variance equals
\begin{equation*}
\frac{d^{n+1}}{\pi^n}\left(\Id_{\R \left(\E \otimes \L^d\right)_x \otimes T_x^*M} + O(d^{-1})\right)
\end{equation*}
as $d$ goes to infinity and the error does not depend on $x$.

\begin{rem}
\label{rem choice of connection}
If $(s,x)$ is such that $s(x)=0$, then $\nabla^d_xs$ does not depend on the connection $\nabla^d$. This explains why the distribution of $\nabla^d_xs_d$ given that $s_d(x)=0$, in particular its variance, does not depend on $\nabla^d$.
\end{rem}

\begin{lem}
\label{lem estimates cond exp x}
For every $x \in M$, we have:
\begin{equation*}
\espcond{\odet{\nabla^d_xs_d}}{s_d(x)=0} = \left(\frac{d^{n+1}}{\pi^n}\right)^\frac{r}{2} (2\pi)^\frac{r}{2} \frac{\vol{\S^{n-r}}}{\vol{\S^n}} \left(1 + O\!\left(d^{-1}\right)\right),
\end{equation*}
where the error term is independent of $x$.
\end{lem}

\begin{proof}
This was proved in~\cite[Lem.~4.7]{Let2016a} for $r<n$. The proof is the same here.
\end{proof}


\subsubsection{Correlated terms far from the diagonal}
\label{subsubsec correlated terms far from the diagonal}

Let us now focus on variables where non trivial correlations may appear in the limit. Let $d \in \N$, for all $(x,y)\in M^2 \setminus \Delta$, the random vector $\ev^d_{x,y}(s_d)=\left(s_d(x),s_d(y)\right)$ is a centered Gaussian vector with variance operator:
\begin{equation}
\label{eq variance evxy}
\ev^d_{x,y}(\ev^d_{x,y})^* = \esp{\ev^d_{x,y}(s_d) \otimes \ev^d_{x,y}(s_d)^*} = \begin{pmatrix}
E_d(x,x) & E_d(x,y)\\ E_d(y,x) & E_d(y,y)
\end{pmatrix},
\end{equation}
where we decomposed this operator according to the direct sum $\R\!\left(\E \otimes \L^d\right)_x \oplus \R\!\left(\E \otimes \L^d\right)_y$.

\begin{dfn}
\label{def Theta d xy}
For all $d \in \N$, for all $(x,y) \in M^2 \setminus \Delta$, we denote by:
\begin{equation*}
\Theta_d(x,y) = \left(\frac{\pi}{d}\right)^n \begin{pmatrix}
E_d(x,x) & E_d(x,y)\\ E_d(y,x) & E_d(y,y)
\end{pmatrix},
\end{equation*}
the variance of the centered Gaussian vector $\left(\frac{\pi}{d}\right)^\frac{n}{2}(s_d(x),s_d(y))$.
\end{dfn}

Note that, by Lem.~\ref{lem def d2}, for all $d \geq d_2$, $\ev_{x,y}^d(\ev_{x,y}^d)^*$ is non-singular, i.e.~$\left(s_d(x),s_d(y)\right)$ is non-degenerate and $\Theta_d(x,y)$ is non-singular.

Let $d \in \N$ and $(x,y) \in M^2 \setminus \Delta$, we denote by $j_{x,y}^d:s \mapsto \left(s(x),s(y),\nabla^d_xs,\nabla^d_ys\right)$ the evaluation of the $1$-jets at $(x,y)$. Then $j_{x,y}^d(s_d) = \left(s_d(x),s_d(y),\nabla^d_xs_d,\nabla^d_ys_d\right)$ is a centered Gaussian vector in:
\begin{equation*}
\label{eq target space 1jet}
\R\!\left(\E \otimes \L^d\right)_x \oplus \R\!\left(\E \otimes \L^d\right)_y \oplus \left(\R\!\left(\E \otimes \L^d\right)_x \otimes T^*_xM\right) \oplus \left(\R\!\left(\E \otimes \L^d\right)_y \otimes T^*_yM\right),
\end{equation*}
whose variance operator $j_{x,y}^d(j_{x,y}^d)^*$ equals:
\begin{equation*}
\esp{j_{x,y}^d(s_d) \otimes \left(j_{x,y}^d(s_d)\right)^*} = \begin{pmatrix}
E_d(x,x) & E_d(x,y) & \partial_y^\sharp E_d(x,x) & \partial_y^\sharp E_d(x,y)\\
E_d(y,x) & E_d(y,y) & \partial_y^\sharp E_d(y,x) & \partial_y^\sharp E_d(y,y)\\
\partial_xE_d(x,x) & \partial_xE_d(x,y) & \partial_x\partial_y^\sharp E_d(x,x) & \partial_x\partial_y^\sharp E_d(x,y)\\
\partial_xE_d(y,x) & \partial_xE_d(y,y) & \partial_x\partial_y^\sharp E_d(y,x) & \partial_x\partial_y^\sharp E_d(y,y)
\end{pmatrix}.
\end{equation*}

\begin{dfn}
\label{def Omega d xy}
For all $d \in \N$, for all $(x,y) \in M^2 \setminus \Delta$, we denote by:
\begin{equation*}
\Omega_d(x,y) = \left(\frac{\pi}{d}\right)^n \begin{pmatrix}
E_d(x,x) & E_d(x,y) & d^{-\frac{1}{2}}\partial_y^\sharp E_d(x,x) & d^{-\frac{1}{2}}\partial_y^\sharp E_d(x,y)\\
E_d(y,x) & E_d(y,y) & d^{-\frac{1}{2}}\partial_y^\sharp E_d(y,x) & d^{-\frac{1}{2}}\partial_y^\sharp E_d(y,y)\\
d^{-\frac{1}{2}}\partial_xE_d(x,x) & d^{-\frac{1}{2}}\partial_xE_d(x,y) & d^{-1}\partial_x\partial_y^\sharp E_d(x,x) & d^{-1}\partial_x\partial_y^\sharp E_d(x,y)\\
d^{-\frac{1}{2}}\partial_xE_d(y,x) & d^{-\frac{1}{2}}\partial_xE_d(y,y) & d^{-1}\partial_x\partial_y^\sharp E_d(y,x) & d^{-1}\partial_x\partial_y^\sharp E_d(y,y)
\end{pmatrix},
\end{equation*}
the variance operator of the centered Gaussian vector: $\left(\frac{\pi}{d}\right)^\frac{n}{2}\left(s_d(x),s_d(y),\frac{1}{\sqrt{d}}\nabla^d_xs_d,\frac{1}{\sqrt{d}}\nabla^d_ys_d\right)$.
\end{dfn}

Let us now assume that $d \geq d_2$, so that the distribution of $\left(s_d(x),s_d(y)\right)$ is non-degenerate. Then the distribution of $\left(\nabla_x^ds,\nabla_y^ds\right)$ given that $s_d(x)=0=s_d(y)$ is a centered Gaussian with variance operator:
\begin{equation*}
\label{eq conditional variance full}
\left(\begin{smallmatrix}
\partial_x\partial_y^\sharp E_d(x,x) & \partial_x\partial_y^\sharp E_d(x,y)\\
\partial_x\partial_y^\sharp E_d(y,x) & \partial_x\partial_y^\sharp E_d(y,y)
\end{smallmatrix}\right)-\left(\begin{smallmatrix}
\partial_xE_d(x,x) & \partial_xE_d(x,y)\\
\partial_xE_d(y,x) & \partial_xE_d(y,y)
\end{smallmatrix}\right)
\left(\begin{smallmatrix}
E_d(x,x) & E_d(x,y)\\
E_d(y,x) & E_d(y,y)
\end{smallmatrix}\right)^{-1}
\left(\begin{smallmatrix}
\partial_y^\sharp E_d(x,x) & \partial_y^\sharp E_d(x,y)\\
\partial_y^\sharp E_d(y,x) & \partial_y^\sharp E_d(y,y)
\end{smallmatrix}\right).
\end{equation*}

\begin{dfn}
\label{def Lambda d xy}
For all $d \geq d_2$, for all $(x,y) \in M^2 \setminus \Delta$, we set:
\begin{align*}
\Lambda_d(x,y) = & \frac{\pi^n}{d^{n+1}}\left(\begin{pmatrix}
\partial_x\partial_y^\sharp E_d(x,x) & \partial_x\partial_y^\sharp E_d(x,y)\\
\partial_x\partial_y^\sharp E_d(y,x) & \partial_x\partial_y^\sharp E_d(y,y)
\end{pmatrix}\right.\\
&-\left.\begin{pmatrix}
\partial_xE_d(x,x) & \partial_xE_d(x,y)\\
\partial_xE_d(y,x) & \partial_xE_d(y,y)
\end{pmatrix}
\begin{pmatrix}
E_d(x,x) & E_d(x,y)\\
E_d(y,x) & E_d(y,y)
\end{pmatrix}^{-1}
\begin{pmatrix}
\partial_y^\sharp E_d(x,x) & \partial_y^\sharp E_d(x,y)\\
\partial_y^\sharp E_d(y,x) & \partial_y^\sharp E_d(y,y)
\end{pmatrix}\right),
\end{align*}
which is the variance of the Gaussian vector $\left(\dfrac{\pi^n}{d^{n+1}}\right)^\frac{1}{2}\left(\nabla^d_xs_d,\nabla^d_ys_d\right)$ given that $s_d(x)=0=s_d(y)$.
\end{dfn}

\begin{rem}
\label{rem choice of connection 2}
Once again, $\Lambda_d(x,y)$ is independent of the choice of $\nabla^d$, and so is the distribution of $(\nabla^d_xs_d,\nabla^d_ys_d)$ given that $s_d(x)=0=s_d(y)$. On the other hand, the distribution of $\left(s_d(x),s_d(y),\nabla^d_xs_d,\nabla^d_ys_d\right)$ heavily depends on $\nabla^d$, and so does $\Omega_d(x,y)$. Hence, we will need to specify a choice of $\nabla^d$ at some point when dealing with $\Omega_d$.
\end{rem}


\subsubsection{Correlated terms close to the diagonal}
\label{subsubsec correlated terms close to the diagonal}

Finally, we need to consider the distribution of the $1$-jets of $s_d$ at $x$ and $y \in M$, when the distance between $x$ and $y$ is of order $\frac{1}{\sqrt{d}}$. As in Sect.~\ref{sec estimates for the bergman kernel}, let $R>0$ be such that $2R$ is less than the injectivity radius of $\X$. There exists $d_3 \in \N$ such that, for all $d \geq d_3$, $b_n \frac{\ln d}{\sqrt{d}} \leq R$.

Let $d \geq d_3$ and let $(x,y) \in \Delta_d \setminus \Delta$. Using the real normal trivialization of $\E \otimes \L^d$ around~$x$ (see Sect.~\ref{subsec real normal trivialization}), we can see $\left(\frac{\pi}{d}\right)^\frac{n}{2}(s_d(x),s_d(y))$ as a random vector in $\R\!\left(\E \otimes \L^d\right)_x \oplus \R\!\left(\E \otimes \L^d\right)_x$. Since the distance from $x$ to $y$ is smaller than the injectivity radius of $M$, we can write $y$ as $\exp_x\left(\frac{z}{\sqrt{d}}\right)$ for some $z \in T_xM$. Moreover $\Norm{z} = \sqrt{d} \, \rho_g(x,y) < b_n \ln d$.

\begin{dfn}
\label{def Theta dz}
Let $d \geq d_3$, let $x \in M$ and let $z \in B_{T_xM}(0,b_n\ln d) \setminus\{0\}$, we set:
\begin{equation*}
\Theta_d(z) = \Theta_d\left(x,\exp_x\left(\frac{z}{\sqrt{d}}\right)\right),
\end{equation*}
seen as an operator on $\R\!\left(\E \otimes \L^d\right)_x \oplus \R\!\left(\E \otimes \L^d\right)_x$ via the real normal trivialization centered at $x$.
\end{dfn}

\begin{rem}
\label{rem bad notation}
Beware that $\Theta_d(z)$ depends on $x$, even if this is not reflected in the notation. However, we will show that the limit of $\Theta_d(z)$ as $d \to +\infty$ does not depend on $x$.
\end{rem}

Recall that $e_d$ was defined by Eq.~\eqref{eq def ed} as a map from $T_xM\times T_xM$ to $\End\left(\R\!\left(\E\otimes \L^d\right)_x\right)$. The definitions of $\Theta_d(x,y)$ (Def.~\ref{def Theta d xy}) and $e_d$ show that, for all $d \geq d_3$, for all $x \in M$ and for all $z \in B_{T_xM}(0,b_n\ln d)\setminus\{0\}$:
\begin{equation}
\label{eq expression Theta dz}
\Theta_d(z) = \begin{pmatrix}
e_d(0,0) & e_d(0,z) \\ e_d(z,0) & e_d(z,z)
\end{pmatrix}.
\end{equation}

We can define $\Omega_d(z)$ and $\Lambda_d(z)$ similarly and express them in terms of $e_d$ and its derivatives.

\begin{dfn}
\label{def Omega dz}
Let $d \geq d_3$, let $x \in M$ and let $z \in B_{T_xM}(0,b_n\ln d)\setminus\{0\}$, we set:
\begin{equation*}
\Omega_d(z) = \Omega_d\left(x,\exp_x\left(\frac{z}{\sqrt{d}}\right)\right),
\end{equation*}
seen as an operator on $\left(\R \oplus \R \oplus T_x^*M \oplus T_x^*M\right) \otimes \R\!\left(\E \otimes \L^d\right)_x$ via the real normal trivialization centered at $x$.
\end{dfn}

Let $\nabla^d$ be a real connection on $\E \otimes \L^d$ such that, in the real normal trivialization around $x$, this connection coincides over the ball $B_{T_x\X}(0,R)$ with the usual differentiation for maps from $T_x\X$ to $\left(\E \otimes \L^d\right)_x$. The existence of such a connection was established at the end of Sect.~\ref{subsec real normal trivialization}. Then, by Def.~\ref{def Omega d xy} and~\ref{def Omega dz}, we have for all $d\geq d_3$, for all $x \in M$ and for all $z \in B_{T_xM}(0,b_n \ln d)\setminus\{0\}$:
\begin{equation}
\label{eq expression Omega dz}
\Omega_d(z) = \begin{pmatrix}
e_d(0,0) & e_d(0,z) & \partial_y^\sharp e_d(0,0) & \partial_y^\sharp e_d(0,z)\\
e_d(z,0) & e_d(z,z) & \partial_y^\sharp e_d(z,0) & \partial_y^\sharp e_d(z,z)\\
\partial_xe_d(0,0) & \partial_xe_d(0,z) & \partial_x\partial_y^\sharp e_d(0,0) & \partial_x\partial_y^\sharp e_d(0,z)\\
\partial_xe_d(z,0) & \partial_xe_d(z,z) & \partial_x\partial_y^\sharp e_d(z,0) & \partial_x\partial_y^\sharp e_d(z,z)
\end{pmatrix}.
\end{equation}

\begin{dfn}
\label{dfn Lambda dz}
Let $d \geq \max(d_2,d_3)$, let $x \in M$ and let $z \in B_{T_x^*M}(0,b_n\ln d)\setminus\{0\}$, we set:
\begin{equation*}
\Lambda_d(z) = \Lambda_d\left(x,\exp_x\left(\frac{z}{\sqrt{d}}\right)\right),
\end{equation*}
seen as an operator on $\left(T_x^*M \oplus T_x^*M\right) \otimes \R\!\left(\E \otimes \L^d\right)_x$ via the real normal trivialization around~$x$.
\end{dfn}

Then, for all $d \geq \max(d_2,d_3)$, for all $x \in M$ and all $z \in B_{T_x^*M}(0,b_n\ln d)\setminus\{0\}$ we have:
\begin{equation*}
\label{eq expression Lambda dz}
\Lambda_d(z)\! =\! \begin{pmatrix}
\partial_x\partial_y^\sharp e_d(0,0) & \partial_x\partial_y^\sharp e_d(0,z)\\
\partial_x\partial_y^\sharp e_d(z,0) & \partial_x\partial_y^\sharp e_d(z,z)
\end{pmatrix}-\begin{pmatrix}
\partial_xe_d(0,0) & \partial_xe_d(0,z)\\
\partial_xe_d(z,0) & \partial_xe_d(z,z)
\end{pmatrix}
\!\Theta_d(z)^{-1} \hspace{-3pt}
\begin{pmatrix}
\partial_y^\sharp e_d(0,0) & \partial_y^\sharp e_d(0,z)\\
\partial_y^\sharp e_d(z,0) & \partial_y^\sharp e_d(z,z)
\end{pmatrix}\! .
\end{equation*}


\subsection{Far off-diagonal term}
\label{subsec far off-diagonal term}

In this section, we state that the far off-diagonal term in Eq.~\eqref{eq main estimate} only contributes an error term. This was already proved in~\cite{Let2016a} for $r<n$. The proof is the same for $r=n$, so we refer to \cite{Let2016a} for the proof. Lem.~\ref{lem almost uncorrelated Theta} below is used in the proof of Prop.~\ref{prop off diagonal is small} but is also of independent interest for our purpose.

\begin{prop}
\label{prop off diagonal is small}
Let $\phi_1,\phi_2 \in \mathcal{C}^0(M)$, then we have the following as $d \to +\infty$:
\begin{equation*}
\int_{M^2 \setminus \Delta_d} \phi_1(x)\phi_2(y) \D_d(x,y) \rmes{M}^2 = \Norm{\phi_1}_\infty\Norm{\phi_2}_\infty O\!\left(d^{r-\frac{n}{2}-1}\right),
\end{equation*}
where the error term is independent of $(\phi_1,\phi_2)$.
\end{prop}

\begin{lem}
\label{lem almost uncorrelated Theta}
For every $(x,y) \in M^2 \setminus \Delta_d$, we have:
\begin{equation*}
\Theta_d(x,y) = \left(\frac{\pi}{d}\right)^n \begin{pmatrix}
E_d(x,x) & 0\\ 0 & E_d(y,y)
\end{pmatrix} \left(\Id + O\!\left(d^{-\frac{n}{2}-1}\right)\right),
\end{equation*}
where the error term is independent of $(x,y) \in M^2 \setminus \Delta_d$.
\end{lem}

\begin{proof}
Since $(x,y)\in M^2 \setminus \Delta_d$, we have $\rho_g(x,y) \geq b_n\frac{\ln d}{\sqrt{d}}$. With our choice of $b_n$ (see Def.~\ref{def bn and Delta d}), the error term in Thm.~\ref{thm off diag estimates} is then $O(d^{\frac{n-k}{2}-1})$, uniformly on $M^2 \setminus \Delta_d$. Thus, by Thm.~\ref{thm off diag estimates},
\begin{equation*}
\Theta_d(x,y) = \left(\frac{\pi}{d}\right)^n \begin{pmatrix}
E_d(x,x) & 0\\ 0 & E_d(y,y)
\end{pmatrix} + O\!\left(d^{-\frac{n}{2}-1}\right).
\end{equation*}
The results follows from the fact that the leading term is $\Id + O(d^{-1})$, by Thm.~\ref{thm Dai Liu Ma}.
\end{proof}


\subsection{Near-diagonal term}
\label{subsec near-diagonal term}

In this section, we conclude the proof of Thm.~\ref{thm asymptotics variance}, up to the technical lemmas whose proofs were postponed until Appendices~\ref{sec technical computations for Section limit distrib} and~\ref{sec technical computations for Section main proof}.

\begin{dfn}
\label{def Dd xz}
Let $d \geq \max(d_1,d_2,d_3)$, let $x \in M$ and let $z \in B_{T_xM}(0,b_n\ln d) \setminus \{0\}$, we define:
\begin{equation*}
D_d(x,z) = d^{-r} \D_d\left(x,\exp_x \left(\frac{z}{\sqrt{d}}\right)\right).
\end{equation*}
\end{dfn}

Recall that $D_{n,r}$ was defined by Def.~\ref{def Dnr}. The main result of this section is the following.

\begin{prop}
\label{prop asymptotic Ddxz}
Let $\alpha \in \left(0,1\right)$, then for all $x \in M$, for all $z \in B_{T_xM}(0,b_n \ln d) \setminus \{0\}$ we have:
\begin{equation*}
D_d(x,z) = D_{n,r}\left(\Norm{z}^2\right)\left(1+ O\!\left(d^{-\alpha}\right)\right) + O\!\left(d^{-\alpha}\right),
\end{equation*}
where the error terms do not depend on $(x,z)$.
\end{prop}

First, let us prove that Prop.~\ref{prop integral expression variance}, \ref{prop off diagonal is small} and~\ref{prop asymptotic Ddxz} together imply Thm.~\ref{thm asymptotics variance}.

\begin{proof}[Proof of Thm.~\ref{thm asymptotics variance}]
The main point is to compute the asymptotics of the near-diagonal term in Eq.~\eqref{eq main estimate}. Let us fix $\alpha \in \left(0,1\right)$, $\beta \in \left(0,\frac{1}{2}\right)$ and $\phi_1, \phi_2 \in \mathcal{C}^0(M)$. Let $x \in M$, recall that $\sqrt{\kappa}$ is the density of $\rmes{M}$ with respect to the Lebesgue measure, in the exponential chart centered at $x$, where $\kappa$ was introduced in Sect.~\ref{subsec near-diagonal estimates}. Then, by a change of variable $y = \exp_x\left(\frac{z}{\sqrt{d}}\right)$, we have:
\begin{multline}
\label{eq near diag integral}
\int_{\Delta_d \setminus \Delta} \phi_1(x)\phi_2(y) \D_d(x,y) \rmes{M}^2 =\\
d^{r-\frac{n}{2}}\int_{x \in M} \phi_1(x) \int_{z \in B_{T_xM}\left(0,b_n \ln d\right)} \phi_2\left(\exp_x\!\left(\frac{z}{\sqrt{d}}\right)\right) D_d(x,z) \kappa\left(\frac{z}{\sqrt{d}}\right)^\frac{1}{2} \dx z\rmes{M}.
\end{multline}
As we already discussed in Sect.~\ref{subsec near-diagonal estimates}, $\kappa(z) = 1 +O(\Norm{z}^2)$ and the error term is independent of~$x$. Hence, $\kappa\left(\frac{z}{\sqrt{d}}\right)^\frac{1}{2} = 1 +O\!\left(\frac{(\ln d)^2}{d}\right)$, and by Prop.~\ref{prop asymptotic Ddxz}, for all $\gamma \in \left(\alpha,1\right)$,
\begin{multline}
\label{eq inner int 1}
\int_{z \in B_{T_xM}\left(0,b_n \ln d\right)} \phi_2\left(\exp_x\!\left(\frac{z}{\sqrt{d}}\right)\right) D_d(x,z) \kappa\left(\frac{z}{\sqrt{d}}\right)^\frac{1}{2} \dx z =\\
\left(\int_{z \in B_{T_xM}\left(0,b_n \ln d\right)} \phi_2\left(\exp_x\!\left(\frac{z}{\sqrt{d}}\right)\right) D_{n,r}(\Norm{z}^2) \dx z\right) \left(1+ O\!\left(d^{-\gamma}\right)\right) + \Norm{\phi_2}_{\infty} O\!\left(\frac{(\ln d)^n}{d^\gamma}\right).
\end{multline}
Since $\gamma > \alpha$, $(\ln d)^n d^{-\gamma} = O(d^{-\alpha})$. Similarly, there exists $C_\beta >0$ such that $b_n \frac{\ln d}{\sqrt{d}} \leq C_\beta d^{-\beta}$ for all $d \in \N^*$. Then we have:
\begin{multline}
\label{eq inner int 2}
\norm{\int_{z \in B_{T_xM}\left(0,b_n \ln d\right)} \left(\phi_2\left(\exp_x\!\left(\frac{z}{\sqrt{d}}\right)\right)-\phi_2(x)\right) D_{n,r}(\Norm{z}^2) \dx z} \\
\leq \varpi_{\phi_2}\left(C_\beta d^{-\beta}\right) \int_{z \in B_{T_xM}\left(0,b_n \ln d\right)} \norm{D_{n,r}(\Norm{z}^2)} \dx z,
\end{multline}
where $\varpi_{\phi_2}$ is the continuity modulus of $\phi_2$ (see Def.~\ref{def continuity modulus}). Besides, by Cor.~\ref{cor estimates Dnr and integrability},
\begin{equation}
\label{eq integral Dnr z}
\begin{aligned}
\int_{z \in B_{T_xM}\left(0,b_n \ln d\right)} \norm{D_{n,r}(\Norm{z}^2)} \dx z &= \vol{\S^{n-1}}\frac{1}{2}\int_{t=0}^{(b_n \ln d)^2} D_{n,r}(t)t^\frac{n-2}{2} \dx t\\
&= \vol{\S^{n-1}} \left(\mathcal{I}_{n,r} + O\!\left(e^{-\frac{1}{4}(b_n \ln d)^2}\right) \right),
\end{aligned}
\end{equation}
and the error term is $O(d^{-1})$, since $\frac{1}{4}(b_n \ln d)^2 \geq \ln d$ for $d$ large enough. By Eq.~\eqref{eq inner int 1}, \eqref{eq inner int 2} and~\eqref{eq integral Dnr z}, the innermost integral on the right-hand side of Eq.~\eqref{eq near diag integral} equals:
\begin{equation*}
\phi_2(x) \vol{\S^{n-1}} \mathcal{I}_{n,r} + \varpi_{\phi_2}\left(C_\beta d^{-\beta}\right)O(1) + \Norm{\phi_2}_\infty O(d^{-\alpha}),
\end{equation*}
and the error terms are independent of $x \in M$ and $(\phi_1,\phi_2)$. Finally, by Eq.~\eqref{eq near diag integral},
\begin{multline*}
\int_{\Delta_d \setminus \Delta} \phi_1(x)\phi_2(y) \D_d(x,y) \rmes{M}^2 = d^{r-\frac{n}{2}}\left(\int_M \phi_1\phi_2 \rmes{M}\right) \vol{\S^{n-1}} \mathcal{I}_{n,r}\\
+ \Norm{\phi_1}_\infty \varpi_{\phi_2}\left(C_\beta d^{-\beta}\right)O(d^{r-\frac{n}{2}}) + \Norm{\phi_1}_\infty\Norm{\phi_2}_\infty O(d^{r-\frac{n}{2}-\alpha}).
\end{multline*}
We conclude the proof by combining this last relation with Prop.~\ref{prop integral expression variance}, \ref{prop off diagonal is small} and, in the case $r=n$, Thm.~\ref{thm expectation} for $\phi_1\phi_2$.
\end{proof}

The remainder of this section is mostly dedicated to the proof of Prop.~\ref{prop asymptotic Ddxz}. We will deduce this proposition from several technical lemmas stated below.

Let $x \in M$, then, any choice of an isometry between $T_xM$ and $\R^n$ and an isometry between $\R\!\left(\E\otimes\L^d\right)_x$ and $\R^r$ allows us to see the Bargmann--Fock process $(s(z))_{z \in \R^n}$, studied in Sect.~4, as a smooth Gaussian process from $T_xM$ to $\R\!\left(\E\otimes\L^d\right)_x$. The distribution of this process does not depend on our choice of isometries. Thus, in the following, we can consider $\Theta(z)$ and $\Theta_d(z)$ (resp.~$\Omega(z)$ and $\Omega_d(z)$, resp.~$\Lambda(z)$ and $\Lambda_d(z)$) as operators on the same space.

\begin{lem}
\label{lem asymptotic Theta d}
Let $\alpha \in \left(0,1\right)$, then for all $x \in M$, for all $z \in B_{T_xM}(0,b_n \ln d) \setminus \{0\}$ we have:
\begin{equation*}
\Theta(z)^{-\frac{1}{2}} \Theta_d(z) \Theta(z)^{-\frac{1}{2}} = \Id + O\!\left(d^{-\alpha}\right),
\end{equation*}
where the error term does not depend on $(x,z)$.
\end{lem}

\begin{proof}
See Appendix~\ref{sec technical computations for Section main proof}.
\end{proof}

\begin{rem}
\label{rem bilinear and not linear}
One could wonder why Lem.~\ref{lem asymptotic Theta d} does not state that $\Theta_d(z) = \Theta(z)\left(\Id + O\!\left(d^{-\alpha}\right)\right)$, which would be somewhat simpler. First, note that this statement is not equivalent to Lem.~\ref{lem asymptotic Theta d}, since some of the eigenvalues of $\Theta(z)$ converge to $0$ as $z \to 0$. In fact, this alternative statement turns out to be false in general. Moreover, even if $\Theta_d(z)$ is a linear map, it represents a variance, i.e.~something intrinsically bilinear. It is then quite natural to consider $\Theta(z)^{-\frac{1}{2}} \Theta_d(z) \Theta(z)^{-\frac{1}{2}}$ since this is how $\Theta_d(z)$ transforms if we act on $\R\!\left(\E\otimes\L^d\right)_x$ by $\Theta(z)^\frac{1}{2}$. This remark also applies to Lem.~\ref{lem asymptotic Omega d} and~\ref{lem asymptotic Lambda d} below.
\end{rem}

Let us forget about the proof of Prop.~\ref{prop asymptotic Ddxz} for a minute, and prove the existence of $d_2$ (see Lem.~\ref{lem def d2}) as a corollary of Lem.~\ref{lem almost uncorrelated Theta} and~\ref{lem asymptotic Theta d}. Note that the proofs of these lemmas only rely on the estimates of Sect.~\ref{sec estimates for the bergman kernel}, so there is no logical loop here.

\begin{proof}[Proof of Lem.~\ref{lem def d2}]
\label{proof existence d2}
We want to prove that, as soon as $d$ is large enough, $\ev_{x,y}^d$ is surjective for all $(x,y) \in M^2 \setminus \Delta$, that is $\det\left(\ev_{x,y}^d(\ev_{x,y}^d)^*\right) \neq 0$. By Eq.~\eqref{eq variance evxy} and the definition of $\Theta_d$ (Def.~\ref{def Theta d xy}),
\begin{equation*}
\det\left(\ev_{x,y}^d(\ev_{x,y}^d)^*\right) = \left(\frac{d}{\pi}\right)^{2rn} \det\left(\Theta_d(x,y)\right),
\end{equation*}
so we have to show that $\det\left(\Theta_d(x,y)\right)$ does not vanish on $M^2 \setminus \Delta$, for $d$ large enough. By Lem.~\ref{lem almost uncorrelated Theta} and Thm.~\ref{thm Dai Liu Ma},
\begin{equation}
\label{eq det Theta d off-diag}
\det\left(\Theta_d(x,y)\right) = 1 + O\!\left(d^{-1}\right),
\end{equation}
uniformly on $M^2 \setminus \Delta_d$. Let $(x,y) \in \Delta_d \setminus \Delta$ and let us assume that $d \geq d_3$ so that we can write $y$ as $\exp_x\left(\frac{z}{\sqrt{d}}\right)$ with $z \in B_{T_xM}(0,b_n \ln d) \setminus \{0\}$. Then, by Lem.~\ref{lem asymptotic Theta d} and~\ref{lem non-degeneracy Omega},
\begin{equation}
\label{eq det Theta d near-diag}
\det\left(\Theta_d(x,y)\right) = \det\left(\Theta_d(z)\right) = \det\left(\Theta(z)\right)\left(1+ O\!\left(\sqrt{d}\right)\right) = \left(1- e^{-\Norm{z}^2}\right)^r \left(1+ O\!\left(\sqrt{d}\right)\right),
\end{equation}
uniformly on $\Delta_d \setminus \Delta$. The result follows from Eq.~\eqref{eq det Theta d off-diag} and~\eqref{eq det Theta d near-diag}.
\end{proof}

We can now go back to the proof of Prop.~\ref{prop asymptotic Ddxz}.

\begin{lem}
\label{lem asymptotic Omega d}
Let $\alpha \in \left(0,1\right)$, then for all $x \in M$, for all $z \in B_{T_xM}(0,b_n \ln d) \setminus \{0\}$ we have:
\begin{equation*}
\Omega(z)^{-\frac{1}{2}} \Omega_d(z) \Omega(z)^{-\frac{1}{2}} = \Id + O\!\left(d^{-\alpha}\right),
\end{equation*}
where the error term does not depend on $(x,z)$.
\end{lem}

\begin{proof}
See Appendix~\ref{sec technical computations for Section main proof}.
\end{proof}

\begin{lem}
\label{lem asymptotic Lambda d}
Let $\alpha \in \left(0,1\right)$, then for all $x \in M$, for all $z \in B_{T_xM}(0,b_n \ln d) \setminus \{0\}$ we have:
\begin{equation*}
\Lambda(z)^{-\frac{1}{2}} \Lambda_d(z) \Lambda(z)^{-\frac{1}{2}} = \Id + O\!\left(d^{-\alpha}\right),
\end{equation*}
where the error term does not depend on $(x,z)$.
\end{lem}

\begin{proof}
Let $\alpha \in \left(0,1\right)$. Let $x \in M$ and $z \in B_{T_xM}(0,b_n \ln d) \setminus \{0\}$. By Def.~\ref{def Omega dz}, $\Omega_d(z)$ is an operator on:
\begin{equation*}
\left(\R^2 \otimes \R\!\left(\E \otimes \L^d\right)_x\right)\oplus \left((T_x^*M)^2\otimes \R\!\left(\E \otimes \L^d\right)_x\right).
\end{equation*}
Using this splitting, we write $\Omega_d(z)$ by blocks as:
\begin{equation*}
\Omega_d(z) = \begin{pmatrix}
\Theta_d(z) & \Omega^1_d(z)^* \\ \Omega^1_d(z) & \Omega^2_d(z)
\end{pmatrix},
\end{equation*}
thus defining $\Omega^1_d(z)$ and $\Omega^2_d(z)$. For $d$ large enough, $\Theta_d(z)$ is invertible and its Schur complement is $\Lambda_d(z) = \Omega^2_d(z) - \Omega^1_d(z) \Theta_d(z)^{-1} \Omega^1_d(z)^*$. It is then known that $\Lambda_d(z)^{-1}$ is the bottom-right block of $\Omega_d(z)^{-1}$, that is:
\begin{equation*}
\Lambda_d(z)^{-1} =\begin{pmatrix} 0 & \Id
\end{pmatrix} \Omega_d(z)^{-1} \begin{pmatrix}
0 \\ \Id
\end{pmatrix} = \begin{pmatrix}
0 & \Id
\end{pmatrix}\left(\Omega(z)^{-1}+ \Omega(z)^{-\frac{1}{2}}O\!\left(d^{-\alpha}\right)\Omega(z)^{-\frac{1}{2}}\right)
\begin{pmatrix}
0 \\ \Id
\end{pmatrix},
\end{equation*}
where the second equality is given by Lem.~\ref{lem asymptotic Omega d} and the error term is independent of $(x,z)$. Similarly, $\Lambda(z)$ is the Schur complement of $\Theta(z)$ in $\Omega(z)$, so that
\begin{equation*}
\Lambda(z)^{-1}=\begin{pmatrix}
0 & \Id
\end{pmatrix} \Omega(z)^{-1} \begin{pmatrix}
0 \\ \Id
\end{pmatrix}.
\end{equation*}
Since by Lem.~\ref{lem boundedness sqrt Lambda sqrt Omega}, $\begin{pmatrix} 0 & \Lambda(z)^\frac{1}{2}\end{pmatrix} \Omega_d(z)^{-\frac{1}{2}}$ is bounded, $\Lambda(z)^\frac{1}{2}\Lambda_d(z)^{-1}\Lambda(z)^\frac{1}{2} = \Id+ O\!\left(d^{-\alpha}\right)$, and the error term still does not depend on $(x,z)$.
\end{proof}

\begin{lem}
\label{lem asymptotic conditional expectation}
Let $\alpha \in \left(0,1\right)$, let $x \in M$ and $z \in B_{T_xM}(0,b_n \ln d) \setminus \{0\}$. Let $X_d(z)$ and $Y_d(z)$ be random vectors in $T_x^*M \otimes \R\!\left(\E \otimes \L^d\right)_x$ such that $\left(X_d(z),Y_d(z)\right)\sim \mathcal{N}(\Lambda_d(z))$. Then we have:
\begin{equation*}
\esp{\odet{X_d(z)}\odet{Y_d(z)}} = \esp{\odet{X_\infty(z)}\odet{Y_\infty(z)}} \left(1 + O\!\left(d^{-\alpha}\right)\right)
\end{equation*}
where $\left(X_\infty(z),Y_\infty(z)\right) \sim \mathcal{N}(\Lambda(z))$ and the error term does not depend on $(x,z)$.
\end{lem}

\begin{proof}
See Appendix~\ref{sec technical computations for Section main proof}.
\end{proof}

We conclude this section with the proof of Prop.~\ref{prop asymptotic Ddxz}. Recall the definitions of $D_{n,r}(t)$ (Def.~\ref{def Dnr}), $\D_d(x,y)$ (Def.~\ref{def Dd}) and $D_d(x,z)$ (Def.~\ref{def Dd xz}).

\begin{proof}[Proof of Prop.~\ref{prop asymptotic Ddxz}]
Let us fix $\alpha \in \left(0,1\right)$. Let $x \in M$ and $z \in B_{T_xM}(0,b_n \ln d) \setminus \{0\}$, we set $y = \exp_x\left(\frac{z}{\sqrt{d}}\right)$. We have defined $\Theta_d(z)$ and $\Lambda_d(z)$ so that:
\begin{equation*}
\frac{1}{d^r}\frac{\espcond{\odet{\nabla^d_{x}s_d}\!\odet{\nabla^d_{y}s_d}}{s_d(x)=0=s_d(y)}}{\odet{\ev_{x,y}^d}} = \frac{\esp{\odet{X_d(z)}\odet{Y_d(z)}}}{\det\left(\Theta_d(z)\right)^\frac{1}{2}},
\end{equation*}
where $\left(X_d(z),Y_d(z)\right)\sim \mathcal{N}(\Lambda_d(z))$. By Lem.~\ref{lem asymptotic Theta d} and~\ref{lem asymptotic conditional expectation}, this equals:
\begin{equation*}
\frac{\esp{\odet{X_\infty(z)}\odet{Y_\infty(z)}}}{\det\left(\Theta(z)\right)^\frac{1}{2}}\left(1 + O\!\left(d^{-\alpha}\right)\right),
\end{equation*}
where $\left(X_\infty(z),Y_\infty(z)\right) \sim \mathcal{N}(\Lambda(z))$ and the error term does not depend on $(x,z)$. By the definition of~$\Lambda(z)$ (cf.~Sect.~\ref{subsec conditional variance of the derivatives}) $\left(X_\infty(z),Y_\infty(z)\right)$ is distributed as $(d_0s,d_zs)$, where $s$ is a copy of the Bargmann--Fock process from $T_xM$ to $\R\!\left(\E \otimes \L^d\right)_x$. Then, by Lem.~\ref{lem relation Bargmann--Fock XY},
\begin{equation*}
\esp{\odet{X_\infty(z)}\odet{Y_\infty(z)}} = \esp{\odet{X(\Norm{z}^2)}\odet{Y(\Norm{z}^2)}},
\end{equation*}
where $\left(X(\Norm{z}^2),Y(\Norm{z}^2)\right)$ was defined by Def.~\ref{def XYt}. Besides, $\det\left(\Theta(z)\right) = \left(1-e^{-\Norm{z}^2}\right)^r$ by Lem.~\ref{lem diagonalization Theta}, so that:
\begin{multline*}
\frac{1}{d^r}\frac{\espcond{\odet{\nabla^d_{x}s_d}\!\odet{\nabla^d_{y}s_d}}{s_d(x)=0=s_d(y)}}{\odet{\ev_{x,y}^d}} =\\
\left(D_{n,r}\left(\Norm{z}^2\right) + (2\pi)^r \left(\frac{\vol{\S^{n-r}}}{\vol{\S^n}}\right)^2\right)\left(1 + O\!\left(d^{-\alpha}\right)\right).
\end{multline*}
Besides, Lem.~\ref{lem asymptotic det Ed} and~\ref{lem estimates cond exp x},
\begin{equation*}
\frac{1}{d^r}\frac{\espcond{\odet{\nabla^d_xs_d}}{s_d(x)=0}}{\odet{\ev_x^d}}\frac{\espcond{\odet{\nabla^d_ys_d}}{s_d(y)=0}}{\odet{\ev_x^d}} = (2\pi)^r \left(\frac{\vol{\S^{n-r}}}{\vol{\S^n}}\right)^2 + O\!\left(d^{-1}\right)
\end{equation*}
This yields the desired relation.
\end{proof}


\section{Proof of Theorem~\ref{thm positivity}}
\label{sec proof of the positivity}

The goal of this section is to prove that the leading constant in Thm.~\ref{thm asymptotics variance} is positive. Sect.~\ref{subsec KSS polynomials} is concerned with the definition and proprieties of Kostlan--Shub--Smale polynomials. In Sect.~\ref{subsec Hermite polynomials and Wiener chaos} we recall some facts about Wiener chaoses. In Sect.~\ref{subsec Wiener-Ito expansion of the linear statistics} we compute the chaotic expansion of the linear statistics $\prsc{\rmes{d}}{\phi}$. Finally, we conclude the proof of Thm.~\ref{thm positivity} in Sect.~\ref{subsec conclusion of the proof}.


\subsection{Kostlan--Shub--Smale polynomials}
\label{subsec KSS polynomials}

In this section, we describe a special case of our real algebraic framework. This is the setting we will be considering throughout the proof of Thm.~\ref{thm positivity}. A good reference for the complex algebraic material of this section is~\cite{GH1994}.


\subsubsection{Definition}
\label{subsubsec definition KSS}

We choose $\X$ to be the complex projective space $\C\P^n$ with the real structure induced by the complex conjugation in $\C^{n+1}$. The real locus of $\X$ is the real projective space $\R \P^n$. We set $\L = \mathcal{O}(1) \to \C\P^n$ as the \emph{hyperplane line bundle}, that is the dual of the tautological line bundle:
\begin{equation*}
\label{eq tautological line bundle}
\mathcal{O}(-1) = \left\{(\zeta,x) \in \C^{n+1} \times \C\P^n \mvert \zeta \in x \right\} \longrightarrow \C\P^n.
\end{equation*}
Recall that ample line bundles on $\C\P^n$ are of the form $\mathcal{O}(d) = \left(\mathcal{O}(1)\right)^{\otimes d}$ with $d \in \N^*$ (see~\cite[Sect.~1.1]{GH1994}). The complex conjugation and the usual Hermitian inner product of $\C^{n+1}$ induce compatible real and metric structures on $\mathcal{O}(-1)$, hence on $\mathcal{O}(1)$ by duality. The resulting Hermitian metric on $\mathcal{O}(1)$ is positive and its curvature is the Fubini--Study Kähler form on $\C\P^n$. With our choice of normalization, the induced Riemannian metric is the quotient of the Euclidean metric on $\S^{2n+1} \subset \C^{n+1}$. Finally, we choose $\E$ to be the rank~$r$ trivial bundle $\C^r \times \C\P^n \to \C\P^n$, with the compatible real and metric structures inherited from the standard ones in $\C^r$.

\begin{ntns}
\label{notation multiindices}
Let $\alpha =\left(\alpha_0,\dots,\alpha_n\right) \in \N^{n+1}$, we denote its \emph{length} by $\norm{\alpha} = \alpha_0 + \dots + \alpha_n$. We also define $X^{\alpha} = X_0^{\alpha_0} \cdots X_n^{\alpha_n}$ and $\alpha ! = \alpha_0! \cdots \alpha_n!$. Finally, if $\norm{\alpha}=d$, we denote by $\binom{d}{\alpha}$ the multinomial coefficient $\frac{d!}{\alpha!}$.
\end{ntns}

It is well-known (cf.~\cite{BSZ2000a,BBL1996,Kos1993, Let2016a}) that $\R H^0(\C\P^n,\C^r \otimes \mathcal{O}(d))$ is the space $\left(\R_d^\text{hom}[X_0,\dots,X_n]\right)^r$ of tuples of real homogeneous polynomials of degree $d$ in $n+1$ variables. The $r$ terms $\R_d^\text{hom}[X_0,\dots,X_n]$ in $\R H^0(\C\P^n,\C^r \otimes \mathcal{O}(d))$ are pairwise orthogonal for the inner product~\eqref{eq definition inner product}. Besides, in restriction to one of these terms, \eqref{eq definition inner product} equals:
\begin{equation}
\label{eq def inner product on polynomials}
(P,Q) \mapsto \int_{x\in \C\P^n} h_d(P(x),Q(x)) \rmes{\C\P^n} = \frac{1}{\pi(d+n)!}\int_{z \in \C^{n+1}} P(z)\overline{Q(z)}e^{-\Norm{z}^2}\dx z.
\end{equation}
An orthonormal basis of $\R_d^\text{hom}[X_0,\dots,X_n]$ is then $\left(\sqrt{\frac{(d+n)!}{\pi^n d!}} \sqrt{\binom{d}{\alpha}}X^\alpha\right)_{\norm{\alpha} = d}$. Hence, a standard Gaussian in $\R H^0(\C\P^n,\C^r \otimes \mathcal{O}(d))$ is a $r$-tuple of independent random polynomials of the form:
\begin{equation}
\label{eq def KSS}
\sqrt{\frac{(d+n)!}{\pi^n d!}} \sum_{\norm{\alpha} =d} a_\alpha \sqrt{\binom{d}{\alpha}}X^\alpha,
\end{equation}
where the coefficients $(a_\alpha)_{\norm{\alpha}=d}$ are independent real standard Gaussian variables. Such a random polynomial is called a \emph{Kostlan--Shub--Smale polynomial} (KSS for short).


\subsubsection{Correlation kernel}
\label{subsubsec correlation kernel KSS}

In this section, we study the distribution of the KSS polynomial (see Eq.~\eqref{eq def KSS}). In the setting of Sect.~\ref{subsubsec definition KSS}, $E_d$ is the Bergman kernel of $\C^r \otimes \mathcal{O}(d) \to \C\P^n$. Since the first factor is trivial, we have $E_d = I_r \otimes E'_d$, where $I_r$ is the identity of $\C^r$ and $E'_d$ is the Bergman kernel of $\mathcal{O}(d) \to \C\P^n$. Note that $E'_d$ is the correlation kernel of the field $s_d'$ defined by one KSS polynomial, seen as a random section of $\mathcal{O}(d)$. By Eq.~\eqref{eq def KSS} we have:
\begin{equation*}
\label{eq expression E'd}
E'_d(x,y) = \esp{s_d'(x) \otimes s_d'(y)^*} = \frac{(d+n)!}{\pi^n d!} \sum_{\norm{\alpha}=d} \binom{d}{\alpha} X^\alpha(x) \otimes X^\alpha(y)^*
\end{equation*}

Note that~\eqref{eq def inner product on polynomials} is invariant under the action of the orthogonal group $O_{n+1}(\R)$ on the right. Hence, the distribution of KSS polynomials~\eqref{eq def KSS} and $E'_d$ are equivariant under this action. Since $O_{n+1}(\R)$ acts transitively on the couples of points of $\R\P^n$ at a given distance, $E'_d(x,y)$ only depends on the geodesic distance $\rho_g(x,y)$, and the same holds for derivatives. Loosely speaking, this implies the following, where derivatives are computed with respect to the Chern connection.
\begin{enumerate}
\item The variance of $s_d'(x)$ does not depend on $x \in \R\P^n$.
\label{property1}
\item For all $x \in \R\P^n$, $s_d'(x)$ and $\nabla^d_xs_d'$ are independent.
\label{property2}
\item If $\left(\deron{}{x_1},\dots,\deron{}{x_n}\right)$ is any orthonormal basis of $T_x\R\P^n$, then $\deron{s_d'}{x_i}(x)$ and $\deron{s_d'}{x_j}(x)$ are independent if $i \neq j$. Moreover, the variance of $\deron{s_d'}{x_i}(x)$ does not depend on $i$, nor on our choice of orthonormal basis, nor on $x \in \R\P^n$.
\label{property3}
\end{enumerate}
These properties are very specific of the case of KSS polynomials. They will be useful in Sect.~\ref{subsec Wiener-Ito expansion of the linear statistics}, to compute the Wiener--It{\=o} expansion of $\prsc{\rmes{d}}{\phi}$. We do not give more details here, since Properties~\ref{property1}, \ref{property2} and~\ref{property3} can also be deduced from the expression of $E'_d$ in local coordinates that we derive below.


\subsubsection{Local expression of the kernel}
\label{subsubsec local expression of the kernel}

Let $x \in \R\P^n$, we want to compute the expression of $E'_d$ in some good coordinates around~$x$. We could use the real normal trivialization, but the computations would be cumbersome. Instead, we use a slightly different trivialization. Since $E'_d$ is equivariant under the action $O_{n+1}(\R)$, we can assume that $x=[1:0:\dots:0]$.

We have a chart $\psi_x:(z_1,\dots,z_n) \longmapsto [1:z_1:\cdots:z_n]$ from $\R^n$ to $B_{\R\P^n}\left(x,\frac{\pi}{2}\right)$. We trivialize $\mathcal{O}(d)$ over $B_{\R\P^n}\left(x,\frac{\pi}{2}\right)$ by identifying each fiber with $\mathcal{O}(d)_x$, by parallel transport with respect to the Chern connection $\nabla^d$ along curves of the form $t \mapsto \psi_x(tz)$ with $z \in \R^n$. Thanks to this trivialization, we can consider $E'_d$ as a map taking values in $\R$. Recall that we defined a scaled version $e_d$ of the Bergman kernel $E_d$ by Eq.~\eqref{eq def ed}. The following is related without being an exact analogue. For all $w,z \in \R^n$, we set:
\begin{equation}
\label{eq def xi d}
\xi_d(w,z) = \frac{\pi^n d!}{(d+n)!} E'_d\left(\psi_x(w),\psi_x(z)\right).
\end{equation}

A computation in local coordinates yields the following lemma. The Chern connection $\nabla^d$ coincides at the origin with the usual differential in our trivialization. Hence taking the values at $(0,0)$ of the following expressions proves that $s'_d$ satisfies Properties~\ref{property1}, \ref{property2} and~\ref{property3} (cf.~Sect.~\ref{subsubsec correlation kernel KSS}).

\begin{lem}
\label{lem expression xi d}
Let $d \in \N^*$ and let $i,j \in \{1,\dots,n\}$. Then for all $w,z \in \R^n$ we have:
\begin{align*}
\xi_d(w,z) &= \left(\frac{1+\prsc{w}{z}}{\sqrt{1+\Norm{w}^2}\sqrt{1+\Norm{z}^2}}\right)^d,\\
\partial_{x_i}\xi_d(w,z) &= d \xi_d(w,z) \left(\frac{z_i}{1+\prsc{w}{z}} - \frac{w_i}{1+\Norm{w}^2}\right),\\
\partial_{y_j}\xi_d(w,z) &= d \xi_d(w,z) \left(\frac{w_j}{1+\prsc{w}{z}} - \frac{z_j}{1+\Norm{z}^2}\right),
\end{align*}

\begin{multline*}
\partial_{x_i}\partial_{y_j}\xi_d(w,z) = \xi_d(w,z)\left(\frac{d\delta_{ij}}{1+\prsc{w}{z}} - \frac{d^2w_iw_j}{(1+\prsc{w}{z})(1+\Norm{w}^2)} - \frac{d^2z_iz_j}{(1+\prsc{w}{z})(1+\Norm{z}^2)}\right.\\
 + \frac{d^2w_iz_j}{(1+\Norm{w}^2)(1+\Norm{z}^2)} + \left.\frac{(d^2-d)z_iw_j}{(1+\prsc{w}{z})^2}\right),
\end{multline*}
where $\delta_{ij}=1$ if $i=j$, and $\delta_{ij}=0$ otherwise.
\end{lem}


\subsection{Hermite polynomials and Wiener chaos}
\label{subsec Hermite polynomials and Wiener chaos}

In the setting of KSS polynomials, we consider $\R H^0(\C\P^n,\C^r \otimes \mathcal{O}(d)) = \left(\R_d^\text{hom}[X_0,\dots,X_n]\right)^r$, equipped with the inner product~\eqref{eq def inner product on polynomials}. For simplicty, in this section and the following, we denote by $V_d$ this Euclidean space and by $\dx \nu_d$ its standard Gaussian measure. With these notations, $(V_d,\dx \nu_d)$ is our probability space and we denote by $L^1(\dx \nu_d)$ (resp.~$L^2(\dx \nu_d)$) the space of integrable (resp.~square integrable) random variables on this space. Thm.~\ref{thm asymptotics variance} shows that for $d$ large enough, for all $\phi \in \mathcal{C}^0(\R\P^n)$, $\prsc{\rmes{d}}{\phi} \in L^2(\dx \nu_d)$. The proof given in Sect.~\ref{sec proof of the main theorem} shows that this is true for any $d \geq \max(d_0,d_1,d_2,d_3)$, in this framework this is true for any $d \in \N^*$. The idea of this section is to find a nice orthogonal decomposition of $L^2(\dx \nu_d)$. We will study $\prsc{\rmes{d}}{\phi}$ thanks to this decomposition in Sect.~\ref{subsec Wiener-Ito expansion of the linear statistics}. These techniques were already used in a similar context in~\cite{AADL2017,Dal2015,DNPR2016,MPRW2016}, for example. See~\cite{NP2012} for background on the following material.

\begin{dfn}
\label{def Hermite polynomials}
For all $k \in \N$, we denote by $H_k$ the $k$-th \emph{Hermite polynomial}. These polynomials are defined recursively by: $H_0 = 1$, $H_1=X$ and, for all $k \in \N^*$, $H_{k+1}(X) = X H_k(X) - kH_{k-1}(X)$.
\end{dfn}

\begin{rem}
\label{rem def equiv Hk}
Equivalently, one can define $H_k$ by: $H_0=1$ and $\forall k \in \N$, $H_{k+1} = X H_k - H'_k$.
\end{rem}

\begin{lem}
\label{lem properties of Hk}
Let $k \in \N$, then $H_k$ is a polynomial of degree $k$ which is even if $k$ is even and odd if $k$ is odd. Moreover,
\begin{align*}
H_{2k}(0) &= (-1)^k \frac{(2k)!}{2^kk!} & &\text{and} & H_{2k+1}(0) & = 0.
\end{align*}
\end{lem}

\begin{proof}
This is proved by induction, using the recursive definition of $H_k$.
\end{proof}

Let us denote by $\dx \mu_N$ the standard Gaussian measure on $\R^N$. We also denote by $L^2(\dx \mu_N)$ the space of square integrable functions with respect to $\dx \mu_N$. Recall that the family $\left(\frac{1}{\sqrt{k!}}H_k\right)_{k \in \N}$ is a Hilbert basis of $L^2(\dx \mu_1)$ (see~\cite[Prop.~1.4.2]{NP2012}). Similarly, in dimension $N$, the family:
\begin{equation*}
\left\{\prod_{i=1}^N \frac{1}{\sqrt{\alpha_i !}}H_{\alpha_i}(X_i) \mvert \alpha \in \N^N \right\}
\end{equation*}
is a Hilbert basis of $L^2(\dx \mu_N)$. The result in dimension $1$ shows that this family is orthonormal. Then, one only needs to check that the space of polynomials in $N$ variables is dense in $L^2(\dx \mu_N)$. For $N=1$ this is proved in~\cite[Prop.~1.1.5]{NP2012}, and the same proof works in any dimension.

As in Sect.~\ref{sec proof of the main theorem}, we denote by $s_d$ a generic element of $\left(V_d,\dx \nu_d\right)$, that we think of as a standard Gaussian vector in $V_d$. Let $\eta \in V_d^*$, then $\eta(s_d) \in L^2(\dx \nu_d)$ is a real centered Gaussian variable. Moreover, for any $\eta, \eta' \in V_d^*$, we have $\esp{\eta(s_d)\eta'(s_d)} = \prsc{\eta}{\eta'}$. Thus, $V_d^*$ is canonically isometric to a subspace of $L^2(\dx \nu_d)$, via $\eta \mapsto \eta(s_d)$. From now on, we identify $V_d^*$ with its image, so that $V_d^*\subset L^2(\dx \nu_d)$ is a centered Gaussian Hilbert space.

\begin{dfn}
\label{def qth Wiener chaos}
Let $(\eta_1,\dots,\eta_{N_d})$ denote an orthonormal basis of $V_d^*$, that is the $(\eta_i(s_d))_{i \in \{1,\dots,N_d\}}$ are independent real standard Gaussian variables. For all $q \in \N$, we define $C_d[q]$, the \emph{$q$-th Wiener chaos} of the field $s_d$, as the subspace of $L^2(\dx \nu_d)$ spanned by the orthogonal family:
\begin{equation*}
\left\{\prod_{i=1}^{N_d} H_{\alpha_i}(\eta_i) \mvert \alpha \in \N^{N_d}, \norm{\alpha} =q \right\}.
\end{equation*}
\end{dfn}

\begin{rems}
\label{rem Wiener chaos}
\begin{itemize}
\item $C_d[0]$ is the space of constant random variables in $L^2(\dx \nu_d)$ and $C_d[1] = V_d^*$.
\item We do not need to take the closure in the definition of $C_d[q]$ since it is finite-dimensional.
\end{itemize}
\end{rems}

\begin{lem}
\label{lem Wiener chaos canonical}
The Wiener chaoses $(C_d[q])_{q \in \N}$ of $s_d$ do not depend on the choice of the orthonormal basis $(\eta_1,\dots,\eta_{N_d})$ appearing in Def.~\ref{def qth Wiener chaos}.
\end{lem}

\begin{proof}
Let $(\eta_1,\dots,\eta_{N_d})$ and $(\eta'_1,\dots,\eta'_{N_d})$ be two orthonormal basis of $V_d^*$. There exists an orthogonal transformation $U$ of $V_d^*$ such that, for all $i \in \{1,\dots,N_d\}$, $\eta'_i = U(\eta_i)$. The situation being symmetric, we only have to prove that, for any $\beta \in \N^{N_d}$ such that $\norm{\beta} = q$, $\prod_{i=1}^{N_d} H_{\beta_i}(\eta_i)$ is a linear combination of elements of the family: $\left\{\prod_{i=1}^{N_d} H_{\alpha_i}(U(\eta_i)) \mvert \alpha \in \N^{N_d}, \norm{\alpha} =q \right\}$. Dropping the dependence on $d$, this amounts to proving that: if $X=(X_1,\dots,X_N) \in \R^N$ and $U \in O_N(\R)$ then, for all $\beta \in \N^N$ such that $\norm{\beta} = q$, $\prod_{i=1}^N H_{\beta_i}(X_i)$ is a linear combination of the $\left(\prod_{i=1}^N H_{\alpha_i}(U(X_i))\right)_{\norm{\alpha}=q}$.

By~\cite[Prop.~1.4.2]{NP2012}, we have:
\begin{equation*}
\forall t \in \R^n, \qquad \sum_{\alpha \in \N^N} \left(\prod_{i=1}^N H_{\alpha_i}(X_i)\right) \frac{t^\alpha}{\alpha !} = \exp\left(\prsc{t}{X} - \frac{\Norm{t}^2}{2}\right),
\end{equation*}
where $\prsc{\cdot}{\cdot}$ is the standard inner product of $\R^N$ and $\Norm{\cdot}$ the associated norm. The right-hand side being invariant under orthogonal transformation, we have:
\begin{equation*}
\forall t \in \R^n, \qquad \sum_{\alpha \in \N^N} \left(\prod_{i=1}^N H_{\alpha_i}(U(X_i))\right) \frac{(U(t))^\alpha}{\alpha !} = \sum_{\alpha \in \N^N} \left(\prod_{i=1}^N H_{\alpha_i}(X_i)\right) \frac{t^\alpha}{\alpha !}.
\end{equation*}
Now, the components of $U(t)$ are homogeneous polynomials of degree $1$ in $(t_1,\dots,t_N)$. Hence $(U(t))^\alpha$ can only contribute terms of degree $\norm{\alpha}$ to the sum. We conclude by identifying the coefficients of these power series of the variable $t$.
\end{proof}

\begin{lem}
\label{lem Wiener chaos decomposition}
For all $d \in \N^*$, $\bigoplus_{q \in \N} C_d[q]$ is dense in $L^2(\dx \nu_d)$. Moreover, the terms of this direct sum are pairwise orthogonal.
\end{lem}

\begin{proof}
The family $\left(\prod_{i=1}^{N_d} H_{\alpha_i}(X_i)\right)_{\alpha \in \N^{N_d}}$ being orthogonal, the $(C_d[q])_{q \in \N}$ are pairwise orthogonal by definition. Let $(s_{1,d},\dots,s_{N_d,d})$ be an orthonormal basis of $V_d$. We have $s_d= \sum a_i s_{i,d}$, where the $a_i$ are independent $\mathcal{N}(1)$. For any $i \in \{1,\dots,N_d\}$, let $\eta_i = \prsc{\cdot}{s_{i,d}}$, so that $\eta_i(s_d)=a_i$. Then, for any $q \in \N$, $C_d[q]$ is spanned by the random variables $\left(\prod_{i=1}^{N_d} H_{\alpha_i}(a_i)\right)_{\norm{\alpha}=q}$.

Any square integrable functional of $s_d$ can be written as $F(a_1,\dots,a_{N_d})$, with $F \in L^2(\dx \mu_{N_d})$. The conclusion follows, since the span of $\left\{\prod_{i=1}^{N_d} H_{\alpha_i}(X_i) \mvert \alpha \in \N^{N_d} \right\}$ is dense in $L^2(\dx \mu_{N_d})$.
\end{proof}

\begin{ntn}
\label{notation chaotic components}
Let $d \in \N^*$ and let $A \in L^2(\dx \nu_d)$. For any $q \in \N$, we denote by $A[q]$ the \emph{$q$-th chaotic component} of $A$, that is its projection onto $C_d[q]$. Then, we have $A = \sum_{q \in \N} A[q]$ in $L^2(\dx \nu_d)$.
\end{ntn}
 By definition, $A[0] = \esp{A}$. Moreover, the $C_d[q]$ being pairwise orthogonal, we have $\esp{A[q]} = 0$ for any $q \in \N^*$, and $\var{A} = \sum_{q \in \N^*} \var{A[q]}$.


\subsection[Wiener--Ito expansion of the linear statistics]{Wiener--It{\=o} expansion of the linear statistics}
\label{subsec Wiener-Ito expansion of the linear statistics}

Recall that we consider a standard Gaussian section $s_d \in V_d=\left(\R_d^\text{hom}[X_0,\dots,X_n]\right)^r$ and that $\rmes{d}$ denotes the Riemannian measure of integration over its real zero set. Let us fix $d \in \N^*$ and $\phi \in \mathcal{C}^0(\R\P^n)$. By Thm.~\ref{thm asymptotics variance}, $\prsc{\rmes{d}}{\phi} \in L^2(\dx \nu_d)$. The goal of this section is to compute the chaotic expansion of these variables. For all $q \in \N$, we denote $\prsc{\rmes{d}[q]}{\phi}$ for $\left(\prsc{\rmes{d}}{\phi}\right)[q]$.

Since $\prsc{\rmes{d}}{\phi} \in L^2(\dx \nu_d)$, for any $A \in L^2(\dx \nu_d)$ we have $\left(A\prsc{\rmes{d}}{\phi}\right) \in L^1(\dx \nu_d)$ and:
\begin{equation*}
\esp{A\prsc{\rmes{d}}{\phi}} = \esp{\int_{x \in Z_d}\phi(x)A(s_d) \rmes{d}}.
\end{equation*}
Even if $A$ depends on $s_d$, we can apply a Kac--Rice formula (cf.~\cite[Thm.~5.3]{Let2016a}). Thus, we have:
\begin{equation*}
\esp{A\prsc{\rmes{d}}{\phi}} = (2\pi)^{-\frac{r}{2}} \int_{x \in \R\P^n} \frac{\phi(x)}{\odet{\ev_x^d}}\espcond{A\odet{\nabla^d_{x}s_d}}{s_d(x)=0} \rmes{\R\P^n},
\end{equation*}

Recall that $s_d$, is a tuple of independent KSS polynomials and that $E_d = I_r \otimes E'_d$, where $E'_d$ is the correlation kernel of one KSS polynomial. By Eq.~\eqref{eq def xi d} and Lem.~\ref{lem expression xi d} we have:
\begin{equation*}
\odet{\ev_x^d} = \det(E_d(x,x))^\frac{1}{2} = \det(E'_d(x,x))^\frac{r}{2} = \left(\frac{(d+n)!}{\pi^n	d!}\right)^\frac{r}{2}.
\end{equation*}
Denoting $\left(d\frac{(d+n)!}{\pi^n	d!}\right)^{-\frac{1}{2}} \nabla_x^ds_d$ by $L_d(x)$ and $\left(\frac{(d+n)!}{\pi^n	d!}\right)^{-\frac{1}{2}}s_d(x)$ by $t_d(x)$, we get:
\begin{equation}
\label{eq Wiener Expansion Kac-Rice}
\esp{A\prsc{\rmes{d}}{\phi}} = \left(\frac{d}{2\pi}\right)^\frac{r}{2} \int_{x \in \R\P^n} \phi(x)\espcond{A\odet{L_d(x)}}{t_d(x)=0} \rmes{\R\P^n}.
\end{equation}

Let $x \in \R\P^n$, without loss of generality we can assume that the coordinates on $\R^{n+1}$ are such that $x = [1:0:\dots:0]$. Let $\zeta_0(x)$ be one of the two unit vectors in $\R \mathcal{O}(d)_x$, the other one being $-\zeta_0(x)$. This gives an isomorphism between $(\R \oplus T_x^*(\R\P^n)) \otimes \mathcal{O}(d)_x$ and $\R \oplus T_x^*(\R\P^n)$, so that we can consider $(t_d(x),L_d(x))$ as an element of $\R^r \oplus (T_x^*\R\P^n)^r$. We denote by $(t_d^{(1)}(x),\dots,t_d^{(r)}(x))$ the components of $t_d(x)$ and by $(L_d^{(1)}(x),\dots,L_d^{(r)}(x))$ those of $L_d(x)$. The couples $(t_d^{(i)}(x),L_d^{(i)}(x))$ are centered Gaussian vectors in $\R \oplus T_x^*\R\P^n$ that are independent from one another. Moreover, by Lem.~\ref{lem expression xi d}, for all $i \in \{1,\dots,r\}$, the variance operator of $(t_d^{(i)}(x),L_d^{(i)}(x))$ is $\Id$.

Let us choose any orthonormal basis of $T^*_x\R\P^n$, and denote by $(L_d^{i1}(x),\dots,L_d^{in}(x))$ the coordinates of $L_d^{(i)}(x)$ in this basis, so that $\left(L_d^{ij}(x)\right)_{\substack{1\leq i\leq r\\ 1 \leq j \leq n}}$ is the matrix of $L_d(x)$. Then,
\begin{equation*}
\left\{t_d^{(i)}(x) \mvert 1 \leq i \leq r\right\} \sqcup \left\{L_d^{ij}(x) \mvert 1 \leq i \leq r, 1 \leq j \leq n\right\}
\end{equation*}
is a family of independent real standard Gaussian variables in $L^2(\dx \nu_d)$ and we can complete it into an orthonormal basis of $C_d[1]$. We denote by $\{S_d^{(i)}(x) \mid r(n+1) < i \leq N_d\}$, the last elements of such a basis. Below, we will work in the Hilbert basis of $L^2(\dx \nu_d)$ obtained by considering Hermite polynomials in these variables.

\begin{rem}
\label{rem specificity of KSS}
We just used the fact that our random field satifies Properties~\ref{property1}, \ref{property2} and~\ref{property3} of Sect.~\ref{subsubsec correlation kernel KSS}. This is what makes this computation specific to the case of KSS polynomials.
\end{rem}

\begin{ntn}
\label{notation Hermite polynomials multiindices}
Let $\alpha \in \N^r$, $\beta \in \N^{r} \times \N^n$ and $\gamma \in \N^{N_d - r(n+1)}$, we will use the following notations:
\begin{align*}
H_\alpha(t_d(x)) &= \prod_{i=1}^r H_{\alpha_i}\left(t_d^{(i)}(x)\right), & \tilde{H}_\beta(L_d(x)) &= \prod_{\substack{1 \leq i\leq r \\ 1 \leq j \leq n}} H_{\beta_{ij}}\left(L_d^{ij}(x)\right),\\
&\text{and} & \hat{H}_\gamma(S_d(x)) &= \prod_{i=r(n+1)+1}^{N_d} H_{\gamma_i}\left(S_d^{(i)}(x)\right).
\end{align*}
\end{ntn}

We first expand $\odet{L_d(x)}$ in $L^2(\dx \nu_d)$. Since $\odet{L_d(x)}$ only depends on the variables $\{L_d^{ij}(x) \mid 1\leq i \leq r, 1\leq j \leq n\}$, we have:
\begin{equation*}
\label{eq expansion odet Ldx}
\odet{L_d(x)} = \sum_{\beta \in \N^r \times \N^n}  B_\beta \frac{\tilde{H}_\beta(L_d(x))}{\sqrt{\beta !}},
\end{equation*}
where $B_\beta = \frac{1}{\sqrt{\beta !}}\esp{\odet{L_d(x)}\tilde{H}_\beta(L_d(x))}$ for all $\beta \in \N^r \times \N^n$. The coefficient $B_\beta$ only depends on the distribution of $L_d(x)$, which is a standard Gaussian for all $x \in \R\P^n$. Hence $B_\beta$ is independent of~$x$. These coefficients have several symmetries. Note that $\odet{\left(L_d^{ij}(x)\right)_{i,j}}$ is invariant under the following operations:
\begin{itemize}
\item multiplying a whole column or a whole row by $-1$;
\item permuting the rows or permuting the columns.
\end{itemize}
Since the Hermite polynomials of odd degrees are odd (cf.~Lem.~\ref{lem properties of Hk}), the first point shows that $B_\beta = 0$ whenever there exists $i \in \{1,\dots,r\}$ such that $\sum_{j=1}^n \beta_{ij}$ is odd or there exists $j \in \{1,\dots,n\}$ such that $\sum_{i=1}^r \beta_{ij}$ is odd. We denote by $I$ the set of multi-indices $\beta \in \N^r \times \N^n$ such that for all $i \in \{1,\dots,r\}$, $\sum_{j=1}^n \beta_{ij}$ is even and, for all $j \in \{1,\dots,n\}$, $\sum_{i=1}^r \beta_{ij}$ is even.

If $\norm{\beta}=2$, then the only way for $\beta$ to belong to $I$ is that there exists $(i,j)$ such that $\beta_{ij}=2$, the other components of $\beta$ being zero. The second point above shows that, in this case, the value of $B_\beta$ does not depend on the index $(i,j)$ such that $\beta_{ij}=2$.

\begin{ntn}
\label{notation B2}
Let $B_2$ denote the common value of the $B_\beta$ for $\beta \in I$ such that $\norm{\beta}=2$.
\end{ntn}

We can also expand $A\in L^2(\dx \nu_d)$ as:
\begin{equation*}
A = \sum_{\alpha,\beta,\gamma} A_{\alpha,\beta,\gamma}(x) \frac{H_\alpha(t_d(x))}{\sqrt{\alpha!}}\frac{\tilde{H}_\beta(L_d(x))}{\sqrt{\beta!}}\frac{\hat{H}_\gamma(S_d(x))}{\sqrt{\gamma!}},
\end{equation*}
where $A_{\alpha,\beta,\gamma}(x)=\esp{A\frac{H_\alpha(t_d(x))}{\sqrt{\alpha!}}\frac{\tilde{H}_\beta(L_d(x))}{\sqrt{\beta!}}\frac{\hat{H}_\gamma(S_d(x))}{\sqrt{\gamma!}}}$. Then, using the orthonormality properties of the Hermite polynomials, we get:
\begin{equation}
\label{eq Wiener exp A odet L}
\espcond{A\odet{L_d(x)}}{t_d(x)=0} = \sum_{\alpha,\beta} A_{\alpha,\beta,0}(x)B_{\beta} \frac{H_{\alpha}(0)}{\sqrt{\alpha!}},
\end{equation}
where the sum runs over multi-indices such that $\alpha \in 2\N^r$ (see~Lem.\ref{lem properties of Hk}), and $\beta \in I$. Then, by Eq.~\eqref{eq Wiener Expansion Kac-Rice} and \eqref{eq Wiener exp A odet L}, for any $A \in C_d[q]$, we have:
\begin{equation*}
\esp{A\prsc{\rmes{d}}{\phi}} = \left(\frac{d}{2\pi}\right)^\frac{r}{2} \sum_{\substack{\alpha \in 2\N^r,\, \beta \in I\\ \norm{\alpha}+\norm{\beta}=q}} B_\beta \frac{H_\alpha(0)}{\sqrt{\alpha!}} \esp{A \int_{x \in \R\P^n} \phi(x) \frac{H_{\alpha}(t_d(x))}{\sqrt{\alpha!}}\frac{\tilde{H}_{\beta}(L_d(x))}{\sqrt{\beta!}} \rmes{\R\P^n}}
\end{equation*}

We have proved the following proposition.

\begin{prop}
\label{prop Wiener expansion}
For all $d \in \N^*$, for all $\phi \in \mathcal{C}^0(\R\P^n)$, for all $q \in \N$, $\prsc{\rmes{d}[2q+1]}{\phi}=0$ and
\begin{equation}
\label{eq chaotic component q}
\prsc{\rmes{d}[2q]}{\phi} = \left(\frac{d}{2\pi}\right)^\frac{r}{2} \int_{x \in \R\P^n} \phi(x) \sum_{\substack{\alpha \in 2\N^r,\, \beta \in I\\ \norm{\alpha}+\norm{\beta}=2q}} B_\beta \frac{H_\alpha(0)}{\sqrt{\alpha!}} \frac{H_{\alpha}(t_d(x))}{\sqrt{\alpha!}}\frac{\tilde{H}_{\beta}(L_d(x))}{\sqrt{\beta!}} \rmes{\R\P^n}.
\end{equation}
\end{prop}

\begin{rems}
\label{rem Wiener expansion}
\begin{itemize}
\item Recall that the values of the $t_d^{(i)}(x)$ and $L_d^{ij}(x)$ depend on the choice of the unit vector $\zeta_0(x)$, that we used to trivialize $\mathcal{O}(d)_x$. The only other choice of such a unit vector is $-\zeta_0(x)$. Changing $\zeta_0(x)$ to $-\zeta_0(x)$ changes $t_d^{(i)}(x)$ to $-t_d^{(i)}(x)$ and $L_d^{ij}(x)$ to $-L_d^{ij}(x)$. Since we only consider multi-indices $(\alpha,\beta)$ such that $\norm{\alpha}+\norm{\beta}$ is even, the monomials appearing with a non-zero coefficient in $H_{\alpha}(t_d(x))\tilde{H}_{\beta}(L_d(x))$ have even total degree. Hence, the value of $H_{\alpha}(t_d(x))\tilde{H}_{\beta}(L_d(x))$ does not depend on the choice of $\zeta_0(x)$.
\item Since $\sum_{\beta \in I,\, \norm{\beta} = p} B_\beta \frac{1}{\sqrt{\beta!}}\tilde{H}_{\beta}(L_d(x))$ is the $p$-th chaotic component of $\odet{L_d(x)}$, it does not depend on our choice of an orthonormal basis of $T_x^*\R\P^n$. Hence, neither does the value of the sum on the right-hand side of Eq.~\eqref{eq chaotic component q}, for any given $x \in \R\P^n$.
\item By~\cite[Lem.~A.14]{Let2016}, we have:
\begin{equation*}
B_0 = \esp{\odet{L_x(d)}} = (2\pi)^r \frac{\vol{\S^{n-r}}}{\vol{\S^n}}.
\end{equation*}
Then, Prop.~\ref{prop Wiener expansion} for $q=0$ shows that, in the setting of KSS polynomials, for all $\phi \in \mathcal{C}^0(\R\P^n)$,
\begin{equation*}
\esp{\prsc{\rmes{d}}{\phi}} = d^\frac{r}{2} \left(\int_{\R\P^n} \phi \rmes{\R\P^n}\right) \frac{\vol{\S^{n-r}}}{\vol{\S^n}}.
\end{equation*}
That is, in this case, the error term in Thm.~\ref{thm expectation} is zero for any $d \in \N^*$.
\end{itemize}
\end{rems}

Let us conclude this section by writing $\prsc{\rmes{d}[2]}{\phi}$ in a more explicit way.

\begin{lem}
\label{lem Wiener chaos 2}
For all $d \in \N^*$, for all $\phi \in \mathcal{C}^0(\R\P^n)$,
\begin{equation*}
\prsc{\rmes{d}[2]}{\phi} = d^\frac{r}{2}\frac{\vol{\S^{n-r}}}{2n\vol{\S^n}} \int_{x \in \R\P^n} \phi(x) \left(\Norm{L_d(x)}^2 - n\Norm{t_d(x)}^2\right) \rmes{\R\P^n}. 
\end{equation*}
\end{lem}

\begin{proof}
By Prop.~\ref{prop Wiener expansion} and Lem.~\ref{lem properties of Hk}, we have:
\begin{equation}
\label{eq Wiener chaos 2}
\prsc{\rmes{d}[2]}{\phi} = \left(\frac{d}{2\pi}\right)^\frac{r}{2} \int_{x \in \R\P^n} \phi(x)\left(-\frac{B_0}{2}(\Norm{t_d(x)}^2 -r ) + \frac{B_2}{\sqrt{2}} (\Norm{L_d(x)}^2 - nr)\right)\rmes{\R\P^n},
\end{equation}
where $B_2$ is defined by~Ntn.~\ref{notation B2}. Since $H_2 = X^2-1$, we have:
\begin{equation*}
n\sqrt{2}B_2 = \sum_{j=1}^n \esp{\odet{L_d(x)} H_2(L_d^{1j}(x))} = \esp{\odet{L_d(x)} \Norm{L_d^{(1)}(x)}^2} - nB_0.
\end{equation*}
It was proved in \cite[App.~B]{Let2016} that $\odet{L_d(x)}$ is distributed as $\Norm{L_d^{(1)}(x)}\Norm{Z_{n-1}}\dots \Norm{Z_{n-r+1}}$, where $\left(L_d^{(1)}(x),Z_{n-1},\dots,Z_{n-r+1}\right)$ are globally independent and, for all $p \in \{n-r+1,\dots,n-1\}$, $Z_p$ is a standard Gaussian vector in $\R^p$. Since $L_d^{(1)}(x) \sim \mathcal{N}(\Id)$ in a Euclidean space of dimension~$n$, we have:
\begin{align*}
B_0 &= \esp{\odet{L_d(x)}} = \esp{\Norm{L_d^{(1)}(x)}} \prod_{p=n-r+1}^{n-1} \esp{\Norm{Z_p}} = (2\pi)^\frac{r}{2} \frac{\vol{\S^{n-r}}}{{\vol{\S^n}}},\\
\text{and} \quad B_2 &= \frac{1}{n\sqrt{2}} \esp{\Norm{L_d^{(1)}(x)}^3} \prod_{p=n-r+1}^{n-1} \esp{\Norm{Z_p}} -\frac{B_0}{\sqrt{2}} = \frac{B_0}{\sqrt{2}}\left(\frac{2\pi}{n}\frac{\vol{\S^n}}{\vol{\S^{n+2}}} - 1\right) = \frac{B_0}{n\sqrt{2}}.
\end{align*}
We plug these relations in Eq.~\eqref{eq Wiener chaos 2} and this yields the result.
\end{proof}


\subsection{Conclusion of the proof}
\label{subsec conclusion of the proof}

In this section, we finally prove Thm.~\ref{thm positivity}. The key point is the following.

\begin{lem}
\label{lem second chaos is positive}
Let $Z_d$ be the common zero set of $r$ independent Kostlan--Shub--Smale polynomials in $\R\P^n$ then we have the following as $d$ goes to infinity:
\begin{equation*}
\var{\vol{Z_d}[2]} \sim d^{r-\frac{n}{2}} r\left(1+\frac{2}{n}\right)\pi^\frac{n}{2} \frac{\vol{\S^{n-r}}^2}{16\vol{\S^n}}.
\end{equation*}
\end{lem}

Let us prove that this lemma implies Thm.~\ref{thm positivity}.

\begin{proof}[Proof of Thm.~\ref{thm positivity}]
Let us consider the common zero set $Z_d$ of $r$ independent KSS polynomials in $\R\P^n$, and denote by $\rmes{d}$ the Riemannian volume measure on $Z_d$. Let $\mathbf{1}$ be the unit constant function on $\R\P^n$, we have $\prsc{\rmes{d}}{\mathbf{1}}= \vol{Z_d}$ and, by Thm.~\ref{thm asymptotics variance},
\begin{equation*}
d^{-r+\frac{n}{2}}\var{\vol{Z_d}} = \vol{\R\P^n} \left(\frac{\vol{\S^{n-1}}}{(2\pi)^r} \mathcal{I}_{n,r} + \delta_{rn}\frac{2}{\vol{\S^n}}\right) + o\left(1\right).
\end{equation*}
On the other hand, as we explained as the end of Sect.~\ref{subsec Hermite polynomials and Wiener chaos},
\begin{equation*}
d^{-r+\frac{n}{2}}\var{\vol{Z_d}} = d^{-r+\frac{n}{2}}\sum_{q \in \N^*} \var{\vol{Z_d}[q]} \geq d^{-r+\frac{n}{2}}\var{\vol{Z_d}[2]}.
\end{equation*}
By Lem.~\ref{lem second chaos is positive}, we get:
\begin{equation*}
\left(\frac{\vol{\S^{n-1}}}{(2\pi)^r} \mathcal{I}_{n,r} + \delta_{rn}\frac{2}{\vol{\S^n}}\right) \geq \frac{r}{8}\left(1+\frac{2}{n}\right)\pi^\frac{n}{2} \left(\frac{\vol{\S^{n-r}}}{\vol{\S^n}}\right)^2 >0. \qedhere
\end{equation*}
\end{proof}

\begin{rem}
In~\cite{Dal2015}, Dalmao proved that for $n=r=1$, we have $\var{\vol{Z_d}} \sim \sigma^2 \sqrt{d}$ with $\sigma^2 \simeq 0.57\cdots$. What we just said shows that $\sigma^2 = 1+ \mathcal{I}_{1,1}$, and the lower bound we get for this term in the proof of Thm.~\ref{thm positivity} is $\frac{3}{8\sqrt{\pi}} \simeq 0.21\cdots$. Thus, asymptotically, chaotic components of order greater than $4$ must contribute to the leading term of $\var{\vol{Z_d}}$.
\end{rem}

We conclude this section by the proof of Lem.~\ref{lem second chaos is positive}.

\begin{proof}[Proof of Lem.~\ref{lem second chaos is positive}]
Recall that $\rmes{d}$ is the Riemannian volume measure on $Z_d$ and that $\mathbf{1}$ is the unit constant function on $\R\P^n$. By Lem.~\ref{lem Wiener chaos 2}, we have:
\begin{equation*}
\vol{Z_d}[2] = \prsc{\rmes{d}[2]}{\mathbf{1}} = d^\frac{r}{2}\frac{\vol{\S^{n-r}}}{2n\vol{\S^n}} \int_{x \in \R\P^n} \left(\Norm{L_d(x)}^2 - n\Norm{t_d(x)}^2\right) \rmes{\R\P^n}.
\end{equation*}
Since this is a centered variable, its variance equals:
\begin{equation*}
d^r\left(\frac{\vol{\S^{n-r}}}{2n\vol{\S^n}}\right)^2 \int_{(x,y) \in (\R\P^n)^2} \esp{\left(\Norm{L_d(x)}^2 - n\Norm{t_d(x)}^2\right)\left(\Norm{L_d(y)}^2 - n\Norm{t_d(y)}^2\right)} \rmes{\R\P^n}^2.
\end{equation*}

Using the invariance of the distribution of $s_d$ under isometries, we get that:
\begin{equation}
\label{eq variance second chaos volume}
\var{\vol{Z_d}[2]} = d^r\frac{\vol{\S^{n-r}}^2}{8n^2\vol{\S^n}} \mathcal{J}_{n,r}(d),
\end{equation}
where, setting $x_0=[1:0:\dots:0]$,
\begin{equation*}
\mathcal{J}_{n,r}(d) = \int_{y \in \R\P^n} \esp{\left(\Norm{L_d(x_0)}^2 - n\Norm{t_d(x_0)}^2\right)\left(\Norm{L_d(y)}^2 - n\Norm{t_d(y)}^2\right)} \rmes{\R\P^n}.
\end{equation*}
Since $B_{\R\P^n}\left(x_0,\frac{\pi}{2}\right) = \left\{[1:z_1:\cdots:z_n] \in \R\P^n \mvert z \in \R^n \right\}$ has full measure in $\R\P^n$, we can restrict the above integral to this ball and use the local coordinates introduced in Sect.~\ref{subsubsec local expression of the kernel}. These coordinates are centered at $x_0$. Moreover, the density of $\rmes{\R\P^n}$ with respect to the Lebesgue measure in this chart is $z \mapsto (1+\Norm{z}^2)^{-\frac{n+1}{2}}$ (cf.~\cite[p.~30]{GH1994}). By a change of variable $y = [1:z_1:\dots:z_n]$, we have:
\begin{equation}
\label{eq def Jnr}
\mathcal{J}_{n,r}(d) = \int_{z \in \R^n} \mathcal{F}_d(z) (1+\Norm{z}^2)^{-\frac{n+1}{2}}\dx z,
\end{equation}
where
\begin{equation*}
\mathcal{F}_d(z) = \esp{\left(\Norm{L_d(0)}^2 - n\Norm{t_d(0)}^2\right)\left(\Norm{L_d(z)}^2 - n\Norm{t_d(z)}^2\right)}.
\end{equation*}
Here, we denoted $t_d(z)$ instead of $t_d([1:z_1:\dots:z_n])$ and $L_d(z)$ instead of $L_d([1:z_1:\dots:z_n])$.

Let us fix $z \in \R^n$ and compute $\mathcal{F}_d(z)$. Using once again the invariance under the action of $O_{n+1}(\R)$ on $\R\P^n$, we can assume that $z=(\Norm{z},0,\dots,0)$. Let $\left(\deron{}{x_1},\dots,\deron{}{x_n}\right)$ denote the basis of the tangent space of $\R\P^n$ at $[1:\Norm{z}:0:\dots:0]$ given by the partial derivatives in our chart $\psi_{x_0}$ (see~Sect.~\ref{subsubsec local expression of the kernel}). This basis is orthogonal, but $\Norm{\deron{}{x_1}} = (1+\Norm{z}^2)^{-1}$ and $\Norm{\deron{}{x_j}} = (1+\Norm{z}^2)^{-\frac{1}{2}}$ for all $j \in \{2,\dots,n\}$.

The random vectors $\left(t_d^{(i)}(0),t_d^{(i)}(z), L_d^{i1}(0),L_d^{i1}(z),\dots,L_d^{in}(0),L_d^{in}(z)\right)$ for $i \in \{1,\dots,r\}$ are independent equidistributed centered Gaussian vector in $\R^{2n+2}$. The previous relations, together with Lem.~\ref{lem expression xi d}, show that their common variance matrix, by blocks of size $2 \times 2$, is:
\begin{equation}
\label{eq local variance matrix tL}
\begin{pmatrix}
A_d(\Norm{z}^2) & \left(B_d(\Norm{z}^2)\right)^\text{t} & 0 & \cdots & 0\\
B_d(\Norm{z}^2) & D_d(\Norm{z}^2) & 0 & \cdots & 0\\
0 & 0 & C_d(\Norm{z}^2) & & \vdots\\
\vdots & \vdots & & \ddots & 0\\
0 & 0 & \cdots & 0 & C_d(\Norm{z}^2)
\end{pmatrix},
\end{equation}
where, for all $t \geq 0$,
\begin{equation}
\label{eq def A B C and D}
\begin{aligned}
A_d(t) &= \begin{pmatrix}
1 & (1+t)^{-\frac{d}{2}} \\ (1+t)^{-\frac{d}{2}} & 1
\end{pmatrix},
&
B_d(t) &= \begin{pmatrix}
0 & \sqrt{dt}(1+t)^{-\frac{d}{2}} \\ -\sqrt{dt}(1+t)^{-\frac{d}{2}} & 0
\end{pmatrix},
\\
C_d(t) &= \begin{pmatrix}
1 & \hspace{-2mm}(1+t)^\frac{1-d}{2} \\ (1+t)^\frac{1-d}{2} & 1
\end{pmatrix}
&
\text{and } D_d(t) &= \begin{pmatrix}
1 & \hspace{-2mm}(1+t-dt)(1+t)^{-\frac{d}{2}} \\ (1+t-dt)(1+t)^{-\frac{d}{2}} & 1
\end{pmatrix}.
\end{aligned}
\end{equation}
Using the independence and equidistribution of the couples $\left(t_d^{(i)}(x),L_d^{(i)}(x)\right)$, we have:
\begin{multline*}
\mathcal{F}_d(z) = r \left(\sum_{j,l} \esp{\left(L_d^{1j}(0)\right)^2 \left(L_d^{1l}(z)\right)^2} - n \sum_l \esp{\left(t_d^{(1)}(0)\right)^2\left(L_d^{1l}(z)\right)^2}\right.\\ \left.- n \sum_j \esp{\left(L_d^{1j}(0)\right)^2\left(t_d^{(1)}(z)\right)^2} + n^2 \esp{\left(t_d^{(1)}(0)\right)^2\left(t_d^{(1)}(z)\right)^2} \right).
\end{multline*}
If $(X,Y)$ is a centered Gaussian vector in $\R^2$ such that $\var{X}=1=\var{Y}$, then by Wick's formula (cf.~\cite[Lem.~11.6.1]{AT2007}) we have: $\esp{X^2Y^2} = 9 + 2 \esp{XY}^2$. We apply this relation to each term of the previous sum. Then, by Eq.~\eqref{eq local variance matrix tL} and Eq.~\eqref{eq def A B C and D}, we have $\mathcal{F}_d(z) = 2rF_d(d\Norm{z}^2)$, where $F_d$ is defined by:
\begin{equation*}
\label{eq def Fd}
\forall t \in \R, \qquad F_d(t) = \left(1+\frac{t}{d}\right)^{-d}\left(\left(1+\frac{t}{d}-t\right)^2+(n-1)\left(1+\frac{t}{d}\right) - 2nt + n^2\right).
\end{equation*}
Then, by a change of variable $t = d\Norm{z}^2$ in Eq.~\eqref{eq def Jnr},
\begin{equation}
\label{eq dominated convergence}
\mathcal{J}_{n,r}(d) = d^{-\frac{n}{2}} r \vol{\S^{n-1}} \int_{t=0}^{+\infty} F_d(t)t^\frac{n-2}{2}\left(1+\frac{t}{d}\right)^{-\frac{n+1}{2}}\dx t.
\end{equation}

Let $t \geq 0$, we have:
\begin{equation*}
F_d(t)t^\frac{n-2}{2}\left(1+\frac{t}{d}\right)^{-\frac{n+1}{2}} \xrightarrow[d \to +\infty]{} \left(t^2 -2t(n+1) +n(n+1)\right)t^\frac{n-2}{2}e^{-t}.
\end{equation*}
Moreover, for all $d \in \N^*$,
\begin{equation*}
\norm{F_d(t)t^\frac{n-2}{2}\left(1+\frac{t}{d}\right)^{-\frac{n+1}{2}}} \leq \left(1+\frac{t}{d}\right)^{-d}t^\frac{n-2}{2}\left(4t^2 +(n+1)(3t+n)\right).
\end{equation*}
Let $d_0 > \frac{n}{2} +2$. Since $\left(1+\frac{t}{d}\right)^{-d}$ is a non-increasing sequence of $d$, for all $d \geq d_0$,
\begin{equation*}
\norm{F_d(t)t^\frac{n-2}{2}\left(1+\frac{t}{d}\right)^{-\frac{n+1}{2}}} \leq \left(1+\frac{t}{d_0}\right)^{-d_0}t^\frac{n-2}{2}\left(4t^2 +(n+1)(3t+n)\right),
\end{equation*}
and the right-hand side is integrable as a function of $t$. By Lebesgue's Theorem, we have:
\begin{equation}
\label{eq last relation}
\int_{t=0}^{+\infty} \frac{F_d(t)t^\frac{n-2}{2}}{\left(1+\frac{t}{d}\right)^\frac{n+1}{2}}\dx t \xrightarrow[d \to +\infty]{} \int_0^{+\infty} \left(t^2 -2t(n+1) +n(n+1)\right)t^\frac{n-2}{2}e^{-t}\dx t = \Gamma\left(\frac{n}{2}+2\right),
\end{equation}
where $\Gamma$ is Euler's Gamma function. The conclusion follows from Eq.~\eqref{eq variance second chaos volume}, \eqref{eq dominated convergence} and~\eqref{eq last relation}.
\end{proof}


\appendix


\section{Technical computations for Section~\ref{sec properties of the limit distribution}}
\label{sec technical computations for Section limit distrib}

Before proving the technical lemmas of Sect.~\ref{sec properties of the limit distribution}, we state several estimates that will be useful in this section and the next. Recall Def.~\ref{def a and vi}, \ref{def b and P} and~\ref{def u}. The following hold as $t$ goes to infinity.
\begin{equation}
\label{eq asymptotic a b+ b-}
a(t) \xrightarrow[t \to +\infty]{} -1, \qquad b_+(t) \xrightarrow[t \to +\infty]{} 0 \qquad \text{and} \qquad b_-(t) \xrightarrow[t \to +\infty]{} \sqrt{2}.
\end{equation}
\begin{equation}
\label{eq asymptotic u v}
\forall i \in \{1,2\}, \ u_i(t) \xrightarrow[t \to +\infty]{} 1 \qquad \text{and} \qquad \forall i \in \{1,2,3,4\}, \ v_i(t) \xrightarrow[t \to +\infty]{} 1.
\end{equation}
The following hold as $t$ goes to $0$.
\begin{align}
a(t) &= 1 - \frac{t}{2} - \frac{t^2}{8} + O(t^3), & (b_+(t))^2 &= 2 - \frac{t}{2} + O(t^2), \label{DL a and b+2}\\
b_+(t)b_-(t) &= \sqrt{t} \left(1 + O(t^2)\right), & (b_-(t))^2 &= \frac{t}{2} + \frac{t^2}{8} + O(t^3), \label{DL b+b- and b-2}\\
u_1(t) &= t +O(t^2), & u_2(t) &= \frac{t^2}{12}+O(t^3), \label{DL u1 and u2}\\
v_1(t) &= 2 + O(t), & v_2(t) &= \frac{t^3}{48} + O(t^4), \label{DL v1 and v2}\\
v_3(t) &= t + O(t^2), & v_4(t) &= 2 + O(t). \label{DL v3 and v4}
\end{align}

\begin{proof}[Proof of Lem.~\ref{lem diagonalization Omega tilde}]
Recall that $P$ is defined by Def.~\ref{def b and P} and $\tilde{\Omega}$ by Eq.~\eqref{eq def Omega tilde}. One can check by a direct computation that, for any $t \in [0,+\infty)$, $P(t) = (A(t)\otimes I_2) \sigma (Q\otimes I_2)$, where
\begin{align*}
A(t) &= \frac{1}{\sqrt{2}} \begin{pmatrix} b_-(t) & - b_+(t) \\ b_+(t) & b_-(t) \end{pmatrix},
&
\sigma &= \begin{pmatrix} 1 & 0 & 0 & 0 \\ 0 & 0 & 0 & 1 \\ 0 & 1 & 0 & 0 \\ 0 & 0 & 1 & 0 \end{pmatrix},
&
&\text{and}
&
Q &= \frac{1}{\sqrt{2}} \begin{pmatrix} 1 & -1 \\ 1 & 1 \end{pmatrix}.
\end{align*}
Moreover, these three matrices are orthogonal. Then we have:
\begin{align*}
\sigma (Q\otimes I_2) \tilde{\Omega}(t) (Q^\text{t}\otimes I_2) \sigma^\text{t} &= \sigma \begin{pmatrix}
1 - e^{-\frac{1}{2}t} & 0 & 0 & -\sqrt{t}e^{-\frac{1}{2}t}\\
0 & 1 + e^{-\frac{1}{2}t} & \sqrt{t}e^{-\frac{1}{2}t} & 0\\
0 & \sqrt{t}e^{-\frac{1}{2}t} & 1 -(1-t)e^{-\frac{1}{2}t} & 0 \\
-\sqrt{t}e^{-\frac{1}{2}t} & 0 & 0 & 1 + (1-t)e^{-\frac{1}{2}t}
\end{pmatrix} \sigma^\text{t} \\
&= \begin{pmatrix}
1 - e^{-\frac{1}{2}t} & -\sqrt{t}e^{-\frac{1}{2}t} & 0 & 0\\
-\sqrt{t}e^{-\frac{1}{2}t} & 1 + (1-t)e^{-\frac{1}{2}t} & 0 & 0\\
0 & 0 & 1 + e^{-\frac{1}{2}t} & \sqrt{t}e^{-\frac{1}{2}t}\\
0 & 0 & \sqrt{t}e^{-\frac{1}{2}t} & 1 -(1-t)e^{-\frac{1}{2}t}
\end{pmatrix}\\
&= I_4 + e^{-\frac{t}{2}} \left(\begin{pmatrix}
1 & \sqrt{t} \\ \sqrt{t} & t-1
\end{pmatrix} \otimes \begin{pmatrix}
-1 & 0 \\ 0 & 1
\end{pmatrix}\right),
\end{align*}
where $I_4$ stands for identity matrix of size $4$. Recalling the definitions of $(v_i(t))_{1 \leq i \leq 4}$, $b_+(t)$ and $b_-(t)$ (see Def.~\ref{def a and vi} and~\ref{def b and P}). We conclude the proof by checking that:
\begin{equation*}
A(t)\begin{pmatrix}
1 & \sqrt{t} \\ \sqrt{t} & t-1
\end{pmatrix} \left(A(t)\right)^\text{t} = \begin{pmatrix}
\frac{t}{2} - \sqrt{1 + \left(\frac{t}{2}\right)^2} & 0 \\ 0 & \frac{t}{2} + \sqrt{1 + \left(\frac{t}{2}\right)^2}
\end{pmatrix}.\qedhere
\end{equation*}
\end{proof}

\begin{proof}[Proof of Lem.~\ref{lem non-degeneracy Omega}]
Let $z \in \R^n \setminus \{0\}$, by Eq.~\eqref{eq expression Omega} and Lem.~\ref{lem variance operator 1-jets}, we have:
\begin{equation*}
\det\left(\Omega(z)\right) = \det\left(\Omega'(z)\right)^r = \det\left(\tilde{\Omega}(\Norm{z}^2)\right)^r \left(1 - e^{-\Norm{z}^2}\right)^{r(n-1)}
\end{equation*}
and it is enough to prove that $\det\left(\tilde{\Omega}(t)\right) > 0$ whenever $t >0$. By Lem.~\ref{lem diagonalization Omega tilde}, we have:
\begin{equation}
\label{eq determinant Omega tilde}
\forall t \geq 0, \qquad \det\left(\tilde{\Omega}(t)\right) = v_1(t)v_2(t)v_3(t)v_4(t) = 1 - (t^2+2) e^{-t} + e^{-2t}=f(t),
\end{equation}
where the last equality defines $f:[0,+\infty) \to \R$. We have $f(0)=0$ and for all $t > 0$, $f'(t) = e^{-t}g(t)$ where $g(t) = t^2 -2t + 2 -2 e^{-t}$. Then $g(0)=0$ and $\forall t >0$, $g'(t) = 2(e^{-t}-1+t)> 0$. Thus $g$ is positive on $(0,+\infty)$ and so is $f$. Finally, we have $\forall t >0$, $\det\left(\tilde{\Omega}(t)\right) >0$.
\end{proof}

\begin{proof}[Proof of Lem.~\ref{lem non-degeneracy Lambda}]
Let $z \in \R^n \setminus \{0\}$, as above we have:
\begin{equation*}
\det\left(\Lambda(z)\right) = \det\left(\Lambda'(z)\right)^r = \det\left(\tilde{\Lambda}(\Norm{z}^2)\right)^r \left(1 - e^{-\Norm{z}^2}\right)^{r(n-1)}
\end{equation*}
and it is enough to prove that $\det\left(\tilde{\Lambda}(t)\right) > 0$ whenever $t >0$. By Lem.~\ref{lem diagonalization Lambda prime}, we have:
\begin{equation*}
\forall t >0, \qquad \det\left(\tilde{\Lambda}(t)\right) = u_1(t)u_2(t) = \frac{1 - (t^2+2) e^{-t} + e^{-2t}}{1-e^{-t}} = \frac{\det\left(\tilde{\Omega}(t)\right)}{1-e^{-t}},
\end{equation*}
by Eq.~\eqref{eq determinant Omega tilde}. We just proved that $\det\left(\tilde{\Omega}(t)\right)$ is positive for every positive $t$. Hence the result.
\end{proof}

\begin{proof}[Proof of Lem.~\ref{lem boundedness sqrt Lambda sqrt Omega}]
First, recall that $\Omega(z) = \Omega'(z) \otimes I_r$ (see Eq.~\eqref{eq expression Omega}) and $\Lambda(z) = \Lambda'(z) \otimes I_r$ (see Eq.~\eqref{eq expression Lambda prime}). Hence, we only need to prove that the map $z \longmapsto \begin{pmatrix}
0 & \Lambda'(z)^\frac{1}{2}
\end{pmatrix}\Omega'(z)^{-\frac{1}{2}}$ is bounded on $\R^n \setminus \{0\}$. Then, let $z \in \R^n \setminus\{0\}$, the matrix of $\Omega'(z)$ in the orthonormal basis $\mathcal{B}_z$ of $\R^2 \otimes \left(\R \oplus \R^n \right)$ (see Sect.~\ref{subsec variance of the 1-jets}) is given by Lem.~\ref{lem variance operator 1-jets}, and the matrix of $\Omega'(z)^{-\frac{1}{2}}$ in $\mathcal{B}_z$ is:
\begin{equation*}
\left(\begin{array}{c|c}
\rule[-8pt]{0pt}{10pt} \tilde{\Omega}(\Norm{z}^2)^{-\frac{1}{2}} & 0 \\
\hline
0 & \rule{0pt}{16pt}\left(\begin{smallmatrix}
1 & e^{-\frac{1}{2}\Norm{z}^2} \\ e^{-\frac{1}{2}\Norm{z}^2} & 1
\end{smallmatrix}\right)^{-\frac{1}{2}} \otimes I_{n-1}
\end{array}\right).
\end{equation*}
Similarly, by Lem.~\ref{lem conditional variance operator 1-jets}, the matrix of $\begin{pmatrix}
0 & \Lambda'(z)^\frac{1}{2}\end{pmatrix}$ in $\mathcal{B}_z$ is:
\begin{equation*}
\left(\begin{array}{c|c}
\begin{matrix}
0 & \rule[-8pt]{0pt}{24pt} \tilde{\Lambda}(\Norm{z}^2)^\frac{1}{2}
\end{matrix}
 & 0 \\
\hline
0 & \rule{0pt}{16pt}\left(\begin{smallmatrix}
1 & e^{-\frac{1}{2}\Norm{z}^2} \\ e^{-\frac{1}{2}\Norm{z}^2} & 1
\end{smallmatrix}\right)^\frac{1}{2} \otimes I_{n-1}
\end{array}\right).
\end{equation*}
Hence, our problem reduces to proving that: $t \longmapsto \begin{pmatrix}
0 & \tilde{\Lambda}(t)^\frac{1}{2}
\end{pmatrix} \tilde{\Omega}(t)^{-\frac{1}{2}}$ is bounded on $(0,+\infty)$.

Recall that, for all $t \in [0,+\infty)$, $P(t) \in O_4(\R)$ was defined by Def.~\ref{def b and P}. By Lem.~\ref{lem diagonalization Omega tilde} and~\ref{lem diagonalization Lambda prime}, for all $t \in (0,+\infty)$ we have:
\begin{multline*}
\begin{pmatrix}
0 & \tilde{\Lambda}(t)^\frac{1}{2}
\end{pmatrix}\tilde{\Omega}(t)^{-\frac{1}{2}}\\
\begin{aligned}
&= \left(\begin{array}{c|c}
0 & Q^\text{t} \left(\begin{smallmatrix} u_1(t)^\frac{1}{2} & 0 \\ 0 & u_2(t)^\frac{1}{2} \end{smallmatrix}\right)Q
\end{array}\right)P(t)^\text{t}
\begin{pmatrix}
v_1(t)^{-\frac{1}{2}} & 0 & 0 & 0\\
0 & v_2(t)^{-\frac{1}{2}} & 0 & 0\\
0 & 0 & v_3(t)^{-\frac{1}{2}} & 0\\
0 & 0 & 0 & v_4(t)^{-\frac{1}{2}}
\end{pmatrix}
P(t)\\
&= \begin{pmatrix}
m_1(t) & m_3(t) & m_5(t) & m_6(t)\\
m_2(t) & m_4(t) & m_6(t) & m_5(t)
\end{pmatrix},
\end{aligned}
\end{multline*}
where:
\begin{align*}
m_1 &= \frac{b_+b_-}{4}\left(- \sqrt{\frac{u_2}{v_1}} + \sqrt{\frac{u_2}{v_2}} - \sqrt{\frac{u_1}{v_3}} + \sqrt{\frac{u_1}{v_4}} \right), & m_2 &= \frac{b_+b_-}{4}\left(-\sqrt{\frac{u_2}{v_1}} + \sqrt{\frac{u_2}{v_2}} + \sqrt{\frac{u_1}{v_3}} - \sqrt{\frac{u_1}{v_4}} \right),\\
m_3 &= \frac{b_+b_-}{4}\left(\sqrt{\frac{u_2}{v_1}} - \sqrt{\frac{u_2}{v_2}} - \sqrt{\frac{u_1}{v_3}} + \sqrt{\frac{u_1}{v_4}} \right), & m_4 &= \frac{b_+b_-}{4}\left(\sqrt{\frac{u_2}{v_1}} - \sqrt{\frac{u_2}{v_2}} + \sqrt{\frac{u_1}{v_3}} - \sqrt{\frac{u_1}{v_4}} \right),
\end{align*}
\begin{align*}
m_5 &= \frac{(b_+)^2}{4}\left(\sqrt{\frac{u_2}{v_1}} + \sqrt{\frac{u_1}{v_3}}\right) + \frac{(b_-)^2}{4}\left(\sqrt{\frac{u_2}{v_2}} + \sqrt{\frac{u_1}{v_4}}\right),\\
m_6 &= \frac{(b_+)^2}{4}\left(\sqrt{\frac{u_2}{v_1}} - \sqrt{\frac{u_1}{v_3}}\right) + \frac{(b_-)^2}{4}\left(\sqrt{\frac{u_2}{v_2}} - \sqrt{\frac{u_1}{v_4}}\right).
\end{align*}
By Lem.~\ref{lem non-degeneracy Omega}, for all $t >0$, the $(v_i(t))_{1 \leq i \leq 4}$ are the eigenvalues of a symmetric positive operator, hence are positive. Similarly for all $t>0$, $u_1(t) >0$ and $u_2(t) >0$ by Lem.~\ref{lem non-degeneracy Lambda}. Thus the $(m_i)_{1 \leq i \leq 6}$ are well-defined continuous maps from $(0,+\infty)$ to $\R$. By Eq.~\eqref{eq asymptotic a b+ b-} and~\eqref{eq asymptotic u v},
\begin{align*}
&\forall i \in \{1,2,3,4,6\}, \quad m_i(t) \xrightarrow[t \to +\infty]{} 0 & &\text{and} & &m_5(t) \xrightarrow[t \to +\infty]{} 1.
\end{align*}
Moreover, by Eq.~\eqref{DL a and b+2}--\eqref{DL v3 and v4}, for all $i \in \{1,2,5,6\}$, $m_i(t) = 1/2 + O(\sqrt{t})$ as to goes to $0$ and, for any $i \in \{3,4\}$, $m_i(t) = -1/2 + O(\sqrt{t})$ as to goes to $0$. Hence for all $i \in \{1,\dots,6\}$, $m_i$ is a bounded function from $(0,+\infty)$ to $\R$, which concludes the proof.
\end{proof}

\begin{proof}[Proof of Lem.~\ref{lem estimates 0 expectation of odet XY}]
Recall that, for all $t>0$, the couples $(X_{ij}(t),Y_{ij}(t))$ are independent centered Gaussian vectors in $\R^2$. We denote by $\Lambda_{ij}(t)$ the variance matrix of $(X_{ij}(t),Y_{ij}(t))$, which equals $\tilde{\Lambda}(t)$ if $j=1$ and 
\begin{equation*}
\begin{pmatrix}
1 & e^{-\frac{1}{2}t^2} \\ e^{-\frac{1}{2}t^2} & 1
\end{pmatrix}
\end{equation*}
otherwise (see Def.~\ref{def XYt}, Lem.~\ref{lem relation Bargmann--Fock XY} and Lem.~\ref{lem conditional variance operator 1-jets}).

For all $i \in \{1,\dots,r\}$, $j \in \{1,\dots,n\}$ and $t>0$, we can write:
\begin{equation*}
\begin{pmatrix} X_{ij}(t) \\ Y_{ij}(t) \end{pmatrix} = \sqrt{\Lambda_{ij}(t)} \begin{pmatrix} A_{ij} \\ B_{ij}\end{pmatrix},
\end{equation*}
where the $(A_{ij})$ and $(B_{ij})$ are globally independent real standard Gaussian variables, not depending on $t$. Note that by Lem.~\ref{lem non-degeneracy Lambda}, the $\Lambda_{ij}(t)$ are positive for any $t >0$. We deduce from Lem.~\ref{lem diagonalization Lambda prime} that for any $i \in \{1,\dots,r\}$, for all $t>0$:
\begin{align*}
\sqrt{\Lambda_{i1}(t)} &= \begin{pmatrix}
\alpha(t) & \beta(t) \\ \beta(t) & \alpha(t)
\end{pmatrix} & &\text{and} \quad \forall j \in \{2,\dots,n\}, & \sqrt{\Lambda_{ij}(t)} &= \begin{pmatrix}
\gamma(t) & \delta(t) \\ \delta(t) & \gamma(t)
\end{pmatrix},
\end{align*}
where
\begin{align}
\alpha(t) &= \frac{1}{2}\left(\sqrt{u_2(t)} + \sqrt{u _1(t)}\right), & \gamma(t) &= \frac{1}{2}\left(\sqrt{1+e^{-\frac{1}{2}t^2}} + \sqrt{1-e^{-\frac{1}{2}t^2}}\right),\label{eq def alpha gamma}\\
\beta(t) &= \frac{1}{2}\left(\sqrt{u_2(t)} - \sqrt{u _1(t)}\right), & \delta(t) &= \frac{1}{2}\left(\sqrt{1+e^{-\frac{1}{2}t^2}} - \sqrt{1-e^{-\frac{1}{2}t^2}}\right)\label{eq def beta delta}.
\end{align}

We denote $A_j=(A_{1j},\dots,A_{rj})^\text{t}$ the $j$-th column of $A$ and similarly $B_j = (B_{1j},\dots,B_{rj})^\text{t}$. Then, $\esp{\odet{X(t)}\odet{Y(t)}} = \esp{\Psi(t,A,B)}$, where
\begin{multline}
\label{eq def Psi}
\Psi(t,A,B) = \odet{\alpha(t)A_1+\beta(t)B_1,\gamma(t)A_2+\delta(t)B_2,\dots,\gamma(t)A_n+\delta(t)B_n}\\ \odet{\beta(t)A_1+\alpha(t)B_1,\delta(t)A_2+\gamma(t)B_2,\dots,\delta(t)A_n+\gamma(t)B_n}.
\end{multline}
By~\eqref{DL u1 and u2}, $\alpha(t) = \frac{1}{2}\sqrt{t} + O(t)$ and $\beta(t) = -\frac{1}{2}\sqrt{t} + O(t)$. We extend continuously $\alpha$, $\beta$, $\gamma$ and $\delta$ by $\alpha(0) = 0 = \beta(0)$ and $\gamma(0) = \frac{1}{\sqrt{2}} = \delta(0)$. The function $\Psi$ also extend continuously at $t=0$.

Then $\alpha, \beta, \gamma$ and $\delta$ are bounded functions on $(0,1]$ and $\Psi$ is the square root of a polynomial of degree $4r$ in $(A,B)$ whose coefficients are bounded functions of $t$. In particular, for all $t \in (0,1]$, $\Psi(t,A,B)$ is dominated by a polynomial in $(A,B)$ whose coefficients are independent of $t$. By Lebesgue's Theorem,
\begin{equation}
\label{eq limit expectation t=0}
\esp{\odet{X(t)}\odet{Y(t)}} \xrightarrow[t \to 0]{} \esp{\Psi(0,A,B)}.
\end{equation}

Let $j\in \{2,\dots,n\}$, we define $X_j = (X_{1j},\dots,X_{rj})^\text{t}$ by $X_j=\gamma(0)A_j+\delta(0)B_j = \frac{1}{\sqrt{2}}(A_j+B_j)$. Then the $(X_{ij})$ with $i\in \{1,\dots,r\}$ and $j\in \{2,\dots,n\}$ are independent real standard Gaussian variables. Setting $X_1 = (X_{11},\dots,X_{r1})^\text{t} = 0$, we have:
\begin{align*}
\Psi(0,A,B) &= \odet{X_1,X_2,\dots,X_n}^2 = \det\left((X_1,X_2,\dots,X_n)(X_1,X_2,\dots,X_n)^\text{t}\right)\\
&= \sum_{1 \leq k_1 < \dots < k_r \leq n} \det\left((X_{ik_j})_{1 \leq i,j \leq r}\right)^2,
\end{align*}
by the Cauchy--Binet formula. Let $1 \leq k_1 < k_2 < \dots < k_r \leq n$. If $k_1 = 1$, the first column of $(X_{ik_j})_{1 \leq i,j \leq r}$ is zero and its determinant equals $0$. Otherwise,
\begin{equation}
\label{eq expectation det squared}
\esp{\det\left((X_{ik_j})_{1 \leq i,j \leq r}\right)^2} = \sum_{\sigma, \tau \in \mathfrak{S}_r} \epsilon(\sigma)\epsilon(\tau) \prod_{i=1}^r \esp{X_{ik_{\sigma(i)}}X_{ik_{\tau(i)}}} = r!.
\end{equation}
Hence, if $r <n$, we have
\begin{equation*}
\esp{\Psi(0,A,B)} = r! \binom{n-1}{r} = \frac{(n-1)!}{(n-r-1)!},
\end{equation*}
and by Eq.~\eqref{eq limit expectation t=0}, we proved Lem.~\ref{lem estimates 0 expectation of odet XY} in this case. If $r=n$, we have $\esp{\Psi(0,A,B)} = 0$ and we must be more precise.

Let us now assume that $r=n$. Then, $X$ and $Y$ are square matrices and their Jacobians are simply the absolute values of their determinants. For all $t>0$, we have:
\begin{multline}
\label{eq relation Psi}
\Psi(t,A,B) = \frac{t}{2} \norm{\det\left(\sqrt{\frac{2}{t}}\alpha(t)A_1+\sqrt{\frac{2}{t}}\beta(t)B_1,\gamma(t)A_2+\delta(t)B_2,\dots,\gamma(t)A_n+\delta(t)B_n\right)}\\
\norm{\det\left(\sqrt{\frac{2}{t}}\beta(t)A_1+\sqrt{\frac{2}{t}}\alpha(t)B_1,\delta(t)A_2+\gamma(t)B_2,\dots,\delta(t)A_n+\gamma(t)B_n\right)}.
\end{multline}
By Eq.~\eqref{DL u1 and u2}, we have: $\sqrt{\frac{2}{t}}\alpha(t) = \frac{1}{\sqrt{2}}+O(\sqrt{t})$ and $\sqrt{\frac{2}{t}}\beta(t) = -\frac{1}{\sqrt{2}}+O(\sqrt{t})$. We can apply the same kind of argument as above. By Lebesgue's Theorem:
\begin{equation*}
\frac{2}{t}\esp{\Psi(t,A,B)} \xrightarrow[t \to 0]{} \esp{\norm{\det(Y_1,X_2,\dots,X_n)}\norm{\det(-Y_1,X_2,\dots,X_n)}} = \esp{\det(Y_1,X_2,\dots,X_n)^2},
\end{equation*}
where $Y_1 = (Y_{11},\dots,Y_{r1})^\text{t} = \frac{1}{\sqrt{2}}(A_1-B_1)$. Since $Y_1,X_2,\dots,X_n$ are independent $\mathcal{N}(\Id)$ in $\R^r$, the same computation as Eq.~\eqref{eq expectation det squared} shows that: $\esp{\det(Y_1,X_2,\dots,X_n)^2} = r! = n!$. Hence, if $r=n$, we have: $\esp{\odet{X(t)}\odet{Y(t)}}  = \esp{\Psi(t,A,B)} \sim \frac{n!}{2}t$, as $t \to 0$.
\end{proof}

\begin{proof}[Proof of Lem.~\ref{lem estimates infty expectation of odet XY}]
For any $t >0$, let us denote by:
\begin{equation}
\label{eq def Lambda hat}
\hat{\Lambda}(t) = \left(\begin{array}{c|c}
\tilde{\Lambda}(t) & 0 \\
\hline
\rule{0pt}{14pt}
0 & \begin{pmatrix}
1 & e^{-\frac{1}{2}t^2} \\ e^{-\frac{1}{2}t^2} & 1
\end{pmatrix} \otimes I_{n-1}
\end{array}\right) \otimes I_r,
\end{equation}
the variance matrix of $(X(t),Y(t))$. 

In the following, we denote by $L=(X,Y)$ a generic element of $\mathcal{M}_{rn}(\R) \times \mathcal{M}_{rn}(\R)$. We have:
\begin{multline}
\label{eq limit expectation t=infty}
\esp{\odet{X(t)}\odet{Y(t)}} =\\
\frac{1}{(2\pi)^{rn}} \det\left(\hat{\Lambda}(t)\right)^{-\frac{1}{2}} \int \odet{X}\odet{Y} \exp\left(-\frac{1}{2}\prsc{\hat{\Lambda}(t)^{-1}L}{L}\right) \dx L.
\end{multline}
By Lem.~\ref{lem diagonalization Lambda prime}, we have $\hat{\Lambda}(t) = \Id + O\!\left(te^{-\frac{t}{2}}\right)$ as $t\to +\infty$. Then, $\det\left(\hat{\Lambda}(t)\right)^{-\frac{1}{2}} = 1 + O\!\left(te^{-\frac{t}{2}}\right)$. Moreover, by the Mean Value Theorem,
\begin{multline*}
\norm{\exp\left(-\frac{1}{2}\prsc{\hat{\Lambda}(t)^{-1}L}{L}\right)- e^{-\frac{1}{2}\Norm{L}^2}} = e^{-\frac{1}{2}\Norm{L}^2}\norm{\exp\left(-\frac{1}{2}\prsc{\left(\hat{\Lambda}(t)^{-1}-\Id\right)L}{L}\right)- 1}\\
\leq e^{-\frac{1}{2}\Norm{L}^2}\frac{\Norm{L}^2}{2}\Norm{\hat{\Lambda}(t)^{-1}-\Id}\exp\left(\frac{\Norm{L}^2}{2}\Norm{\hat{\Lambda}(t)^{-1}-\Id}\right).
\end{multline*}
Then, since $\hat{\Lambda}(t)^{-1} = \Id + O\!\left(te^{-\frac{t}{2}}\right)$, this last term is smaller than $e^{-\frac{1}{4}\Norm{L}^2}\frac{\Norm{L}^2}{2}\Norm{\hat{\Lambda}(t)^{-1}-\Id}$ for all $t$ large enough. Hence,
\begin{multline*}
\int \odet{X}\odet{Y} \norm{\exp\left(-\frac{1}{2}\prsc{\hat{\Lambda}(t)^{-1}L}{L}\right)- e^{-\frac{1}{2}\Norm{L}^2}} \dx L \\
\leq \frac{1}{2}\Norm{\hat{\Lambda}(t)^{-1}-\Id} \int \odet{X}\odet{Y}\Norm{L}^2 e^{-\frac{1}{4}\Norm{L}^2} \dx L = O\!\left(t e^{-\frac{t}{2}}\right).
\end{multline*}
Thanks to this relation and Eq.~\eqref{eq limit expectation t=infty}, we get that:
\begin{equation*}
\esp{\odet{X(t)}\odet{Y(t)}} = \esp{\odet{X(\infty)}\odet{Y(\infty)}} + O\!\left(t e^{-\frac{t}{2}}\right),
\end{equation*}
where $\left(X(\infty),Y(\infty)\right) \sim \mathcal{N}(\Id)$ in $\mathcal{M}_{rn}(\R) \times \mathcal{M}_{rn}(\R)$. Finally, by \cite[Lem.~A.14]{Let2016},
\begin{equation*}
\esp{\odet{X(\infty)}\odet{Y(\infty)}} = \esp{\odet{X(\infty)}}^2 = (2\pi)^r \left(\frac{\vol{\S^{n-r}}}{\vol{\S^n}}\right)^2. \qedhere
\end{equation*}
\end{proof}


\section{Technical computations for Section~\ref{sec proof of the main theorem}}
\label{sec technical computations for Section main proof}

\begin{proof}[Proof of Lem.~\ref{lem asymptotic Theta d}]
Let $\alpha \in (0,1)$, we want to prove that $\Theta(z)^{-\frac{1}{2}} \Theta_d(z) \Theta(z)^{-\frac{1}{2}} - \Id = O(d^{-\alpha})$ uniformly for $x \in M$ and $z \in B_{T_xM}(0, b_n \ln d)$. Recall that $Q = \frac{1}{\sqrt{2}}\left(\begin{smallmatrix}
1 & -1 \\ 1 & 1
\end{smallmatrix}\right)$. Since $Q \in O_2(\R)$, it is equivalent to prove that:
\begin{equation}
\label{eq to prove asymptotic Theta d}
\left(Q \otimes \Id_{\R \left(\E \otimes \L^d\right)_x}\right) \Theta(z)^{-\frac{1}{2}} \left(\Theta_d(z) - \Theta(z) \right) \Theta(z)^{-\frac{1}{2}} \left(Q \otimes \Id_{\R \left(\E \otimes \L^d\right)_x}\right)^{-1} = O(d^{-\alpha}).
\end{equation}

Recall that $e_d$ was defined by Eq.~\eqref{eq def ed} and that $e_\infty = \xi \Id_{\R \left(\E \otimes \L^d\right)_x}$ (see Sect.~\ref{sec properties of the limit distribution}). For any $d \in \N$, for all $x \in M$ and for all $w,z \in B_{T_xM}(0,b_n \ln d)$ we set: $\epsilon_d(w,z) = e_d(w,z)-e_\infty(w,z)$. By Eq.~\eqref{eq def Theta} and~\eqref{eq expression Theta dz} we have:
\begin{equation*}
\Theta_d(z) - \Theta(z) = \begin{pmatrix}
\epsilon_d(0,0) & \epsilon_d(0,z) \\ \epsilon_d(z,0) & \epsilon_d(z,z)
\end{pmatrix}.
\end{equation*}
Then, by Lem.~\ref{lem diagonalization Theta}, for all $x \in M$ and $z \in B_{T_xM}(0,b_n \ln d) \setminus \{0\}$ we have:
\begin{equation*}
\left(Q \otimes \Id_{\R \left(\E \otimes \L^d\right)_x}\right) \Theta(z)^{-\frac{1}{2}} \left(\Theta_d(z) - \Theta(z) \right)\Theta(z)^{-\frac{1}{2}} \left(Q \otimes \Id_{\R \left(\E \otimes \L^d\right)_x}\right)^{-1} = \begin{pmatrix}
 a_d(z) & b_d(z)^* \\ b_d(z) & c_d(z)
\end{pmatrix},
\end{equation*}
where
\begin{align*}
a_d(z) &= \frac{1}{2}\left(1 - e^{-\frac{1}{2}\Norm{z}^2}\right)^{-1} (\epsilon_d(z,z)-\epsilon_d(z,0)-\epsilon_d(0,z)+\epsilon_d(0,0)),\\
b_d(z) &= -\frac{1}{2}\left(1 - e^{-\Norm{z}^2}\right)^{-1/2}(\epsilon_d(z,z)-\epsilon_d(z,0)+\epsilon_d(0,z)-\epsilon_d(0,0)),\\
\text{and} \qquad c_d(z) &= \frac{1}{2}\left(1 + e^{-\frac{1}{2}\Norm{z}^2}\right)^{-1} (\epsilon_d(z,z)+\epsilon_d(z,0)+\epsilon_d(0,z)+\epsilon_d(0,0)).
\end{align*}

Let $\beta \in (\alpha, 1)$, by Prop.~\ref{prop near diag estimates} we have $\Norm{D^2_{(w,z)}\epsilon_d} \leq Cd^{-\beta}$, where $C$ is independent of $x \in M$ and $w,z \in B_{T_xM}(0,b_n \ln d)$. Then, a second order Taylor expansion around $(0,0)$ gives:
\begin{equation*}
\Norm{\epsilon_d(z,z)-\epsilon_d(z,0)-\epsilon_d(0,z)+\epsilon_d(0,0)} \leq C\Norm{z}^2 d^{-\beta}.
\end{equation*}
Since we consider $z \in B_{T_xM}(0,b_n \ln d)$ and $1 - e^{-\frac{1}{2}\Norm{z}^2} \sim \frac{\Norm{z}^2}{2}$ as $z \to 0$, we have:
\begin{equation*}
\Norm{a_d(z)} \leq \frac{C\Norm{z}^2d^{-\beta}}{2\left(1 - e^{-\frac{1}{2}\Norm{z}^2}\right)} = O\left((\ln d)^2d^{-\beta}\right) = O(d^{-\alpha}),
\end{equation*}
where the error term does not depend on $(x,z)$. We obtain Eq.~\eqref{eq to prove asymptotic Theta d} by reasonning similarly for $b_d(z)$ and $c_d(z)$.
\end{proof}

\begin{proof}[Proof of Lem.~\ref{lem asymptotic Omega d}]
The idea of the proof is the same as that of Lem.~\ref{lem asymptotic Theta d} above. Let $\alpha \in (0,1)$, we want to prove that:
\begin{equation}
\label{eq goal}
\Omega(z)^{-\frac{1}{2}}\left(\Omega_d(z) - \Omega(z)\right)\Omega(z)^{-\frac{1}{2}} = O\!\left(d^{-\alpha}\right).
\end{equation}
Recall that we defined: $\epsilon_d(w,z) = e_d(w,z) - e_\infty(w,z)$ for any $x \in M$ and $w,z \in B_{T_xM}(0,b_n \ln d)$. We can express $\Omega_d(z) - \Omega(z)$ in terms of $\epsilon_d$ and its derivatives. Then we  write the matrix of the left-hand side of Eq.~\eqref{eq goal} in an orthonormal basis that diagonalizes $\Omega(z)$. The coefficients of this matrix are linear combinations of $\epsilon_d$ and its derivatives. We  will prove that they are $O\!\left(d^{-\alpha}\right)$ using Taylor expansions and the estimates of Sect.~\ref{subsec off-diagonal estimates}.

The details are longer than in the proof of Lem.~\ref{lem asymptotic Theta d} for two reasons. First, the basis in which $\Omega(z)$ is diagonal now depends on $z$. Second, some of the eigenvalues of $\Omega(z)$ are $O(\Norm{z}^6)$ as $z \to 0$, so that we need to consider Taylor expansions of order $6$ for some coefficients. In addition, the matrices involved are less easily described than in the proof of Lem.~\ref{lem asymptotic Theta d}.

Recall that $e_d$ was defined by Eq.~\eqref{eq def ed} and that $e_\infty = \xi \Id_{\R \left(\E \otimes \L^d\right)_x}$ (see Sect.~\ref{sec properties of the limit distribution}). We expressed $\Omega(z)$ in terms of $e_\infty$ in Eq.~\eqref{eq expression Omega} and $\Omega_d(z)$ in terms of $e_d$ in Eq.~\eqref{eq expression Omega dz}. As an operator on:
\begin{equation*}
\R \left(\E \otimes \L^d\right)_x \oplus \R \left(\E \otimes \L^d\right)_x \oplus \left(T_x^*M \otimes \R \left(\E \otimes \L^d\right)_x\right) \oplus \left(T_x^*M \otimes \R \left(\E \otimes \L^d\right)_x\right),
\end{equation*}
we have:
\begin{equation*}
\label{eq expression Omega d - Omega}
\Omega_d(z) - \Omega(z) = \left( \begin{array}{cc|cc}
\epsilon_d(0,0) & \epsilon_d(0,z) & \partial^\sharp_y \epsilon_d(0,0) & \partial^\sharp_y \epsilon_d(0,z) \\ \rule[-5pt]{0pt}{15pt} \epsilon_d(z,0) & \epsilon_d(z,z) & \partial^\sharp_y \epsilon_d(z,0) & \partial^\sharp_y \epsilon_d(z,z) \\
\hline
\rule[-5pt]{0pt}{15pt}
\partial_x \epsilon_d(0,0) & \partial_x \epsilon_d(0,z) & \partial_x \partial^\sharp_y \epsilon_d(0,0) & \partial_x \partial^\sharp_y \epsilon_d(0,z) \\
\partial_x \epsilon_d(z,0) & \partial_x \epsilon_d(z,z) & \partial_x \partial^\sharp_y \epsilon_d(z,0) & \partial_x \partial^\sharp_y \epsilon_d(z,z)
\end{array} \right).
\end{equation*}
Let us choose an orthonormal basis $\left(\deron{}{x_1},\dots,\deron{}{x_n}\right)$ of $T_xM$ such that $z = \Norm{z}\deron{}{x_1}$. We denote by $(dx_1,\dots,dx_n)$ its dual basis. We can then define a basis of $\R^2 \otimes \left(\R \oplus T_x^*M\right)$ similar to $\mathcal{B}_z$ (see Sect.~\ref{subsec variance of the 1-jets}). For any $i \in \{1,\dots,n\}$, we denote by $\partial_{x_i}$ (resp.~$\partial_{y_i}$) the partial derivative with respect to the $i$-th component of first (resp.~second) variable for maps from $T_xM \times T_xM$ to $\End\left(\R\left( \E\otimes \L^d\right)_x\right)$. Then we can split $\Omega_d(z) - \Omega(z)$ according to the previous basis in the following way:
\begin{equation}
\label{eq splitting Omega d - Omega}
\Omega_d(z) - \Omega(z) = \begin{pmatrix}
A_d(z) & B_d^{(1)}(z)^* & \cdots & B_d^{(n)}(z)^* \\
B_d^{(1)}(z) & C_d^{(11)}(z) & \cdots & C_d^{(1n)}(z) \\
\vdots & \vdots & \ddots & \vdots \\
B_d^{(n)}(z) & C_d^{(n1)}(z) & \cdots & C_d^{(nn)}(z)
\end{pmatrix},
\end{equation}
where,
\begin{align}
& & &A_d(z) = \begin{pmatrix}
\epsilon_d(0,0) & \epsilon_d(0,z)\\
\epsilon_d(z,0) & \epsilon_d(z,z)
\end{pmatrix},
\label{eq def Ad}\\
&\forall i \in \{1,\dots,n\}, &
&B_d^{(i)}(z) = \begin{pmatrix}
\partial_{x_i} \epsilon_d(0,0) & \partial_{x_i} \epsilon_d(0,z)\\
\partial_{x_i} \epsilon_d(z,0) & \partial_{x_i} \epsilon_d(z,z)
\end{pmatrix},
\label{eq def Bdi}\\
&\forall i,j \in \{1,\dots,n\}, &
&C_d^{(ij)}(z) = \begin{pmatrix}
\partial_{x_i} \partial^\sharp_{y_j} \epsilon_d(0,0) & \partial_{x_i} \partial^\sharp_{y_j} \epsilon_d(0,z)\\[2pt]
\partial_{x_i} \partial^\sharp_{y_j} \epsilon_d(z,0) & \partial_{x_i} \partial^\sharp_{y_j} \epsilon_d(z,z)
\end{pmatrix}
\label{eq def Cdij}.
\end{align}

Let us denote by $\mathcal{P}(z)$ the operator whose matrix in our basis is:
\begin{equation*}
\label{eq def P cal}
\left(\begin{array}{c|c}
P(\Norm{z}^2) & 0 \\
\hline
0 & Q \otimes I_{n-1}
\end{array}\right) \otimes I_r,
\end{equation*}
where $P$ was defined by Def.~\ref{def b and P} and $Q = \frac{1}{\sqrt{2}}\left(\begin{smallmatrix}
1 & -1 \\ 1 & 1
\end{smallmatrix}\right)$. Since $\mathcal{P}(z)$ is orthogonal, \eqref{eq goal} is equivalent to the following:
\begin{equation}
\label{eq goal 2}
\mathcal{P}(z) \Omega(z)^{-\frac{1}{2}}\left(\Omega_d(z) - \Omega(z)\right)\Omega(z)^{-\frac{1}{2}}\mathcal{P}(z)^{-1} = O\!\left(d^{-\alpha}\right).
\end{equation}
By Lem.~\ref{cor diagonalization Omega'}, the matrix of $\mathcal{P}(z) \Omega(z)^{-\frac{1}{2}}\mathcal{P}(z)^{-1}$ is $\left(\begin{smallmatrix}
V(z) & 0 \\ 0 & N(z) \otimes I_{n-1}
\end{smallmatrix} \right)$, where
\begin{align*}
V(z) &= \begin{pmatrix}
v_1(\Norm{z}^2)^{-\frac{1}{2}} & 0 & 0 & 0\\
0 & \hspace{-1mm}v_2(\Norm{z}^2)^{-\frac{1}{2}} & 0 & 0\\
0 & 0 & \hspace{-1mm}v_3(\Norm{z}^2)^{-\frac{1}{2}} & 0\\
0 & 0 & 0 & \rule[-3pt]{0pt}{1pt} \hspace{-1mm}v_4(\Norm{z}^2)^{-\frac{1}{2}}
\end{pmatrix} \otimes I_r \\
\text{and} \quad N(z) &= \begin{pmatrix}
\left(1 - e^{-\frac{1}{2}\Norm{z}^2}\right)^{-\frac{1}{2}} & 0 \\ 0 & \hspace{-1mm}\left(1 + e^{-\frac{1}{2}\Norm{z}^2}\right)^{-\frac{1}{2}}
\end{pmatrix} \otimes I_r.
\end{align*}
On the other hand, by Eq.~\eqref{eq splitting Omega d - Omega},
\begin{equation*}
\mathcal{P}(z)\left(\Omega_d(z) - \Omega(z)\right)\mathcal{P}(z)^{-1} = \begin{pmatrix}
\tilde{A}_d(z) & \tilde{B}_d^{(1)}(z)^* & \cdots & \tilde{B}_d^{(n)}(z)^* \\
\tilde{B}_d^{(1)}(z) & \tilde{C}_d^{(11)}(z) & \cdots & \tilde{C}_d^{(1n)}(z) \\
\vdots & \vdots & \ddots & \vdots \\
\tilde{B}_d^{(n)}(z) & \tilde{C}_d^{(n1)}(z) & \cdots & \tilde{C}_d^{(nn)}(z)
\end{pmatrix},
\end{equation*}
where
\begin{align}
&\begin{pmatrix}
\tilde{A}_d(z) & \tilde{B}_d^{(1)}(z)^* \\ \tilde{B}_d^{(1)}(z) & \tilde{C}_d^{(11)}(z)
\end{pmatrix} = \left(P(\Norm{z}^2)\otimes \Id\right) \begin{pmatrix}
A_d(z) & B_d^{(1)}(z)^* \\ B_d^{(1)}(z) & C_d^{(11)}(z)
\end{pmatrix} \left(P(\Norm{z}^2)^\text{t} \otimes \Id\right),
\label{eq def Ad tilde}\\
&\forall i \in \{2,\dots,n\}, \quad \begin{pmatrix}
\tilde{B}_d^{(i)}(z) & \tilde{C}_d^{(i1)}(z)
\end{pmatrix} = (Q \otimes \Id) \begin{pmatrix}
B_d^{(i)}(z) & C_d^{(i1)}(z)
\end{pmatrix} \left(P(\Norm{z}^2)^\text{t} \otimes \Id\right),
\label{eq def Bdi tilde}\\
&\forall i,j \in \{2,\dots,n\}, \quad \tilde{C}_d^{(ij)}(z) = (Q \otimes \Id) C_d^{(ij)}(z) (Q^\text{t} \otimes \Id)\label{eq def Cdij tilde}.
\end{align}
Then, in order to prove Eq.~\eqref{eq goal 2}, we have to prove that:
\begin{align}
& & &V(z)\begin{pmatrix}
\tilde{A}_d(z) & \tilde{B}_d^{(1)}(z)^* \\ \tilde{B}_d^{(1)}(z) & \tilde{C}_d^{(11)}(z)
\end{pmatrix}V(z) = O(d^{-\alpha}),
\label{eq goal Ad}\\
&\forall i \in \{2,\dots,n\}, & &N(z)\begin{pmatrix}
\tilde{B}_d^{(i)}(z) & \tilde{C}_d^{(i1)}(z)
\end{pmatrix}V(z) = O(d^{-\alpha}),
\label{eq goal Bdi}\\
&\forall i,j \in \{2,\dots,n\}, & &N(z)\tilde{C}_d^{(ij)}(z)N(z) = O(d^{-\alpha}) \label{eq goal Cdij}.
\end{align}

Since these are heavy computations, we do not reproduce them in totality here. In the following, we give some details about the proof of~\eqref{eq goal Ad}, which is the most difficult of these three relations to establish. The proofs of~\eqref{eq goal Bdi} and~\eqref{eq goal Cdij} are similar and left to the fearless reader.

Let us focus on the proof of~\eqref{eq goal Ad}. We denote
\begin{equation*}
V(z)\begin{pmatrix}
\tilde{A}_d(z) & \tilde{B}_d^{(1)}(z)^* \\ \tilde{B}_d^{(1)}(z) & \tilde{C}_d^{(11)}(z)
\end{pmatrix}V(z) = \begin{pmatrix}
a_d^{(1)} & a_d^{(2)*} & b_d^{(1)*} & b_d^{(2)*} \\[3pt]
a_d^{(2)} & a_d^{(3)} & b_d^{(3)*} & b_d^{(4)*} \\[3pt]
b_d^{(1)} & b_d^{(3)} & c_d^{(1)} & c_d^{(2)*} \\[3pt]
b_d^{(2)} & b_d^{(4)} & c_d^{(2)} & c_d^{(3)}
\end{pmatrix}.
\end{equation*}
Then by Def.~\ref{def b and P} and Eq.~\eqref{eq def Ad}, \eqref{eq def Bdi}, \eqref{eq def Cdij} and \eqref{eq def Ad tilde}, we have:
\begin{multline}
\label{eq expression ad1}
a_d^{(1)}(z) = \frac{1}{4}\left(v_1(\Norm{z}^2)\right)^{-1} \times \\
\left(\begin{aligned}
&(b_+(\Norm{z}^2))^2 \left(\partial_{x_1}\partial_{y_1}^\sharp \epsilon_d(z,z) + \partial_{x_1}\partial_{y_1}^\sharp \epsilon_d(z,0) + \partial_{x_1}\partial_{y_1}^\sharp \epsilon_d(0,z) + \partial_{x_1}\partial_{y_1}^\sharp \epsilon_d(0,0)\right)\\
&+b_+(\Norm{z}^2)b_-(\Norm{z}^2) \left(\partial_{x_1} \epsilon_d(z,z) - \partial_{x_1} \epsilon_d(z,0) + \partial_{x_1} \epsilon_d(0,z) - \partial_{x_1} \epsilon_d(0,0)\right)\\
&+b_+(\Norm{z}^2)b_-(\Norm{z}^2) \left(\partial_{y_1}^\sharp \epsilon_d(z,z) + \partial_{y_1}^\sharp \epsilon_d(z,0) - \partial_{y_1}^\sharp \epsilon_d(0,z) - \partial_{y_1}^\sharp \epsilon_d(0,0)\right)\\
&+(b_-(\Norm{z}^2))^2 \left(\epsilon_d(z,z) - \epsilon_d(z,0) - \epsilon_d(0,z) + \epsilon_d(0,0)\right)
\end{aligned}\right),
\end{multline}
\begin{multline}
\label{eq expression ad2}
a_d^{(2)}(z) = \frac{1}{4}\left(v_1(\Norm{z}^2)v_2(\Norm{z}^2)\right)^{-\frac{1}{2}} \times \\
\left(\begin{aligned}
&-b_+(\Norm{z}^2)b_-(\Norm{z}^2) \left(\partial_{x_1}\partial_{y_1}^\sharp \epsilon_d(z,z) + \partial_{x_1}\partial_{y_1}^\sharp \epsilon_d(z,0) + \partial_{x_1}\partial_{y_1}^\sharp \epsilon_d(0,z) + \partial_{x_1}\partial_{y_1}^\sharp \epsilon_d(0,0)\right)\\
&-(b_-(\Norm{z}^2))^2 \left(\partial_{x_1} \epsilon_d(z,z) - \partial_{x_1} \epsilon_d(z,0) + \partial_{x_1} \epsilon_d(0,z) - \partial_{x_1} \epsilon_d(0,0)\right)\\
&+(b_+(\Norm{z}^2))^2 \left(\partial_{y_1}^\sharp \epsilon_d(z,z) + \partial_{y_1}^\sharp \epsilon_d(z,0) - \partial_{y_1}^\sharp \epsilon_d(0,z) - \partial_{y_1}^\sharp \epsilon_d(0,0)\right)\\
&+b_+(\Norm{z}^2)b_-(\Norm{z}^2) \left(\epsilon_d(z,z) - \epsilon_d(z,0) - \epsilon_d(0,z) + \epsilon_d(0,0)\right)
\end{aligned}\right),
\end{multline}
\begin{multline}
\label{eq expression ad3}
a_d^{(3)}(z) = \frac{1}{4}\left(v_2(\Norm{z}^2)\right)^{-1} \times \\
\left(\begin{aligned}
&(b_-(\Norm{z}^2))^2 \left(\partial_{x_1}\partial_{y_1}^\sharp \epsilon_d(z,z) + \partial_{x_1}\partial_{y_1}^\sharp \epsilon_d(z,0) + \partial_{x_1}\partial_{y_1}^\sharp \epsilon_d(0,z) + \partial_{x_1}\partial_{y_1}^\sharp \epsilon_d(0,0)\right)\\
&-b_+(\Norm{z}^2)b_-(\Norm{z}^2) \left(\partial_{x_1} \epsilon_d(z,z) - \partial_{x_1} \epsilon_d(z,0) + \partial_{x_1} \epsilon_d(0,z) - \partial_{x_1} \epsilon_d(0,0)\right)\\
&-b_+(\Norm{z}^2)b_-(\Norm{z}^2) \left(\partial_{y_1}^\sharp \epsilon_d(z,z) + \partial_{y_1}^\sharp \epsilon_d(z,0) - \partial_{y_1}^\sharp \epsilon_d(0,z) - \partial_{y_1}^\sharp \epsilon_d(0,0)\right)\\
&+(b_+(\Norm{z}^2))^2 \left(\epsilon_d(z,z) - \epsilon_d(z,0) - \epsilon_d(0,z) + \epsilon_d(0,0)\right)
\end{aligned}\right),
\end{multline}
\begin{multline}
\label{eq expression bd1}
b_d^{(1)}(z) = \frac{1}{4}\left(v_1(\Norm{z}^2)v_3(\Norm{z}^2)\right)^{-\frac{1}{2}} \times \\
\left(\begin{aligned}
&-(b_+(\Norm{z}^2))^2 \left(\partial_{x_1}\partial_{y_1}^\sharp \epsilon_d(z,z) + \partial_{x_1}\partial_{y_1}^\sharp \epsilon_d(z,0) - \partial_{x_1}\partial_{y_1}^\sharp \epsilon_d(0,z) - \partial_{x_1}\partial_{y_1}^\sharp \epsilon_d(0,0)\right)\\
&-b_+(\Norm{z}^2)b_-(\Norm{z}^2) \left(\partial_{x_1} \epsilon_d(z,z) - \partial_{x_1} \epsilon_d(z,0) - \partial_{x_1} \epsilon_d(0,z) + \partial_{x_1} \epsilon_d(0,0)\right)\\
&-b_+(\Norm{z}^2)b_-(\Norm{z}^2) \left(\partial_{y_1}^\sharp \epsilon_d(z,z) + \partial_{y_1}^\sharp \epsilon_d(z,0) + \partial_{y_1}^\sharp \epsilon_d(0,z) + \partial_{y_1}^\sharp \epsilon_d(0,0)\right)\\
&-(b_-(\Norm{z}^2))^2 \left(\epsilon_d(z,z) - \epsilon_d(z,0) + \epsilon_d(0,z) - \epsilon_d(0,0)\right)
\end{aligned}\right),
\end{multline}
\begin{multline}
\label{eq expression bd2}
b_d^{(2)}(z) = \frac{1}{4}\left(v_1(\Norm{z}^2)v_4(\Norm{z}^2)\right)^{-\frac{1}{2}} \times \\
\left(\begin{aligned}
&b_+(\Norm{z}^2)b_-(\Norm{z}^2) \left(\partial_{x_1}\partial_{y_1}^\sharp \epsilon_d(z,z) + \partial_{x_1}\partial_{y_1}^\sharp \epsilon_d(z,0) - \partial_{x_1}\partial_{y_1}^\sharp \epsilon_d(0,z) - \partial_{x_1}\partial_{y_1}^\sharp \epsilon_d(0,0)\right)\\
&+(b_-(\Norm{z}^2))^2 \left(\partial_{x_1} \epsilon_d(z,z) - \partial_{x_1} \epsilon_d(z,0) - \partial_{x_1} \epsilon_d(0,z) + \partial_{x_1} \epsilon_d(0,0)\right)\\
&-(b_+(\Norm{z}^2))^2 \left(\partial_{y_1}^\sharp \epsilon_d(z,z) + \partial_{y_1}^\sharp \epsilon_d(z,0) + \partial_{y_1}^\sharp \epsilon_d(0,z) + \partial_{y_1}^\sharp \epsilon_d(0,0)\right)\\
&-b_+(\Norm{z}^2)b_-(\Norm{z}^2) \left(\epsilon_d(z,z) - \epsilon_d(z,0) + \epsilon_d(0,z) - \epsilon_d(0,0)\right)
\end{aligned}\right),
\end{multline}
\begin{multline}
\label{eq expression bd3}
b_d^{(3)}(z) = \frac{1}{4}\left(v_2(\Norm{z}^2)v_3(\Norm{z}^2)\right)^{-\frac{1}{2}} \times \\
\left(\begin{aligned}
&b_+(\Norm{z}^2)b_-(\Norm{z}^2) \left(\partial_{x_1}\partial_{y_1}^\sharp \epsilon_d(z,z) + \partial_{x_1}\partial_{y_1}^\sharp \epsilon_d(z,0) - \partial_{x_1}\partial_{y_1}^\sharp \epsilon_d(0,z) - \partial_{x_1}\partial_{y_1}^\sharp \epsilon_d(0,0)\right)\\
&-(b_+(\Norm{z}^2))^2 \left(\partial_{x_1} \epsilon_d(z,z) - \partial_{x_1} \epsilon_d(z,0) - \partial_{x_1} \epsilon_d(0,z) + \partial_{x_1} \epsilon_d(0,0)\right)\\
&+(b_-(\Norm{z}^2))^2 \left(\partial_{y_1}^\sharp \epsilon_d(z,z) + \partial_{y_1}^\sharp \epsilon_d(z,0) + \partial_{y_1}^\sharp \epsilon_d(0,z) + \partial_{y_1}^\sharp \epsilon_d(0,0)\right)\\
&-b_+(\Norm{z}^2)b_-(\Norm{z}^2) \left(\epsilon_d(z,z) - \epsilon_d(z,0) + \epsilon_d(0,z) - \epsilon_d(0,0)\right)
\end{aligned}\right),
\end{multline}
\begin{multline}
\label{eq expression bd4}
b_d^{(4)}(z) = \frac{1}{4}\left(v_2(\Norm{z}^2)v_4(\Norm{z}^2)\right)^{-\frac{1}{2}} \times \\
\left(\begin{aligned}
&-(b_-(\Norm{z}^2))^2 \left(\partial_{x_1}\partial_{y_1}^\sharp \epsilon_d(z,z) + \partial_{x_1}\partial_{y_1}^\sharp \epsilon_d(z,0) - \partial_{x_1}\partial_{y_1}^\sharp \epsilon_d(0,z) - \partial_{x_1}\partial_{y_1}^\sharp \epsilon_d(0,0)\right)\\
&+b_+(\Norm{z}^2)b_-(\Norm{z}^2) \left(\partial_{x_1} \epsilon_d(z,z) - \partial_{x_1} \epsilon_d(z,0) - \partial_{x_1} \epsilon_d(0,z) + \partial_{x_1} \epsilon_d(0,0)\right)\\
&+b_+(\Norm{z}^2)b_-(\Norm{z}^2) \left(\partial_{y_1}^\sharp \epsilon_d(z,z) + \partial_{y_1}^\sharp \epsilon_d(z,0) + \partial_{y_1}^\sharp \epsilon_d(0,z) + \partial_{y_1}^\sharp \epsilon_d(0,0)\right)\\
&-(b_+(\Norm{z}^2))^2 \left(\epsilon_d(z,z) - \epsilon_d(z,0) + \epsilon_d(0,z) - \epsilon_d(0,0)\right)
\end{aligned}\right),
\end{multline}
\begin{multline}
\label{eq expression cd1}
c_d^{(1)}(z) = \frac{1}{4}\left(v_3(\Norm{z}^2)\right)^{-1} \times \\
\left(\begin{aligned}
&(b_+(\Norm{z}^2))^2 \left(\partial_{x_1}\partial_{y_1}^\sharp \epsilon_d(z,z) - \partial_{x_1}\partial_{y_1}^\sharp \epsilon_d(z,0) - \partial_{x_1}\partial_{y_1}^\sharp \epsilon_d(0,z) + \partial_{x_1}\partial_{y_1}^\sharp \epsilon_d(0,0)\right)\\
&+b_+(\Norm{z}^2)b_-(\Norm{z}^2) \left(\partial_{x_1} \epsilon_d(z,z) + \partial_{x_1} \epsilon_d(z,0) - \partial_{x_1} \epsilon_d(0,z) - \partial_{x_1} \epsilon_d(0,0)\right)\\
&+b_+(\Norm{z}^2)b_-(\Norm{z}^2) \left(\partial_{y_1}^\sharp \epsilon_d(z,z) - \partial_{y_1}^\sharp \epsilon_d(z,0) + \partial_{y_1}^\sharp \epsilon_d(0,z) - \partial_{y_1}^\sharp \epsilon_d(0,0)\right)\\
&+(b_-(\Norm{z}^2))^2 \left(\epsilon_d(z,z) + \epsilon_d(z,0) + \epsilon_d(0,z) + \epsilon_d(0,0)\right)
\end{aligned}\right),
\end{multline}
\begin{multline}
\label{eq expression cd2}
c_d^{(2)}(z) = \frac{1}{4}\left(v_3(\Norm{z}^2)v_4(\Norm{z}^2)\right)^{-\frac{1}{2}} \times \\
\left(\begin{aligned}
&-b_+(\Norm{z}^2)b_-(\Norm{z}^2) \left(\partial_{x_1}\partial_{y_1}^\sharp \epsilon_d(z,z) - \partial_{x_1}\partial_{y_1}^\sharp \epsilon_d(z,0) - \partial_{x_1}\partial_{y_1}^\sharp \epsilon_d(0,z) + \partial_{x_1}\partial_{y_1}^\sharp \epsilon_d(0,0)\right)\\
&-(b_-(\Norm{z}^2))^2 \left(\partial_{x_1} \epsilon_d(z,z) + \partial_{x_1} \epsilon_d(z,0) - \partial_{x_1} \epsilon_d(0,z) - \partial_{x_1} \epsilon_d(0,0)\right)\\
&+(b_+(\Norm{z}^2))^2 \left(\partial_{y_1}^\sharp \epsilon_d(z,z) - \partial_{y_1}^\sharp \epsilon_d(z,0) + \partial_{y_1}^\sharp \epsilon_d(0,z) - \partial_{y_1}^\sharp \epsilon_d(0,0)\right)\\
&+b_+(\Norm{z}^2)b_-(\Norm{z}^2) \left(\epsilon_d(z,z) + \epsilon_d(z,0) + \epsilon_d(0,z) + \epsilon_d(0,0)\right)
\end{aligned}\right),
\end{multline}
\begin{multline}
\label{eq expression cd3}
c_d^{(3)}(z) = \frac{1}{4}\left(v_4(\Norm{z}^2)\right)^{-1} \times \\
\left(\begin{aligned}
&(b_-(\Norm{z}^2))^2 \left(\partial_{x_1}\partial_{y_1}^\sharp \epsilon_d(z,z) - \partial_{x_1}\partial_{y_1}^\sharp \epsilon_d(z,0) - \partial_{x_1}\partial_{y_1}^\sharp \epsilon_d(0,z) + \partial_{x_1}\partial_{y_1}^\sharp \epsilon_d(0,0)\right)\\
&-b_+(\Norm{z}^2)b_-(\Norm{z}^2) \left(\partial_{x_1} \epsilon_d(z,z) + \partial_{x_1} \epsilon_d(z,0) - \partial_{x_1} \epsilon_d(0,z) - \partial_{x_1} \epsilon_d(0,0)\right)\\
&-b_+(\Norm{z}^2)b_-(\Norm{z}^2) \left(\partial_{y_1}^\sharp \epsilon_d(z,z) - \partial_{y_1}^\sharp \epsilon_d(z,0) + \partial_{y_1}^\sharp \epsilon_d(0,z) - \partial_{y_1}^\sharp \epsilon_d(0,0)\right)\\
&+(b_+(\Norm{z}^2))^2 \left(\epsilon_d(z,z) + \epsilon_d(z,0) + \epsilon_d(0,z) + \epsilon_d(0,0)\right)
\end{aligned}\right).
\end{multline}
We need to prove that each one of the terms \eqref{eq expression ad1} to \eqref{eq expression cd3} is a $O(d^{-\alpha})$, where the constant involved in this notation is independent of $(x,z)$. The main difficulty comes the fact that $v_2$ and $v_3$ converge to $0$ as $z \rightarrow 0$ (see Eq.~\eqref{DL v1 and v2} and~\eqref{DL v3 and v4}).

The term with the worst apparent singularity at $z=0$ is~$a_d^{(3)}$ (see~\eqref{eq expression ad3}). We will show below that $a_d^{(3)}(z) = O(d^{-\alpha})$ uniformly in $(x,z)$. The proofs that the other nine coefficients are $O(d^{-\alpha})$ follow the same lines, and they are strictly easier technically. We leave them to the reader.

By Eq.~\eqref{DL v1 and v2}, $v_2(\Norm{z}^2) \sim \frac{\Norm{z}^6}{48}$ as $z \to 0$. Hence, we have to expand the second factor in~\eqref{eq expression ad3} up to a $O(\Norm{z}^6)$. Let $\beta \in (\alpha,1)$, recall that, by Prop.~\ref{prop near diag estimates}, the partial derivatives of $\epsilon_d$ of order up to $6$ are $O(d^{-\beta})$ uniformly on $B_{T_xM}(0,b_n\ln d) \times B_{T_xM}(0,b_n\ln d)$. Recall also that we chose our coordinates so that $z=(\Norm{z},0,\dots,0)$. Using Taylor expansions around $(0,0)$ for $\epsilon_d$ and its derivatives, we get:
\begin{multline}
\label{eq Taylor xy}
\partial_{x_1}\partial_{y_1}^\sharp \epsilon_d(z,z) + \partial_{x_1}\partial_{y_1}^\sharp \epsilon_d(z,0) + \partial_{x_1}\partial_{y_1}^\sharp \epsilon_d(0,z) + \partial_{x_1}\partial_{y_1}^\sharp \epsilon_d(0,0) = \\
4 \partial_{x_1}\partial_{y_1}^\sharp \epsilon_d(0,0) + 2 \Norm{z} \left(\partial_{x_1}^2\partial_{y_1}^\sharp \epsilon_d(0,0) + \partial_{x_1}\left(\partial_{y_1}^\sharp\right)^2 \epsilon_d(0,0) \right)\\
+ \Norm{z}^2 \left(\partial_{x_1}^3\partial_{y_1}^\sharp \epsilon_d(0,0) +\partial_{x_1}^2\left(\partial_{y_1}^\sharp\right)^2 \epsilon_d(0,0) + \partial_{x_1}\left(\partial_{y_1}^\sharp\right)^3 \epsilon_d(0,0)\right)\\
+ \Norm{z}^3\left(\frac{1}{3}\partial_{x_1}^4\partial_{y_1}^\sharp \epsilon_d(0,0) + \frac{1}{2}\partial_{x_1}^3\left(\partial_{y_1}^\sharp\right)^2 \epsilon_d(0,0) + \frac{1}{2}\partial_{x_1}^2\left(\partial_{y_1}^\sharp\right)^3 \epsilon_d(0,0) + \frac{1}{3}\partial_{x_1}\left(\partial_{y_1}^\sharp\right)^4 \epsilon_d(0,0) \right)\\
+ \Norm{z}^4 O(d^{-\beta}),
\end{multline}
\begin{multline}
\label{eq Taylor x}
\partial_{x_1} \epsilon_d(z,z) - \partial_{x_1} \epsilon_d(z,0) + \partial_{x_1} \epsilon_d(0,z) - \partial_{x_1} \epsilon_d(0,0) = \\
2 \Norm{z} \partial_{x_1}\partial_{y_1}^\sharp \epsilon_d(0,0) + \Norm{z}^2 \left(\partial_{x_1}^2\partial_{y_1}^\sharp \epsilon_d(0,0) + \partial_{x_1}\left(\partial_{y_1}^\sharp\right)^2 \epsilon_d(0,0) \right)\\
+ \Norm{z}^3 \left(\frac{1}{2}\partial_{x_1}^3\partial_{y_1}^\sharp \epsilon_d(0,0) + \frac{1}{2}\partial_{x_1}^2\left(\partial_{y_1}^\sharp\right)^2 \epsilon_d(0,0) + \frac{1}{3}\partial_{x_1}\left(\partial_{y_1}^\sharp\right)^3 \epsilon_d(0,0)\right)\\
+ \Norm{z}^4\left(\frac{1}{6}\partial_{x_1}^4\partial_{y_1}^\sharp \epsilon_d(0,0) + \frac{1}{4}\partial_{x_1}^3\left(\partial_{y_1}^\sharp\right)^2 \epsilon_d(0,0) + \frac{1}{6} \partial_{x_1}^2\left(\partial_{y_1}^\sharp\right)^3 \epsilon_d(0,0) + \frac{1}{12}\partial_{x_1}\left(\partial_{y_1}^\sharp\right)^4 \epsilon_d(0,0) \right)\\
+ \Norm{z}^5 O(d^{-\beta}),
\end{multline}
\begin{multline}
\label{eq Taylor y}
\partial_{y_1}^\sharp \epsilon_d(z,z) + \partial_{y_1}^\sharp \epsilon_d(z,0) - \partial_{y_1}^\sharp \epsilon_d(0,z) - \partial_{y_1}^\sharp \epsilon_d(0,0) = \\
2 \Norm{z} \partial_{x_1}\partial_{y_1}^\sharp \epsilon_d(0,0) + \Norm{z}^2 \left(\partial_{x_1}^2\partial_{y_1}^\sharp \epsilon_d(0,0) + \partial_{x_1}\left(\partial_{y_1}^\sharp\right)^2 \epsilon_d(0,0) \right)\\
+ \Norm{z}^3 \left(\frac{1}{3}\partial_{x_1}^3\partial_{y_1}^\sharp \epsilon_d(0,0) + \frac{1}{2}\partial_{x_1}^2\left(\partial_{y_1}^\sharp\right)^2 \epsilon_d(0,0) + \frac{1}{2}\partial_{x_1}\left(\partial_{y_1}^\sharp\right)^3 \epsilon_d(0,0)\right)\\
+ \Norm{z}^4\left(\frac{1}{12}\partial_{x_1}^4\partial_{y_1}^\sharp \epsilon_d(0,0) + \frac{1}{6}\partial_{x_1}^3\left(\partial_{y_1}^\sharp\right)^2 \epsilon_d(0,0) + \frac{1}{4} \partial_{x_1}^2\left(\partial_{y_1}^\sharp\right)^3 \epsilon_d(0,0) + \frac{1}{6}\partial_{x_1}\left(\partial_{y_1}^\sharp\right)^4 \epsilon_d(0,0) \right)\\
+ \Norm{z}^5 O(d^{-\beta}),
\end{multline}
\begin{multline}
\label{eq Taylor}
\epsilon_d(z,z) - \epsilon_d(z,0) - \epsilon_d(0,z) + \epsilon_d(0,0) = \\
\Norm{z}^2 \partial_{x_1}\partial_{y_1}^\sharp \epsilon_d(0,0) + \Norm{z}^3 \left(\frac{1}{2}\partial_{x_1}^2\partial_{y_1}^\sharp \epsilon_d(0,0) + \frac{1}{2}\partial_{x_1}\left(\partial_{y_1}^\sharp\right)^2 \epsilon_d(0,0) \right)\\
+ \Norm{z}^4 \left(\frac{1}{6}\partial_{x_1}^3\partial_{y_1}^\sharp \epsilon_d(0,0) + \frac{1}{4}\partial_{x_1}^2\left(\partial_{y_1}^\sharp\right)^2 \epsilon_d(0,0) + \frac{1}{6}\partial_{x_1}\left(\partial_{y_1}^\sharp\right)^3 \epsilon_d(0,0)\right)\\
+ \Norm{z}^5\left(\frac{1}{24}\partial_{x_1}^4\partial_{y_1}^\sharp \epsilon_d(0,0) + \frac{1}{12}\partial_{x_1}^3\left(\partial_{y_1}^\sharp\right)^2 \epsilon_d(0,0) + \frac{1}{12} \partial_{x_1}^2\left(\partial_{y_1}^\sharp\right)^3 \epsilon_d(0,0) + \frac{1}{24}\partial_{x_1}\left(\partial_{y_1}^\sharp\right)^4 \epsilon_d(0,0) \right)\\
+ \Norm{z}^6 O(d^{-\beta}).
\end{multline}

Now, we can combine Eq.~\eqref{eq Taylor xy}, \eqref{eq Taylor x},\eqref{eq Taylor y} and~\eqref{eq Taylor} with the expansions around $0$ of $(b_+(\Norm{z}^2))^2$ (cf.~Eq.~\eqref{DL a and b+2}), $(b_-(\Norm{z}^2))^2$ and $b_+(\Norm{z}^2)b_-(\Norm{z}^2)$ (cf.~Eq.~\eqref{DL b+b- and b-2}). Using Prop.~\ref{prop near diag estimates} once again, we obtain:
\begin{equation*}
a_d^{(3)}(z) = \frac{1}{4v_2(\Norm{z}^2)} \Norm{z}^6 O(d^{-\beta}) = O((\ln d)^6 d^{-\beta}) = O(d^{-\alpha}),
\end{equation*}
where we used Eq.~\eqref{eq asymptotic u v}, \eqref{DL v1 and v2} and the fact that $\Norm{z} \leq b_n \ln d$. This concludes the proof for $a_d^{(3)}$. As we already explained, we proceed similarly for the other nine coefficients to get~\eqref{eq goal Ad}, and the same kind of computations yields~\eqref{eq goal Bdi} and~\eqref{eq goal Cdij}.
\end{proof}

\begin{proof}[Proof of Lem.~\ref{lem asymptotic conditional expectation}]
Let $\alpha \in (0,1)$, $x \in M$ and $z \in B_{T_xM}(0,b_n \ln d)\setminus \{0\}$. We will denote by $L =(X,Y)$ a generic element of $\R^2 \otimes T_x^*M \otimes \R \left(\E \otimes \L^d\right)_x$. We also set $\chi(L) = \odet{X}\odet{Y}$. We have:
\begin{multline}
\label{eq expectation chi d z}
\esp{\odet{X_d(z)}\odet{Y_d(z)}} = \frac{1}{(2\pi)^{rn}} \det\left(\Lambda_d(z)\right)^{-\frac{1}{2}} \int \chi(L)\exp\left(-\frac{1}{2}\prsc{\Lambda_d(z)^{-1}L}{L}\right) \dx L\\
=\frac{1}{(2\pi)^{rn}} \left(\frac{\det\Lambda(z)}{\det\Lambda_d(z)}\right)^\frac{1}{2}  \int \chi\left(\Lambda(z)^\frac{1}{2} L\right) \exp\left(-\frac{1}{2}\prsc{\Lambda(z)^\frac{1}{2}\Lambda_d(z)^{-1}\Lambda(z)^\frac{1}{2}L}{L}\right) \dx L,
\end{multline}
by a change of variable. And, by Lem.~\ref{lem asymptotic Lambda d}, we have $\det \Lambda_d(z) = \left(\det \Lambda(z) \right) \left(1 + O\!\left(d^{-\alpha}\right)\right)$.

If we set $\Xi_d(z) = \Lambda(z)^\frac{1}{2}\Lambda_d(z)^{-1}\Lambda(z)^\frac{1}{2} - \Id$, then $\Xi_d(z) = O\!\left(d^{-\alpha}\right)$, and these estimates are uniform in $(x,z)$. As in the proof of Lem.~\ref{lem estimates infty expectation of odet XY}, by the Mean Value Theorem, for all $L$ we have:
\begin{equation*}
\norm{\exp\left(-\frac{1}{2}\prsc{\Xi_d(z)L}{L}\right) - 1} \leq \frac{1}{2}\Norm{L}^2 \Norm{\Xi_d(z)} \exp\left(\frac{1}{2}\Norm{L}^2 \Norm{\Xi_d(z)}\right).
\end{equation*}
Since $\Xi_d(z) = O\!\left(d^{-\alpha}\right)$, for $d$ large enough $\Norm{\Xi_d(z)} \leq \frac{1}{2}$. Hence,
\begin{multline}
\label{eq error Xi d z}
\int \chi\left(\Lambda(z)^\frac{1}{2} L\right) e^{-\frac{1}{2}\Norm{L}^2}\norm{\exp\left(-\frac{1}{2}\prsc{\Xi_D(z)L}{L}\right)- 1} \dx L \\
\leq \frac{\Norm{\Xi_d(z)}}{2} \int \chi\left(\Lambda(z)^\frac{1}{2} L\right) \Norm{L}^2 e^{-\frac{1}{4}\Norm{L}^2} \dx L.
\end{multline}

Recall that, by Lem.~\ref{lem diagonalization Lambda prime}, the eigenvalues of the positive symmetric operator $\Lambda(z)$ are $u_1(\Norm{z}^2)$, $u_2(\Norm{z}^2)$, $1+\exp\left(-\frac{1}{2}\Norm{z}^2\right)$ and $1-\exp\left(-\frac{1}{2}\Norm{z}^2\right)$, with some multiplicities. These are bounded functions of $z$ (see Eq.~\eqref{eq asymptotic u v} and~\eqref{DL u1 and u2}). Hence, $\chi\left(\Lambda(z)^\frac{1}{2} L\right)$ is the square root of a polynomial in $L$ whose coefficients are bounded functions of $z$. Thus, the integral on the right-hand side of Eq.~\eqref{eq error Xi d z} is bounded, independently of $(x,z)$. We get:
\begin{multline*}
\int \chi\left(\Lambda(z)^\frac{1}{2} L\right) \exp\left(-\frac{1}{2}\prsc{\Lambda(z)^\frac{1}{2}\Lambda_d(z)^{-1}\Lambda(z)^\frac{1}{2}L}{L}\right) \dx L\\
\begin{aligned}
&= \int \chi\left(\Lambda(z)^\frac{1}{2} L\right) e^{-\frac{1}{2}\Norm{L}^2} \dx L + O\!\left(d^{-\alpha}\right)\\
&= (2\pi)^{rn} \esp{\odet{X_\infty(z)}\odet{Y_\infty(z)}} + O\!\left(d^{-\alpha}\right).
\end{aligned}
\end{multline*}
Finally, by Eq.~\eqref{eq expectation chi d z}, we find
\begin{equation}
\label{eq asymptotic conditional expectation}
\esp{\odet{X_d(z)}\odet{Y_d(z)}} = \esp{\odet{X_\infty(z)}\odet{Y_\infty(z)}} + O\!\left(d^{-\alpha}\right).
\end{equation}

By Lem.~\ref{lem non-degeneracy Lambda}, for all $z \neq 0$, $\Lambda(z)$ is non-singular. Hence $\esp{\odet{X_\infty(z)}\odet{Y_\infty(z)}}$, is a positive function of $z$. By Lem.~\ref{lem relation Bargmann--Fock XY},
\begin{equation*}
\esp{\odet{X_\infty(z)}\odet{Y_\infty(z)}} = \esp{\odet{X(\Norm{z}^2)}\odet{Y(\Norm{z}^2)}},
\end{equation*}
and by Lem.~\ref{lem estimates 0 expectation of odet XY} and~\ref{lem estimates infty expectation of odet XY}, if $r < n$, this quantity admits positive limits when $\Norm{z}$ goes to $0$ or $\Norm{z}$ goes to $+\infty$. Thus, in this case, $\esp{\odet{X_\infty(z)}\odet{Y_\infty(z)}}$ is bounded from below by positive constant, independent of $(x,z)$. Then, Eq.~\eqref{eq asymptotic conditional expectation} shows that:
\begin{equation*}
\esp{\odet{X_d(z)}\odet{Y_d(z)}} = \esp{\odet{X_\infty(z)}\odet{Y_\infty(z)}} \left(1 + O\!\left(d^{-\alpha}\right)\right)
\end{equation*}
and this concludes the proof for $r<n$.

If $r=n$, the leading term in Eq.~\eqref{eq asymptotic conditional expectation} goes to $0$ as $\Norm{z} \rightarrow 0$, so that we need to be more precise. From now on, we assume that $r=n$. Let us assume for now that, in this case, we have:
\begin{equation}
\label{eq hypothesis}
\int \chi\left(\Lambda(z)^\frac{1}{2} L\right) \Norm{L}^2 e^{-\frac{1}{4}\Norm{L}^2} \dx L = O\!\left(\Norm{z}^2\right)
\end{equation}
as $z \rightarrow 0$, where the constant involved in the $O\!\left(\Norm{z}^2\right)$ is uniform in $(x,z)$. Then, proceeding as we did in the case $r<n$, we get the following equivalent of Eq.~\eqref{eq asymptotic conditional expectation}:
\begin{equation*}
\esp{\odet{X_d(z)}\odet{Y_d(z)}} = \esp{\odet{X_\infty(z)}\odet{Y_\infty(z)}} + O\!\left(\Norm{z}^2 d^{-\alpha}\right).
\end{equation*}
By Lem.~\ref{lem estimates 0 expectation of odet XY},
\begin{equation*}
\esp{\odet{X_\infty(z)}\odet{Y_\infty(z)}} = \esp{\odet{X(\Norm{z}^2)}\odet{Y(\Norm{z}^2)}} \sim \frac{n!}{2} \Norm{z}^2,
\end{equation*}
as $z \rightarrow 0$. Hence,
\begin{equation*}
\esp{\odet{X_d(z)}\odet{Y_d(z)}} = \esp{\odet{X_\infty(z)}\odet{Y_\infty(z)}} \left(1 + O\!\left(d^{-\alpha}\right)\right)
\end{equation*}
uniformly for $x \in M$ and $\Norm{z} \leq 1$. In the domain $\Norm{z} \geq 1$, $\esp{\odet{X_\infty(z)}\odet{Y_\infty(z)}}$ is bounded from below by a positive constant independent of $(x,z)$, and we proceed as in the case $r<n$, using Eq.~\eqref{eq asymptotic conditional expectation}. This yields the result for $r=n$.

To conclude the proof, we still have to prove that \eqref{eq hypothesis} holds when $r=n$. Let us write $L =(A,B)$ and $\Lambda(z)^\frac{1}{2}L =(X(z),Y(z))$ with $A,B,X(z)$ and $Y(z) \in T_x^*M \otimes \R \left(\E \otimes \L^d\right)_x$. We choose any orthonormal basis of $\R \left(\E \otimes \L^d\right)_x$ and an orthonormal basis of $T_xM$ such that the coordinates of $z$ are $(\Norm{z},0,\dots,0)$. We denote by $(A_{ij})$, $(B_{ij})$, $(X_{ij}(z))$ and $(Y_{ij}(z)) \in \mathcal{M}_{rn}(\R)$ the matrices of $A,B,X(z)$ and $Y(z)$ in these bases. 

The matrix of $\Lambda(z)$ in the basis defined by $\mathcal{B}'_z$ (see Sect.~\ref{subsec conditional variance of the derivatives}) and our basis of $\R \left(\E \otimes \L^d\right)_x$ is $\hat{\Lambda}(\Norm{z}^2)$, where $\hat{\Lambda}$ was defined by Eq.~\eqref{eq def Lambda hat}. That is, using the same notations as in the proof of Lem.~\ref{lem estimates 0 expectation of odet XY} (see~Eq.~\eqref{eq def alpha gamma} and~\eqref{eq def beta delta}), for all $i \in \{1,\dots,r\}$:
\begin{align*}
\begin{pmatrix}
X_{i1} \\ Y_{i1}
\end{pmatrix} &= \begin{pmatrix}
\alpha(\Norm{z}^2) & \beta(\Norm{z}^2) \\ \beta(\Norm{z}^2) & \alpha(\Norm{z}^2)
\end{pmatrix}
\begin{pmatrix}
A_{i1} \\ B_{i1}
\end{pmatrix}
& &\text{and} \ \forall j \geq 2, &
\begin{pmatrix}
X_{ij} \\ Y_{ij}
\end{pmatrix} &= \begin{pmatrix}
\gamma(\Norm{z}^2) & \delta(\Norm{z}^2) \\ \delta(\Norm{z}^2) & \gamma(\Norm{z}^2)
\end{pmatrix}
\begin{pmatrix}
A_{ij} \\ B_{ij}
\end{pmatrix}.
\end{align*}
Hence, we have:
\begin{equation*}
\chi\left(\Lambda(z)^\frac{1}{2}L\right) = \chi(X(z),Y(z)) = \odet{X(z)}\odet{Y(z)} = \Psi\left(\Norm{z}^2,(A_{ij}),(B_{ij})\right),
\end{equation*}
where $\Psi$ was defined by Eq.~\eqref{eq def Psi}. Recall that $\Psi$ satisfies~\eqref{eq relation Psi} when $r=n$. As in the proof of Lem.~\ref{lem estimates 0 expectation of odet XY} (cf.~App.~\ref{sec technical computations for Section limit distrib}), by Lebesgue's Theorem we have:
\begin{multline*}
\frac{2}{\Norm{z}^2} \int \chi\left(\Lambda(z)^\frac{1}{2} L\right) \Norm{L}^2 e^{-\frac{1}{4}\Norm{L}^2} \dx L = \int \frac{2}{\Norm{z}^2}\Psi\left(\Norm{z}^2,(A_{ij}),(B_{ij})\right) \Norm{L}^2 e^{-\frac{1}{4}\Norm{L}^2} \dx L\\
\xrightarrow[\Norm{z} \to 0]{} \int \det\left(\frac{A_1-B_1}{\sqrt{2}},\frac{A_2+B_2}{\sqrt{2}},\dots,\frac{A_n+B_n}{\sqrt{2}}\right)^2 \Norm{L}^2 e^{-\frac{1}{4}\Norm{L}^2} \dx L,
\end{multline*}
where $A_j$ (resp.~$B_j$) denotes the $j$-th column of the matrix of $A$ (resp.~$B$) and $L=(A,B)$. This limit is finite, which proves that~\eqref{eq hypothesis} is satisfied and concludes the proof.
\end{proof}


\bibliographystyle{amsplain}
\bibliography{VarianceofthevolumeofrandomrealalgebraicsubmanifoldsII}

\providecommand{\bysame}{\leavevmode\hbox to3em{\hrulefill}\thinspace}
\providecommand{\MR}{\relax\ifhmode\unskip\space\fi MR }
\providecommand{\MRhref}[2]{%
  \href{http://www.ams.org/mathscinet-getitem?mr=#1}{#2}
}
\providecommand{\href}[2]{#2}
\begin{thebibliography}{10}

\bibitem{AT2007}
R.~J. Adler and J.~E. Taylor, \emph{Random fields and geometry}, 1st ed.,
  Monographs in Mathematics, Springer, New York, 2007.

\bibitem{AADL2017}
D.~Armentano, J.-M. Azaïs, F.~Dalmao, and J.~R. Le{\'o}n, \emph{On the
  asymptotic variance of the number of real roots of random polynomial
  systems}, Proc. Amer. Math. Soc. \textbf{146} (2018), no.~12, 5437--5449.

\bibitem{AW2009}
J.-M. Azaïs and M.~Wschebor, \emph{Level sets and extrema of random processes
  and fields}, 1st ed., John Wiley \& Sons, Hoboken, NJ, 2009.

\bibitem{Bay2018a}
T.~Bayraktar, \emph{Expected number of real roots for random linear
  combinations of orthogonal polynomials associated with radial weights},
  Potential Anal. \textbf{48} (2018), no.~4, 459--471.

\bibitem{Ber2009}
R.~J. Berman, \emph{Bergman kernels and equilibrium measures for line bundles
  over projective manifolds}, Amer. J. Math. \textbf{131} (2009), no.~5,
  1485--1524.

\bibitem{BSZ2000a}
P.~Bleher, B.~Shiffman, and S.~Zelditch, \emph{Universality and scaling of
  correlations between zeros on complex manifolds}, Invent. Math. \textbf{142}
  (2000), no.~2, 351--395.

\bibitem{BBL1996}
E.~Bogomolny, O.~Bohigas, and P.~Leboeuf, \emph{Quantum chaotic dynamics and
  random polynomials}, J. Statist. Phys. \textbf{85} (1996), no.~5--6,
  639--679.

\bibitem{CH2016a}
Y.~Canzani and B.~Hanin, \emph{Local integral statistics for monochromatic
  random waves}, arXiv:1610.09438 (2016).

\bibitem{CH2016}
\bysame, \emph{$\mathcal{C}^\infty$ scaling asymptotics for the spectral
  projector of the {L}aplacian}, J. Geom. Anal. \textbf{28} (2018), no.~1,
  111--122.

\bibitem{DLM2006}
X.~Dai, K.~Liu, and X.~Ma, \emph{On the asymptotic expansion of the {B}ergman
  kernel}, J. Differential Geom. \textbf{72} (2006), no.~1, 1--41.

\bibitem{Dal2015}
F.~Dalmao, \emph{Asymptotic variance and {CLT} for the number of zeros of
  {K}ostlan--{S}hub--{S}male random polynomials}, C. R. Math. Acad. Sci. Paris
  \textbf{353} (2015), no.~12, 1141--1145.

\bibitem{DNPR2016}
F.~Dalmao, I.~Nourdin, G.~Peccati, and M.~Rossi, \emph{Phase singularities in
  complex arithmetic random waves}, arXiv:1608.05631 (2016).

\bibitem{GW2011}
D.~Gayet and J.-Y. Welschinger, \emph{Exponential rarefaction of real curves
  with many components}, Publ. Math. Inst. Hautes Études Sci. (2011), no.~113,
  69--96.

\bibitem{GW2015}
\bysame, \emph{Expected topology of random real algebraic submanifolds}, J.
  Inst. Math. Jussieu \textbf{14} (2015), no.~4, 673--702.

\bibitem{GW2016}
\bysame, \emph{{B}etti numbers of random real hypersurfaces and determinants of
  random symmetric matrices}, J. of Eur. Math. Soc. \textbf{18} (2016), no.~4,
  733--772.

\bibitem{GH1994}
Ph. Griffiths and J.~Harris, \emph{Principles of algebraic geometry}, 2nd ed.,
  Wiley Classics Library, John Wiley \& Sons, New York, 1994, Reprint of the
  1978 original.

\bibitem{Kos1993}
E.~Kostlan, \emph{On the distribution of roots of random polynomials}, From
  topology to computation: proceedings of the {S}malefest ({B}erkeley, {CA},
  1990), Springer, New York, 1993, pp.~419--431.

\bibitem{KL1997}
M.~F. Kratz and J.~R. Le{\'o}n, \emph{Hermite polynomial expansion for
  non-smooth functionals of stationary {G}aussian processes: crossings and
  extremes}, Stochastic Process. Appl. \textbf{66} (1997), no.~2, 237--252.

\bibitem{KL2001}
\bysame, \emph{Central limit theorems for level functionals of stationary
  {G}aussian processes and fields}, J. Theoret. Probab. \textbf{14} (2001),
  no.~3, 639--672.

\bibitem{KKW2013}
M.~Krishnapur, P.~Kurlberg, and I.~Wigman, \emph{Nodal length fluctuations for
  arithmetic random waves}, Ann. of Math. (2) \textbf{177} (2013), no.~2,
  699--737.

\bibitem{Let2016}
T.~Letendre, \emph{Expected volume and {E}uler characteristic of random
  submanifolds}, J. Funct. Anal. \textbf{270} (2016), no.~8, 3047--3110.

\bibitem{Let2016a}
\bysame, \emph{Variance of the volume of random real algebraic submanifolds},
  Trans. Amer. Math. Soc. \textbf{371} (2019), no.~6, 4129--–4192.

\bibitem{MM2007}
X.~Ma and G.~Marinescu, \emph{Holomorphic {M}orse inequalities and {B}ergman
  kernels}, 1st ed., Progress in Mathematics, vol. 254, Birkhäuser, Basel,
  2007.

\bibitem{MM2015}
\bysame, \emph{Exponential estimate for the asymptotics of {B}ergman kernels},
  Math. Ann. \textbf{362} (2015), no.~3--4, 1327--1347.

\bibitem{MPRW2016}
D.~Marinucci, G.~Peccati, M.~Rossi, and I.~Wigman, \emph{Non-universality of
  nodal length distribution for arithmetic random waves}, Geom. Funct. Anal.
  \textbf{26} (2016), no.~3, 926--960.

\bibitem{NS2016}
F.~Nazarov and M.~Sodin, \emph{Asymptotic laws for the spatial distribution and
  the number of connected components of zero sets of {G}aussian random
  functions}, Zh. Mat. Fiz. Anal. Geom. \textbf{12} (2016), no.~3, 205--278.

\bibitem{Nic2015}
L.~Nicolaescu, \emph{Critical sets of random smooth functions on compact
  manifolds}, Asian J. Math. \textbf{19} (2015), no.~3, 391--432.

\bibitem{NP2012}
I.~Nourdin and G.~Peccati, \emph{Normal approximation with {M}alliavin
  calculus}, 1st ed., Cambridge Tracts in Mathematics, vol. 192, Cambridge
  University Press, Cambridge, 2012.

\bibitem{SW2016}
P.~Sarnak and I.~Wigman, \emph{Topologies of nodal sets of random band limited
  functions}, Advances in the theory of automorphic forms and their
  {$L$}-functions, Contemp. Math., vol. 664, Amer. Math. Soc., Providence, RI,
  2016, pp.~351--365.

\bibitem{SZ1999}
B.~Shiffman and S.~Zelditch, \emph{Distribution of zeros of random and quantum
  chaotic sections of positive line bundles}, Comm. Math. Phys. \textbf{200}
  (1999), no.~3, 661--683.

\bibitem{SZ2010}
\bysame, \emph{Number variance of random zeros on complex manifolds {II}:
  smooth statistics}, Pure Appl. Math. Q. \textbf{6} (2010), no.~4, 1145--1167.

\bibitem{Sil1989}
R.~Silhol, \emph{Real algebraic surfaces}, 1st ed., Lecture Notes in
  Mathematics, vol. 1392, Springer, Berlin - Heidelberg, 1989.

\bibitem{Slu1991}
E.~Slud, \emph{Multiple {W}iener-{I}t{\=o} integral expansions for
  level-crossing-count functionals}, Probab. Theory Related Fields \textbf{87}
  (1991), no.~3, 349--364.

\bibitem{TV2015}
T.~Tao and V.~Vu, \emph{Local universality of zeroes of random polynomials},
  Int. Math. Res. Not. \textbf{2015} (2015), no.~13, 5053--5139.

\bibitem{Wig2010}
I.~Wigman, \emph{Fluctuations of the nodal length of random spherical
  harmonics}, Comm. Math. Phys. \textbf{298} (2010), no.~3, 787--831.

\end{thebibliography}

\end{document}